\documentclass[a4paper,reqno]{amsart}
\pdfoutput=1
\usepackage[allcolors=blue,colorlinks]{hyperref}
\usepackage{amssymb,amsfonts}
\usepackage{mathrsfs}
\usepackage[numbers]{natbib}
\usepackage[all,cmtip]{xy}
\usepackage{color}

\newcommand{\LF}[2]{\langle #1,#2\rangle}
\newcommand{\DLF}[2]{\langle\!\langle #1,#2\rangle\!\rangle}
\newcommand{\vx}{\vec{x}}
\newcommand{\vc}{\vec{c}}
\newcommand{\vom}{\vec{\omega}}
\newcommand{\ra}{\rightarrow}
\newcommand{\lra}{\longrightarrow}
\newcommand{\ZZ}{\mathbb{Z}}

\newcommand{\QQ}{\mathbb{Q}}
\newcommand{\RR}{\mathbb{R}}
\newcommand{\Rr}{\mathcal{R}}
\newcommand{\bR}{\mathbf{R}}
\newcommand{\NN}{\mathbb{N}}
\newcommand{\AAA}{\mathbb{A}}
\newcommand{\XX}{\mathbb{X}}
\newcommand{\Cc}{\mathcal{C}}
\newcommand{\Qq}{\mathcal{Q}}
\newcommand{\Dd}{\mathcal{D}}
\newcommand{\bDd}{\mathcal{D}^{b}}
\newcommand{\Oo}{\mathcal{O}}
\newcommand{\Aa}{\mathcal{A}}
\newcommand{\Xx}{\mathcal{X}}

\newcommand{\Ll}{\mathcal{L}}
\newcommand{\Hh}{\mathcal{H}}
\newcommand{\QHh}{\vec{\mathcal{H}}}
\newcommand{\Tt}{\mathcal{T}}

\newcommand{\PP}{\mathbb{P}}
\newcommand{\Ff}{\mathcal{F}}
\newcommand{\Gg}{\mathcal{G}}
\newcommand{\Kk}{\mathcal{K}}

\newcommand{\vSs}{\mathscr{S}}
\newcommand{\Uu}{\mathcal{U}}
\newcommand{\Ww}{\mathcal{W}}
\newcommand{\Bb}{\mathcal{B}}
\newcommand{\Mm}{\mathcal{M}}
\newcommand{\bp}{\mathbf{p}}
\newcommand{\ovp}{\bar{p}}
\newcommand{\bt}{\mathbf{t}}
\newcommand{\bq}{\mathbf{q}}
\newcommand{\bL}{\mathbf{L}}

\newcommand{\spitz}[1]{\langle #1\rangle}
\newcommand{\rperp}[1]{#1^{\perp}}
\newcommand{\lperp}[1]{{}^{\perp}#1}
\newcommand{\rperpo}[1]{#1^{\perp_0}}
\newcommand{\lperpo}[1]{{}^{\perp_0}#1}
\newcommand{\rperpe}[1]{#1^{\perp_1}}
\newcommand{\lperpe}[1]{{}^{\perp_1}#1}

\newcommand{\eps}{\varepsilon}

\DeclareMathOperator{\Knull}{K_0}
\DeclareMathOperator{\D}{D}

\DeclareMathOperator{\Ker}{Ker}

\DeclareMathOperator{\vect}{vect}
\DeclareMathOperator{\coh}{coh}
\DeclareMathOperator{\Qcoh}{Qcoh}

\DeclareMathOperator{\Dis}{Dis}
\DeclareMathOperator{\PC}{PC}
\DeclareMathOperator{\Mod}{Mod}
\renewcommand{\mod}{\operatorname{mod}}
\DeclareMathOperator{\End}{End}
\DeclareMathOperator{\Hom}{Hom}
\DeclareMathOperator{\RHom}{\bR\Hom}
\DeclareMathOperator{\Ext}{Ext}
\DeclareMathOperator{\Gen}{Gen}

\DeclareMathOperator{\add}{add}
\DeclareMathOperator{\Add}{Add}

\DeclareMathOperator{\rad}{rad}
\DeclareMathOperator{\can}{can}
\DeclareMathOperator{\cc}{cc}
\DeclareMathOperator{\op}{op}

\DeclareMathOperator{\her}{her}
\DeclareMathOperator{\Pres}{Pres}

\DeclareMathOperator{\fp}{fp}
\DeclareMathOperator{\pd}{pd}
\DeclareMathOperator{\pdim}{pdim}
\DeclareMathOperator{\rk}{rk}
\DeclareMathOperator{\matring}{M}

\DeclareMathOperator{\soc}{soc}
\newcommand{\maxspec}{\operatorname{Max-Spec}}
\newcommand{\tube}{\ensuremath{\mathbf{t}}}

\newcommand{\Der}{\mathcal{D}}
\newcommand{\Derived}[1]{\Der(#1)}
\newcommand{\bDerived}[1]{\Der^b(#1)}

\newtheorem*{theorem-1}{Theorem~1}
\newtheorem*{theorem-2}{Theorem~2}
\newtheorem*{theorem-3}{Theorem~3}
\newtheorem*{theorem-4}{Theorem~4}
\newtheorem*{theorem-5}{Theorem~5}

\newtheorem{proposition}{Proposition}[section]
\newtheorem{theorem}[proposition]{Theorem}
\newtheorem{corollary}[proposition]{Corollary}

\newtheorem{lemma}[proposition]{Lemma}

\theoremstyle{definition}
\newtheorem{definition}[proposition]{Definition}
\newtheorem{notation}[proposition]{Notation}
\newtheorem{remark}[proposition]{Remark}
\newtheorem{numb}[proposition]{\!\!}

\newtheorem{example}[proposition]{Example}
\newtheorem*{definit}{Definition}

\numberwithin{equation}{section}
\begin{document}

\title[Large tilting sheaves over weighted curves]{Large tilting
  sheaves over weighted \\ noncommutative regular projective curves}
\author[L. Angeleri H\"ugel]{Lidia Angeleri H\"ugel}
\address{Universit\`a degli Studi di Verona\\
  Strada Le Grazie 15 - Ca' Vignal 2\\
  I - 37134 Verona\\
  Italy} \email{lidia.angeleri@univr.it} \author[D. Kussin]{Dirk
  Kussin} \address{{Graduate School of Mathematics \\ Nagoya
    University \\ Furo-cho \\ Chikusa-ku \\ Nagoya 464-8602 \\ Japan}}
\email{dirk@math.uni-paderborn.de}
\subjclass[2010]{Primary: 14A22, 18E15, Secondary: 14H45, 
   14H52, 16G70}
 \keywords{weighted noncommutative regular projective curve, tilting sheaf,
   resolving class,  Pr\"ufer sheaf, genus zero, domestic curve, tubular curve, elliptic
   curve, slope of quasicoherent sheaf} 
\begin{abstract}
  Let $\XX$ be a weighted noncommutative regular projective curve over
  a field $k$. The category $\Qcoh\XX$ of quasicoherent sheaves is a
  hereditary, locally noetherian Grothendieck category. We classify
  all tilting sheaves which have a non-coherent torsion subsheaf. In
  case of nonnegative orbifold Euler characteristic we classify all
  large (that is, non-coherent) tilting sheaves and the corresponding
  resolving classes. In particular we show that in the elliptic and in
  the tubular cases every large tilting sheaf has a well-defined
  slope.
\end{abstract}
\maketitle
\tableofcontents
\section{Introduction}
Tilting theory is a well-established technique to relate different
mathematical theories. An overview of its role in various areas of
mathematics can be found in~\cite{angeleri:happel:krause:2007}. One of
the first results along these lines, due to
Beilinson~\cite{beilinson:1978}, establishes a connection between
algebraic geometry and representation theory of finite dimensional
algebras. For instance, the projective line $\XX=\PP_1(k)$ over a
field $k$ turns out to be closely related with the Kronecker algebra
$\Lambda$, the path algebra of the quiver
$\xy\xymatrixcolsep{2pc}\xymatrix{ \bullet \ar@<0.5ex>[r]
  \ar@<-0.5ex>[r] & \bullet } \endxy$
over $k$. The connection is provided by the vector bundle
$T=\Oo\oplus\Oo(1)$, which is a tilting sheaf in $\coh\XX$ with
endomorphism ring $\Lambda$. The derived Hom-functor $\bR\Hom(T,-)$
then defines an equivalence between the derived categories of
$\Qcoh\XX$ and $\Mod\Lambda$. There are many more such examples, where
a noetherian tilting object $T$ in a triangulated category $\Dd$
provides an equivalence between $\Dd$ and the derived category of
$\End(T)$. We refer
to~\cite{geigle:lenzing:1987,hille:vandenbergh:2007,herschend:iyama:minamoto:oppermann:2014},
and
to~\cite{buan:marsh:reineke:reiten:todorov:2006,kussin:lenzing:meltzer:2013}
for the context of Calabi-Yau and cluster categories.

The weighted projective lines introduced
in~\cite{geigle:lenzing:1987}, and their generalizations
in~\cite{lenzing:1998}, called noncommutative curves of genus zero
in~\cite{kussin:2009}, provide the basic framework for the present
article. They are characterized by the existence of a tilting bundle
in the category of coherent sheaves $\coh\XX$. In this case the
corresponding (derived-equivalent) finite-dimensional algebras are the
(concealed-) canonical
algebras~\cite{ringel:1984,ringel:crawley-boevey:1990,lenzing:delapena:1999},
an important class of algebras in representation theory. A
particularly interesting and beautiful case is the so-called tubular
case. Here every indecomposable coherent sheaf is semistable (with
respect to the slope), and the semistable coherent sheaves of slope
$q$ form a family of tubes, for every $q$
(\cite{lenzing:meltzer:1992,kussin:2009}). This classification is akin
to Atiyah's classification of indecomposable vector bundles over an
elliptic curve~\cite{atiyah:1957}.

\medskip

The tilting objects mentioned so far are small in the sense that they
are noetherian objects, and that their endomorphism rings are
finite-dimensional algebras. For arbitrary rings $R$ there is the
notion of a (not necessarily noetherian or finitely generated) tilting
module $T$, which was extended to Grothendieck categories
in~\cite{colpi:1999, colpi:fuller:2007}.
\begin{definit} An object $T$ in a Grothendieck category $\QHh$ is
  called \emph{tilting} if $T$ generates precisely the objects in
  $\rperpe{T}=\{X\in\QHh\mid\Ext^1(T,X)=0\}$. The class $\rperpe{T}$
  is then called a \emph{tilting class}.
\end{definit}
Such ``large'' tilting objects in general do not produce derived
equivalences in the way mentioned above. But they yield recollements
of triangulated categories
\cite{bazzoni:2010,angeleri:koenig:liu:2011,chen:xi:2012}, still
providing a strong relationship between the derived categories
involved.

Large tilting modules occur frequently. For example, they arise when
looking for complements to partial tilting modules, or when computing
intersections of tilting classes given by classical tilting modules,
and they parametrize resolving subcategories of finitely presented
modules.  We refer to~\cite{angeleri:2013} for a survey on these
results.

Another reason for the interest in large tilting modules is their deep
connection with localization theory. This is best illustrated by the
example of a Dedekind domain $R$.  The tilting modules over $R$ are
parametrized by the subsets $V\subseteq\maxspec{R}$, and they arise
from localizations at sets of simple modules. More precisely, the
universal localization $R\hookrightarrow R_V$ at the simples supported
in $V$ yields the tilting module $T_V=R_V\oplus R_V/R$, and the set
$V=\emptyset$ corresponds to the regular module $R$, the only finitely
generated tilting module~\cite[Cor.~6.12]{angeleri:sanchez:2011}.

Similar results hold true in more general contexts. Over a commutative
noetherian ring, the tilting modules of projective dimension one
correspond to categorical localizations in the sense of
Gabriel~\cite{angeleri:pospisil:stovicek:trlifaj:2014}. Over a
hereditary ring, tilting modules parametrize universal
localizations~\cite{angeleri:marks:vitoria:2015}.

An interesting example is provided by the Kronecker algebra
$\Lambda$. Here we have a complete analogy to the Dedekind case if we
replace the maximal spectrum by the index set $\mathbb X$ of the
tubular family $\tube=\coprod_{x\in\mathbb X}\mathcal U_x$. Indeed,
the infinite dimensional tilting modules are parametrized by the
subsets $V\subseteq\mathbb X$, and they arise from localizations at
sets of simple regular modules. Again, the universal localization
$\Lambda\hookrightarrow\Lambda_V$ at the simple regular modules
supported in $V$ yields the tilting module
$T_V=\Lambda_V\oplus \Lambda_V/\Lambda$, and the set $V=\emptyset$
corresponds to the Lukas tilting module $\mathbf L$.

For arbitrary tame hereditary algebras, the classification of tilting
modules is more complicated due to the possible presence of finite
dimensional direct summands from non-homogeneous tubes. Infinite
dimensional tilting modules are parametrized by pairs $(B,V)$ where
$B$ is a so-called branch module, and $V$ is a subset of $\mathbb
X$.
The tilting module corresponding to $(B,V)$ has finite dimensional
part $B$ and an infinite dimensional part which is of the form $T_V$
inside a suitable subcategory, see~\cite{angeleri:sanchez:2013}.

\medskip

In the present paper, we tackle the problem of classifying large
tilting objects in hereditary Grothendieck categories. In particular,
we will consider the category $\Qcoh\XX$ of quasicoherent sheaves over
a weighted noncommutative regular projective curve $\XX$ over a field
$k$, in the sense of~\cite{kussin:2014}. We will discuss how the
results described above for tame hereditary algebras extend to this
more general setting.

As in module categories, a crucial role will be played by the following notion.
\begin{definit}
  Let $\QHh$ be a locally coherent Grothendieck category, and let
  $\Hh$ the class of finitely presented objects in $\QHh$. We call a
  class $\vSs\subseteq\Hh$ \emph{resolving} if it generates $\QHh$ and
  has the following closure properties: $\vSs$ is closed under
  extensions, direct summands, and $S'\in\vSs$ whenever
  $0\ra S'\ra S\ra S''\ra 0$ is exact with $S,\,S''\in\vSs$.
\end{definit}
We will use \cite{saorin:stovicek:2011} to show the following general existence result for  tilting
objects.
\begin{theorem-1}\label{theorem-1}{\rm [Theorem~\ref{thm:tilting-from-resolving}]}
  Let $\QHh$ be a locally coherent Grothendieck category and
  $\vSs\subseteq\Hh$ be resolving with $\pd(S)\leq 1$ for all
  $S\in\vSs$. Then there is a tilting object $T$ in $\QHh$ with
  $\rperpe{T}=\rperpe{\vSs}$.
\end{theorem-1}
Tilting classes as above of the form $\rperpe{T}=\rperpe{\vSs}$ for
some class $\vSs$ of finitely presented objects are said to be of
\emph{finite type}.

When $\QHh=\Qcoh\XX$, the category of finitely presented objects
$\Hh=\coh\XX$ is given by the coherent sheaves, and we have
\begin{theorem-2}\label{theorem-2}{\rm [Theorem~\ref{thm:class-correspondence}]}
  Let $\XX$ be a weighted noncommutative regular projective curve and
  $\QHh=\Qcoh\XX$. The assignment
  $\vSs\mapsto\rperpe{\vSs}$  defines a  bijection between
  \begin{itemize}
  \item   resolving classes $\vSs$ in $\Hh$, and
  \item   tilting classes $\rperpe{T}$  of finite type.
  \end{itemize}
\end{theorem-2}
In a module category, all tilting classes have finite type
by~\cite{bazzoni:herbera:2008}. In well behaved cases we can import
this result to our situation. The complexity of the category $\coh\XX$
of coherent sheaves over $\XX$ depends on the orbifold Euler
characteristic $\chi'_{orb}$.  If $\chi'_{orb}(\XX)>0$, then the
category $\coh\XX$ is of (tame) domestic type, and it is
derived-equivalent to the category $\mod H$ for a (finite-dimensional)
tame hereditary algebra $H$. In this case, all tilting classes have
finite type, and we obtain a complete classification of all large
tilting sheaves (Theorem~\ref{thm:full-classif-domestic}), which - not
surprisingly - is very similar to the classification
in~\cite{angeleri:sanchez:2013}.  But also in the tubular case, where
$\XX$ is weighted of orbifold Euler characteristic
$\chi'_{orb}(\XX)=0$, tilting classes turn out to have finite type.

\medskip

Before we discuss our classification results, let us give some details
on the tools we will employ.  Our starting point is given by the
following splitting property.
\begin{theorem-3}\label{theorem-3}{\rm [Theorem~\ref{thm:torsion-splitting}]}
  Let $T\in\Qcoh\XX$ be a sheaf with $\Ext^1(T,T)=0$. Then there is a
  split exact sequence $0\ra tT\ra T\ra T/tT\ra 0$ where
  $tT\subseteq T$ denotes the (largest) torsion subsheaf of $T$ and is
  a direct sum of finite length sheaves and of injective sheaves.
\end{theorem-3}
This result shows that the classification of large (= non-coherent)
tilting sheaves splits, roughly speaking, into two steps:
\begin{enumerate}
\item[(i)] The first is the classification of large tilting sheaves
  $T$ which are torsionfree (that is, with $tT=0$). This seems to be a
  very difficult problem in general, but it turns out that in the
  cases when $\XX$ is a noncommutative curve of genus zero which is of
  domestic or of tubular type, we get all these tilting sheaves with
  the help of Theorem~\ref{theorem-1}.
\item[(ii)] If, on the other hand, the torsion part $tT$ of a large
  tilting sheaf $T$ is non-zero, then it is quite straightforward to
  determine the shape of $tT$; it is a direct sum of Pr\"ufer sheaves
  and a certain so-called branch sheaf $B$, which is coherent. We can
  then apply perpendicular calculus to $B$, in order to reduce the
  problem to the case that $tT$ is a direct sum of Pr\"ufer sheaves,
  or to $tT=0$, which is the torsionfree case~(i).
\end{enumerate}
If Pr\"ufer sheaves occur in the torsion part, then the corresponding
torsionfree part is uniquely determined. This leads to the following,
general result:
\begin{theorem-4}\label{theorem-4}{\rm [Corollary~\ref{cor:nctp}]}
  Let $\XX$ be a weighted noncommutative regular projective curve. The
  tilting sheaves in $\Qcoh\XX$ which have a non-coherent torsion
  subsheaf are up to equivalence in bijective correspondence with
  pairs $(B,V)$, where $V$ is a non-empty subset of $\XX$ and $B$ is a
  branch sheaf.
\end{theorem-4}
We will see in Section~\ref{calculus} that the tilting sheaf
corresponding to $(B,V)$ has coherent part $B$ and a non-coherent part
$T_V$ formed inside a suitable subcategory by a construction which is
analogous to universal localization. In particular, the torsionfree
part of $T_V$ can be interpreted as a projective generator of the
quotient category obtained from $\Qcoh\XX$ by localization at the
simple objects supported in $V$.  Of course, there are also tilting
sheaves given by pairs $(B,V)$ with $V=\emptyset$. Here the
non-coherent part is the Lukas tilting sheaf inside a suitable
subcategory, that is, it is given by the resolving class formed by all
vector bundles. Altogether, the pairs $(B,V)$ correspond to Serre
subcategories of $\coh\XX$, and tilting sheaves are closely related
with Gabriel localization, like in the case of tilting modules over
commutative noetherian rings,
cf.~also~\cite[Sec.~5]{angeleri:kussin:2015}.

Let us now discuss the tubular
case. Following~\cite{reiten:ringel:2006}, we define for every
$w\in\RR\cup\{\infty\}$ the class $\Mm(w)$ of quasicoherent sheaves of
slope $w$. Reiten and Ringel have shown~\cite{reiten:ringel:2006} that
every indecomposable object has a well-defined slope. Our main result
is as follows.
\begin{theorem-5}\label{theorem-5}{\rm [Theorem~\ref{thm:tubular-full-classification}]}
  Let $\XX$ be of tubular type. Then every large tilting sheaf in
  $\Qcoh\XX$ has a well-defined slope $w$. If $w$ is irrational, then
  there is up to equivalence precisely one tilting sheaf of slope
  $w$. If $w$ is rational or $\infty$, then the large tilting sheaves
  of slope $w$ are classified like in the domestic case.
\end{theorem-5}
In Section~\ref{sec:elliptic}, we will briefly discuss the elliptic
case, where $\chi'_{orb}(\XX)=0$ and $\XX$ is non-weighted.  Some of
our main results will extend to this situation. In particular,
Theorem~\ref{thm:elliptic-tilting} will resemble the tubular case
described above.  As it turns out, this will be much easier than in
the (weighted) tubular case, using an Atiyah's~\cite{atiyah:1957} type
classification, namely, that all coherent sheaves lie in homogeneous
tubes.

When the orbifold Euler characteristic $\chi'_{orb}(\XX)\ge 0$, our
results also yield a classification of certain resolving classes in
$\coh\XX$ (see Corollaries~\ref{resdom} and~\ref{restub} and
Theorem~\ref{thm:elliptic-tilting}(5)).
If $\chi'_{orb}(\XX)<0$, then $\coh\XX$ is wild. We stress that
Theorem 4
also holds in this case, but we have not
attempted to classify the torsionfree large tilting sheaves in the
wild case.

There is one main difference to the module case. We recall that one of
the standard characterising properties of a tilting module
$T\in\Mod R$ is the existence of an exact sequence
$$0\ra R\ra T_0\ra T_1\ra 0$$ 
with $T_0,\,T_1\in\Add(T)$. As already mentioned, in $\Qcoh\XX$ we
have lack of a projective generator. When $\XX$ has genus zero, the replacement for the ring $R$
in our category is a tilting bundle $T_{\can}$ whose endomorphism ring
is a canonical algebra. It will be a non-trivial result that for every
large tilting sheaf $T$ we can always find such a tilting bundle
$T_{\can}$ and a short exact sequence $0\ra T_{\can}\ra T_0\ra T_1\ra
0$, even with $T_0,\,T_1\in\add(T)$. If $T$ has a non-coherent torsion
part, then we can even do this with $\Hom(T_1,T_0)=0$,
cf.~Theorem~\ref{thm:axiom-TS3-for-T_VB}.

Since noncommutative curves of genus zero are derived-equivalent to
canonical algebras in the sense of Ringel and
Crawley-Boevey~\cite{ringel:crawley-boevey:1990}, our results are
closely related to the classification of large tilting modules over
canonical algebras. The module case is treated more directly in
\cite{angeleri:kussin:2015}, where we also address the dual concept of cotilting modules
and the classification of pure-injective modules.

\section{Weighted noncommutative regular projective curves}
\label{sec:sheaves-and-modules}
In this section we collect some preliminaries on the category of
quasicoherent sheaves we are going to study and we introduce large
tilting sheaves.
\subsection*{The axioms}
We define the class of noncommutative curves, which we will study in
this paper, by the axioms (NC~1) to (NC~5) below; the condition (NC~6)
will follow from the others. Such a curve $\XX$ is given by a category
$\Hh$ which is regarded as the category $\coh\XX$ of \emph{coherent
  sheaves} over $\XX$. Formally it has the same main properties like a
category of coherent sheaves over a (commutative) regular projective
curve over a field $k$ (we refer to~\cite{kussin:2014}):
\begin{enumerate}
\item[(NC 1)] $\Hh$ is small, connected, abelian and every object in
  $\Hh$ is noetherian;
\item[(NC 2)] $\Hh$ is a $k$-category with finite-dimensional Hom- and
  Ext-spaces;
\item[(NC 3)] There is an autoequivalence $\tau$ on $\Hh$, called
  Auslander-Reiten translation, such that Serre
  duality $$\Ext^1_{\Hh}(X,Y) = \D\Hom_{\Hh}(Y,\tau X)$$ holds, where
  $\D=\Hom_k (-,k)$. (In particular $\Hh$ is then hereditary.)
\item[(NC 4)] $\Hh$ contains an object of infinite length.
\end{enumerate}
\subsection*{Splitting of coherent sheaves}
Assume $\Hh$ satisfies (NC~1) to (NC~4). The following rough picture
of the category $\Hh$ is very useful
(\cite[Prop.~1.1]{lenzing:reiten:2006}). Every indecomposable coherent
sheaf $E$ is either of finite length, or it is \emph{torsionfree},
that is, it does not contain any simple sheaf; in the latter case $E$
is also called a (vector) bundle. We thus write $\Hh=\Hh_+\vee\Hh_0$,
with $\Hh_+=\vect\XX$ the class of vector bundles and $\Hh_0$ the class of
sheaves of finite length; we have $\Hom(\Hh_0,\Hh_+)=0$. Decomposing
$\Hh_0$ in its connected components we
have $$\Hh_0=\coprod_{x\in\XX}\Uu_x,$$ where $\XX$ is an index set
(explaining the terminology $\Hh=\coh\XX$) and every $\Uu_x$ is a
connected uniserial length category.
\subsection*{Weighted noncommutative regular projective curves}
Assume that $\Hh$ is a $k$-category satisfying properties (NC~1) to
(NC~4) and the following additional condition.
\begin{enumerate}
\item[(NC 5)] $\XX$ consists of infinitely many points. 
\end{enumerate}
Then we call $\XX$ (or $\Hh$) a \emph{weighted} (or \emph{orbifold})
\emph{noncommutative regular projective curve} over $k$. ``Regular''
can be replaced by ``smooth'' if $k$ is a perfect field; we refer
to~\cite[Sec.~7]{kussin:2014}. We refer also
to~\cite{lenzing:reiten:2006}; we excluded certain degenerated cases
described therein by our additional axiom (NC~5). It is shown in~\cite{kussin:2014} that
a weighted noncommutative regular projective curve $\XX$ satisfies
automatically also the following condition. 
\begin{enumerate}
\item[(NC 6)] For all points $x\in\XX$ there are (up to isomorphism)
  precisely $p(x)<\infty$ simple objects in $\Uu_x$, and for almost
  all $x$ we have $p(x)=1$.
\end{enumerate}
We remark that the ``classical'' case $\Hh=\coh X$ with $X$ a regular
projective curve is included in this setting. This classical case is
extended into two directions: (1) curves with a noncommutative
function field $k(\XX)$ are allowed; here $k(\XX)$ is a skew field
which is finite dimensional over its centre, which has the form $k(X)$
for a regular projective curve $X$; (2) additionally (a finite number
of) weights are allowed.
\subsection*{Genus zero}
We consider also the following condition.
\begin{enumerate}
\item[(g-0)] $\Hh$ admits a tilting object. 
\end{enumerate}
It is shown in~\cite{lenzing:delapena:1999} that then $\Hh$ even
contains a torsionfree tilting object $T_{\can}$ whose endomorphism
algebra is canonical, in the sense
of~\cite{ringel:crawley-boevey:1990}. We call such a tilting object
canonical, or, by considering the full subcategory formed by the
indecomposable summands of $T_{\can}$, \emph{canonical
  configuration}. We recall that $T\in\Hh$ is called \emph{tilting},
if $\Ext^1(T,T)=0$, and if for all $X\in\Hh$ we have $X=0$ whenever
$\Hom(T,X)=0=\Ext^1(T,X)$. (This notion will be later generalized to
quasicoherent sheaves.) If $\Hh$ satisfies (NC~1) to (NC~4) and (g-0),
then we say that $\XX$ is a \emph{noncommutative curve of genus zero};
the condition (NC~5) is then automatically satisfied, we refer
to~\cite{kussin:2009}. The weighted projective lines, defined by
Geigle-Lenzing~\cite{geigle:lenzing:1987}, are special cases of
noncommutative curves of genus zero.
\subsection*{The Grothendieck group and the Euler form}
We write $[X]$ for the class of a coherent sheaf $X$ in the
Grothendieck group $\Knull(\Hh)$ of $\Hh$. The Grothendieck group is
equipped with the Euler form, which is defined on classes
of objects $X$, $Y$ in $\Hh$ by
  $$\LF{[X]}{[Y]}=\dim_k
  \Hom(X,Y)-\dim_k \Ext^1(X,Y).$$
  We will usually write $\LF{X}{Y}$, without the brackets.

  In case $\XX$ is of genus zero, $\Hh$ admits a tilting object whose
  endomorphism ring is a finite dimensional algebra, and thus the
  Grothendieck group $\Knull(\Hh)$ of $\Hh$ is finitely generated free
  abelian. (From this it follows more directly that every $\XX$ of
  genus zero satisfies (NC~6).)

  \bigskip \emph{In the following, if not otherwise specified, let
  $\Hh=\coh\XX$ be a weighted noncommutative regular projective curve.}
\subsection*{Homogeneous and exceptional tubes}
For every $x\in\XX$ the connected uniserial length categories are
called \emph{tubes}. The number $p(x)\geq 1$ is called the \emph{rank}
of the tube $\Uu_x$. Tubes of rank $1$ are called \emph{homogeneous},
those with $p(x)>1$ \emph{exceptional}. If $S_x$ is a simple sheaf in
$\Uu_x$, then $\Ext^1(S_x,S_x)\neq 0$ in the homogeneous case, and
$\Ext^1(S_x,S_x)=0$ in the exceptional case. More generally, a
coherent sheaf $E$ is called \emph{exceptional}, if $E$ is
indecomposable and $E$ has no self-extensions. It follows then by a
well-known argument of Happel and Ringel that $\End(E)$ is a skew
field. It is well-known and easy to see that the exceptional sheaves
in $\Uu_x$ are just those indecomposables of length $\leq p(x)-1$
(which exist only for $p(x)>1$). In particular there are only finitely
many exceptional sheaves of finite length.

If $p=p(x)$, then all simple sheaves in $\Uu_x$ are given (up to
isomorphism) by the Auslander-Reiten orbit $S_x=\tau^pS_x,\,\tau
S_x,\dots,\tau^{p-1}S_x$. 

For the terminology on wings and branches in exceptional tubes we refer to~\ref{wings}.

\subsection*{Non-weighted curves}
By a \emph{(non-weighted) noncommutative regular projective curve}
over the field $k$ we mean a category $\Hh=\coh\XX$ satisfying
axioms~(NC~1) to (NC~5), and additionally
\begin{list}{(NC 6')}{\setlength{\topsep}{2.2pt plus 2.2pt}}
\item $\Ext^1(S,S)\neq 0$ (equivalently: $\tau S\simeq S$)
  holds for all simple objects $S\in\Hh$.
\end{list}
This condition means that all tubes are homogeneous, that is, $p(x)=1$
for all $x\in\XX$; therefore these curves are also called homogeneous
in~\cite{kussin:2009}. For a detailed treatment of this setting we
refer to~\cite{kussin:2014}. We stress that thus, by abuse of
language, non-weighted curves are special cases of weighted curves.
\subsection*{Grothendieck categories with finiteness conditions}
Let us briefly recall some notions we will need in the sequel.  An
abelian category $\Aa$ is a \emph{Grothendieck category}, if it is
cocomplete, has a generator, and direct limits are exact. Every
Grothendieck category is also complete and has an injective
cogenerator. A Grothendieck category is called \emph{locally coherent}
(resp.\ \emph{locally noetherian}, resp.\ \emph{locally finite}) if it
admits a system of generators which are coherent (resp.\ noetherian,
resp.\ of finite length). In this case every object in $\Aa$ is a
direct limit of coherent (resp.\ noetherian, resp.\ finite length)
objects. If $\Aa$ is locally coherent then the coherent and the
finitely presented objects coincide, and the full subcategory
$\fp(\Aa)$ of finitely presented objects is abelian. For more details
on Grothendieck categories we refer
to~\cite{gabriel:1962,stenstroem:1975,herzog:1997,krause:1997}.
\subsection*{The Serre construction}
$\Hh=\coh\XX$ is a noncommutative noetherian projective scheme in the
sense of Artin-Zhang~\cite{artin:zhang:1994} and satisfies Serre's
theorem. This means that there is a positively $H$-graded (not
necessarily commutative) noetherian ring $R$ (with $(H,\leq)$ an
ordered abelian group of rank one) such that
\begin{equation}
  \label{eq:Serre-Thm-small}
  \Hh=\frac{\mod^H(R)}{\mod^H_0(R)},
\end{equation}
the quotient category of the category of finitely generated $H$-graded
modules modulo the Serre subcategory of those modules which are
finite-dimensional over $k$. (We refer
to~\cite[Prop.~6.2.1]{kussin:2009}, \cite{kussin:2014}
and~\cite[Lem.~IV.4.1]{reiten:vandenbergh:2002}.) With this
description we can define $\QHh=\Qcoh\XX$ as the quotient category
\begin{equation}
  \label{eq:Serre-Thm-large}
  \QHh=\frac{\Mod^H(R)}{\Mod^H_0(R)},
\end{equation}
where $\Mod^H_0(R)$ denotes the localizing subcategory of $\Mod^H(R)$
of all $H$-graded torsion, that is, locally finite-dimensional,
modules. 
The category
$\QHh$ is hereditary abelian, and a locally noetherian Grothendieck
category; every object in $\QHh$ is a direct limit of objects in $\Hh$
(therefore the symbol $\QHh$). 
The full abelian subcategory $\Hh$ consists of the coherent
(= finitely presented = noetherian) objects in $\QHh$, we also write
$\Hh=\fp(\QHh)$. Every indecomposable coherent sheaf has a local
endomorphism ring, and $\Hh$ is a Krull-Schmidt category. 

We remark that $\QHh$ can, by~\cite[II.~Thm.~1]{gabriel:1962}, also be
recovered from its subcategory $\Hh$ of noetherian objects as
the category of left-exact (covariant) $k$-functors from $\Hh^{\op}$ to
$\Mod(k)$.
\subsection*{Pr\"ufer sheaves}
Let $E$ be an indecomposable sheaf in a tube $\Uu_x$. By the
\emph{ray} starting in $E$ we mean the (infinite) sequence of all the
indecomposable sheaves in $\Uu_x$, which contain $E$ as a
subsheaf. The corresponding monomorphisms (inclusions) form a direct
system. If the socle of $E$ is the simple $S$, then the corresponding
direct limit of this system is the \emph{Pr\"ufer sheaf}
$S[\infty]$. In other words, $S[\infty]$ is the union of all
indecomposable sheaves of finite length containing $S$ (or
$E$). Dually we define \emph{corays} ending in $E$ as the sequence of
all indecomposable sheaves in $\Uu_x$ admitting $E$ as a factor.

If $S$ is a simple sheaf, then we denote by $S[n]$ the (unique)
indecomposable sheaf of length $n$ with socle $S$. Thus, the
collection $S[n]$ $(n\geq 1$) forms the ray starting in $S$, and their
union is $S[\infty]$. The Pr\"ufer sheaves form an important class of
indecomposable (we refer to~\cite{ringel:1975}), quasicoherent,
non-coherent sheaves.
\subsection*{Structure sheaf. Rank. Line bundles.} 
The structure sheaf $L\in\Hh$ is the sheaf associated with the graded
$R$-module $R$ under the equivalence~\eqref{eq:Serre-Thm-small}.  Let
$\Hh/\Hh_0$ be the quotient category of $\Hh$ modulo the Serre
category of sheaves of finite length, let $\pi\colon\Hh\ra\Hh/\Hh_0$
the quotient functor, which is exact. The endomorphism ring $k(\Hh)$
of the class of $\pi L$ in the quotient category $\Hh/\Hh_0$ is a skew
field, called the \emph{function field} of $\Hh$, and,
by~\cite[Prop.~3.4]{lenzing:reiten:2006}, we have
$\Hh/\Hh_0=\mod(k(\Hh))$ and $\QHh/\vec{\Hh_0}=\Mod(k(\Hh))$. The
$k(\Hh)$-dimension on $\Hh/\Hh_0$ induces the \emph{rank} function on
$\Hh$ by the formula $\rk(F):=\dim_{k(\Hh)}(\pi F)$. It is additive on
short exact sequences and thus induces a linear form
$\rk\colon\Knull(\Hh)\ra\ZZ$. The objects in $\Hh_0$ are just the
objects of rank zero, every non-zero vector bundle has a positive
rank,~\cite[Prop.~1.2]{lenzing:reiten:2006}. The vector bundles of
rank one are called \emph{line bundles}. The structure sheaf $L$ is a
so-called \emph{special} line bundle which means that for each
$x\in\XX$ there is (up to isomorphism) \emph{precisely one} simple
sheaf $S_x$ concentrated in $x$ with $\Ext^1(S_x,L)\neq 0$; we refer
also to~\cite[Prop.~7.8]{kussin:2014}. We will later require an
additional property for $L$, when treating the orbifold Euler
characteristic.

Furthermore,  every non-zero morphism from a line bundle $L'$ to a
vector bundle is a monomorphism, and $\End(L')$ is a skew
field,~\cite[Lem.~1.3]{lenzing:reiten:2006}. Every vector bundle has a
line bundle filtration,~\cite[Prop.~1.6]{lenzing:reiten:2006}.

More details on the role of line bundles in connection with canonical
configurations will be given in
Remark~\ref{rem:canonical-subconfiguration}
and~\ref{nr:canonical-sequences}.

\subsection*{The sheaf of rational functions}
The \emph{sheaf} $\Kk$ \emph{of rational functions} is the injective
envelope of the structure sheaf $L$ in the category $\QHh$. This is,
besides the Pr\"ufer sheaves, another very important quasicoherent,
non-coherent sheaf. It is torsionfree by~\cite[Lem.~14]{kussin:2000},
 and it is a generic
sheaf in the sense of~\cite{lenzing:1997}; its endomorphism ring is
the function field, $\End_{\QHh}(\Kk)\simeq\End_{\Hh/\Hh_0}(\pi
L)\simeq k(\Hh)$. 
\subsection*{The derived category}
Since $\QHh=\Qcoh\XX$ is a hereditary category, the derived
category
\begin{equation}
  \label{eq:def-derived-cat}
  \Dd=\Derived{\QHh}=\Add\biggl(\bigvee_{n\in\ZZ}\QHh[n]\biggr)
\end{equation}
is the repetitive category of $\QHh$. This means: Every object in $\Dd$
can be written as $\bigoplus_{i\in I}X_i [i]$ for a subset
$I\subseteq\ZZ$ and $X_i\in\QHh$ for all $i$, and for all objects
$X,\,Y\in\QHh$ and all integers $n,\,m$ we
have $$\Ext^{n-m}_{\QHh}(X,Y)=\Hom_{\Dd}(X[m],Y[n]).$$ Similarly, we
have the \emph{bounded} derived category $\bDd=\bDerived{\QHh}$, where
in~\eqref{eq:def-derived-cat}, $\Add$ is replaced by $\add$, and the
subset $I$ in $\ZZ$ as above is finite.
\subsection*{Generalized Serre duality} 
It follows easily from~\cite[Thm.~4.4]{krause:2005} that on $\QHh$ we
have Serre duality in the following sense. Let $\tau$ be the
Auslander-Reiten translation on $\Hh$ and $\tau^-$ its (quasi-)
inverse. For all $X\in\Hh$ and all $Y\in\QHh$ we
have
$$\D\Ext^1_{\QHh}(X,Y)=\Hom_{\QHh}(Y,\tau
X)\quad\text{and}\quad\Ext^1_{\QHh}(Y,X)=\D\Hom_{\QHh}(\tau^- X,Y),$$
with $\D$ denoting the duality $\Hom_k (-,k)$.
\begin{remark}\label{rem:pure-injective}
  \emph{Every coherent sheaf $F\in\Hh$ is pure-injective.} Indeed, if
  $\mu$ is a pure exact sequence in $\QHh$, then
  $\Hom_{\QHh}(\tau^-F,\mu)$ is exact. Since $\Ext^2_{\QHh}(-,-)$
  vanishes, this amounts to exactness of $\Ext^1_{\QHh}(\tau^-F,\mu)$,
  and hence of $\D\Ext^1_{\QHh}(\tau^-F,\mu)$, which in turn is
  equivalent to exactness of $\Hom_{\QHh}(\mu, F)$ by Serre
  duality. This gives the claim.
\end{remark}
\subsection*{Almost split sequences} Since the objects of $\Hh$ are
pure-injective, it follows directly from~\cite[Prop.~3.2]{krause:2005}
that the category $\Hh$ has almost split sequences which also satisfy
the almost split properties in the larger category $\QHh$; more
precisely: for every indecomposable $Z\in\Hh$ there is a non-split
short exact sequence $$0\ra X\stackrel{\alpha}\lra
Y\stackrel{\beta}\lra Z\ra 0$$ in $\Hh$ with $X=\tau Z$ indecomposable such
that for every object $Z'\in\QHh$ any morphism $Z'\ra Z$ that is not a
retraction factors through $\beta$ (and equivalently, for every object
$X'\in\QHh$ any morphism $X\ra X'$ that is not a section factors
through $\alpha$). 
\subsection*{Orbifold Euler characteristic and representation type}
For the details of all the notions and results in this subsection we
refer to~\cite{kussin:2014}. Let $\Hh$ be a weighted noncommutative
regular projective curve over $k$ with structure sheaf $L$.  The
centre of the function field $k(\Hh)$ is of the form $k(X)$, the
function field of a unique regular projective curve $X$ over $k$. We
call $X$ the \emph{centre curve} of $\Hh$. The dimension
$[k(\Hh):k(X)]$ is finite, a square number, which we denote by
$s(\Hh)^2$. The (closed) points of $X$ are in one-to-one
correspondence to the (closed) points of $\XX$. Let $\Oo=\Oo_X$ be the
structure sheaf of $X$. For every $x\in X$ we have the local rings
$(\Oo_x,\mathfrak{m}_x)$, and the residue class field
$k(x)=\Oo_x/\mathfrak{m}_x$. For all $x\in\XX$ there are the
ramification indices $e_{\tau}(x)\geq 1$. There exist only finitely
many points $x\in\XX$ with $p(x)e_{\tau}(x)>1$. By a result of Reiten
and van den Bergh~\cite{reiten:vandenbergh:2002} the category $\Hh$
can be realized as $\Hh=\coh(\Aa)$, the category of coherent
$\Aa$-modules, where $\Aa$ is a torsionfree coherent sheaf of
hereditary $\Oo$-orders in a finite-dimensional central simple
$k(X)$-algebra. Moreover, $\QHh=\Qcoh(\Aa)$. For our structure sheaf
$L$ we can and will always additionally assume, that it is an
indecomposable direct summand of $\Aa$.\medskip

Let $\ovp$ be the least common multiple of the weights $p(x)$. One
defines the \emph{average Euler form}
$\DLF{E}{F}=\sum_{j=0}^{\ovp-1}\LF{\tau^j E}{F}$, and then the
(normalized) \emph{orbifold Euler characteristic} of $\Hh$ by
$\chi'_{orb}(\XX)=\frac{1}{s(\Hh)^2\ovp^2}\DLF{L}{L}$. If $k$ is
perfect, one has a nice formula to compute the Euler characteristic: 
\begin{equation}
  \label{eq:orbifold-euler-char}
  \chi'_{orb}(\XX)=\chi(X)-\frac{1}{2}
  \sum_{x}\Bigl(1-\frac{1}{p(x)e_{\tau}(x)}\Bigr)[k(x):k].
\end{equation}
Here, $\chi(X)=\dim_k\Hom_X(\Oo,\Oo)-\dim_k\Ext^1_X(\Oo,\Oo)$ is the
Euler characteristic of the centre curve $X$. If $k$ is not perfect,
there is still a similar formula, we refer
to~\cite[Cor.~13.13]{kussin:2014}.\medskip

The orbifold Euler characteristic determines the representation type
of the category $\Hh=\coh\XX$ (see also
Theorem~\ref{thm:stability} below):
\begin{itemize}
\item  $\XX$ is domestic: $\chi'_{orb}(\XX)>0$
\item $\XX$ is elliptic: $\chi'_{orb}(\XX)=0$, and $\XX$ non-weighted
  ($\ovp=1$)
\item $\XX$ is tubular: $\chi'_{orb}(\XX)=0$, and $\XX$ properly
  weighted ($\ovp>1$)
\item $\XX$ is wild: $\chi'_{orb}(\XX)<0$.
\end{itemize}
In this paper we will prove some general results for all
representation types, and we will obtain finer and complete
classification results in the cases of nonnegative orbifold Euler
characteristic. \medskip

If $\XX$ is non-weighted with structure sheaf $L$, then we call the
number $g(\XX)=[\Ext^1(L,L):\End(L)]$ the \emph{genus} of $\XX$. The
condition $g(\XX)=0$ is equivalent to condition (g-0), the condition
$g(\XX)=1$ to the elliptic case.  \medskip

If $\XX$ is weighted then there is an underlying non-weighted curve
$\XX_{nw}$, which follows from~(NC~6) by perpendicular
calculus~\cite{geigle:lenzing:1991}.  Then $\Hh=\coh\XX$ contains a
tilting bundle (that is, $\Hh$ satisfies (g-0)) if and only if
$g(\XX_{nw})=0$. In other words, $\Hh$ satisfies (g-0) if the genus,
in the non-orbifold sense, is zero.
\subsection*{Degree and slope}
We define the \emph{degree} function $\deg\colon\Knull(\Hh)\ra\ZZ$, by
$$\deg(F)=\frac{1}{\kappa\varepsilon}\DLF{L}{F}-\frac{1}{\kappa\varepsilon}\DLF{L}{L}\rk(F),$$ 
with $\kappa=\dim_k\End(L)$ and $\varepsilon$ the positive integer such
that the resulting linear form $\Knull(\Hh)\ra\ZZ$ becomes
surjective. We have $\deg(L)=0$, and $\deg$ is positive and
$\tau$-invariant on sheaves of finite length. The \emph{slope} of a
non-zero coherent sheaf $F$ is defined as
$\mu(F)=\deg(F)/\rk(F)\in\widehat{\QQ}=\QQ\cup\{\infty\}$. Moreover,
$F$ is called \emph{stable} (\emph{semi-stable}, resp.) if for every
non-zero proper subsheaf $F'$ of $F$ we have $\mu(F')<\mu(F)$ (resp.\
$\mu(F')\leq\mu(F)$).

More details on these numerical invariants will be given in~\ref{nr:numerical-invariants}.

\subsection*{Stability} 
The stability notions are very useful for the classification of vector
bundles (we refer to~\cite[Prop.~5.5]{geigle:lenzing:1987},
\cite{lenzing:reiten:2006}, \cite[Prop.~8.1.6]{kussin:2009},
\cite{kussin:2014}):
\begin{theorem}\label{thm:stability}
  Let $\Hh=\coh\XX$ be a weighted noncommutative regular projective
  curve over $k$.
  \begin{enumerate}
  \item If $\chi'(\XX)>0$ (domestic type), then every indecomposable
    vector bundle is stable and exceptional. Moreover, $\coh\XX$
    admits a tilting bundle. \smallskip
  \item If $\chi'(\XX)=0$ (elliptic or tubular type), then every
    indecomposable coherent sheaf is semistable. For every
    $\alpha\in\widehat{\QQ}$ the full subcategory $\Cc_\alpha$ of
    $\Hh$ formed by the semistable sheaves of slope $\alpha$ is a
    non-trivial abelian unsiserial category whose connected components
    form stable tubes; the tubular family $\Cc_\alpha$ is parametrized
    again by a weighted noncommutative regular projective curve
    $\XX_\alpha$ over $k$ with $\chi'(\XX_{\alpha})=0$ and which is
    derived-equivalent to $\XX$.  Moreover, if $\XX$ is tubular (that
    is, $\ovp>1$), then $\coh\XX$ admits a tilting bundle. If $\XX$ is
    elliptic (that is, $\ovp=1$) then every indecomposable coherent
    sheaf is non-exceptional. \smallskip
  \item If $\chi'(\XX)<0$, then every Auslander-Reiten component in
    $\Hh_+=\vect\XX$ is of type $\ZZ A_{\infty}$, and $\Hh$ is of
    wild representation type. ($\coh\XX$ may or may not
    satisfy~(g-0).) \qed
  \end{enumerate}
\end{theorem}
\subsection*{Orthogonal and generated classes} Let $\Xx$ be a class of
objects in $\QHh$. We will use the following notation:
\begin{gather*}
  \rperpo{\Xx}=\{F\in\QHh\mid\Hom(\Xx,F)=0\},\quad
  \rperpe{\Xx}=\{F\in\QHh\mid\Ext^1(\Xx,F)=0\},\\
  \lperpo{\Xx}=\{F\in\QHh\mid\Hom(F,\Xx)=0\},\quad
  \lperpe{\Xx}=\{F\in\QHh\mid\Ext^1(F,\Xx)=0\},\\
    \rperp{\Xx}=\rperpo{\Xx}\cap\rperpe{\Xx},
    \quad \lperp{\Xx}=\lperpo{\Xx}\cap\lperpe{\Xx}.
\end{gather*}
Following~\cite{geigle:lenzing:1991} we call $\lperp{\Xx}$ (resp.\
$\rperp{\Xx}$) the \emph{left-perpendicular} (resp.\
\emph{right-per\-pen\-dicu\-lar}) category of $\Xx$. By $\Add(\Xx)$
(resp.\ $\add(\Xx)$) we denote the class of all direct summands of direct sums of
the form $\bigoplus_{i\in I}X_i$, where $I$ is any set (resp.\ finite
set) and $X_i\in\Xx$ for all $i$. By $\Gen(\Xx)$ we denote
the class of all objects $Y$ \emph{generated by} $\Xx$, that is, such
that there is an epimorphism $X\ra Y$ with $X\in\Add(\Xx)$.\medskip

Let $(I,\leq)$ be an ordered set and $\Xx_i$ classes of objects for
all $i\in I$, in any additive category.  We write
$\bigvee_{i\in I}\Xx_i$ for $\add(\bigcup_{i\in I}\Xx_i)$ if
additionally $\Hom(\Xx_j,\Xx_i)=0$ for all $i<j$ is satisfied. In
particular, notations like $\Xx_1\vee\Xx_2$ and
$\Xx_1\vee\Xx_2\vee\Xx_3$ make sense (where $1<2<3$).\medskip

The following induction technique will be very important.
\subsection*{Reduction of weights}
Let $S$ be an exceptional simple sheaf. In other words, $S$ lies on
the mouth of a tube, with index $x$, of rank $p(x)>1$. Then the right
perpendicular category $\rperp{S}$ is equivalent to $\Qcoh\XX'$,
where $\XX'$ is a curve such that the rank $p'\!(x)$ of the tube of
index $x$ is $p'\!(x)=p(x)-1$ and all other weights and all the numbers
$e_{\tau}(y)$ are preserved. We refer to~\cite{geigle:lenzing:1991}
for details. From the formula~\eqref{eq:orbifold-euler-char}
(and~\cite[Cor.~13.13]{kussin:2014}, which holds over any field) of
the orbifold Euler characteristic we see $\chi'(\XX')>\chi'(\XX)$, and
we conclude that $\XX'$ is of domestic type if $\XX$ is tubular or
domestic.
\subsection*{Tubular shifts}
If $x\in\XX$ is a point of weight $p(x)\geq 1$, then there is an
autoequivalence $\sigma_x$ of $\Hh$ (which extends to an
autoequivalence of $\QHh$), called the \emph{tubular shift} associated
with $x$. We refer to~\cite[Sec.~0.4]{kussin:2009} for more
details. We just recall that for every vector bundle $E$ there is a
universal exact sequence
\begin{equation}
  \label{eq:def-tubular-shift}
  0\ra E\ra\sigma_x(E)\ra E_x\ra 0,
\end{equation}
where $E_x=\bigoplus_{j=0}^{p(x)-1}\Ext^1(\tau^j S_x,E)\otimes\tau^j
S_x\in\Uu_x$ with the tensor product taken over the skew field 
$\End(S_x)$. We also
write $$\sigma_x(E)=E(x)\quad\text{and}\quad(\sigma_x)^n(E)=E(nx),$$
and we will use the more handy notation
$$E_x=\bigoplus_{j=0}^{p(x)-1}(\tau^j S_x)^{e(j,x,E)}$$
with the exponents given by the
multiplicities $$e(j,x,E)=[\Ext^1(\tau^j S_x,E):\End(S_x)].$$ In the
particular case when $E=L$ is the structure sheaf (which is a special
line bundle), and $S_x$ is such that $\Hom(L,S_x)\neq 0$, we have
$e(j,x,L)=e(x)$ for $j=p(x)-1$ and $=0$ otherwise.
\subsection*{Tilting sheaves}
Let $\QHh$ be a Grothendieck category, for instance $\QHh=\Qcoh\XX$.
\begin{definition}
  An object $T\in\QHh$ is called a \emph{tilting object} or
  \emph{tilting sheaf} if $\Gen(T)=\rperpe{T}$. Then $\Gen(T)$ is called
  the associated \emph{tilting class}.
\end{definition}
This definition is inspired by~\cite[Def.~2.3]{colpi:1999}, but we
dispense with the self-smallness assumption made there. In a module
category, we thus recover the definition of a \emph{tilting module}
(of projective dimension one) from~\cite{colpi:trlifaj:1995}.
\begin{lemma}[{\cite[Prop.~2.2]{colpi:1999}}]
  An object $T\in\QHh$ is tilting if and only if the following
  conditions are satisfied:
  \begin{enumerate}
    \item[(TS0)] $T$ has projective dimension  $\pd(T)\leq 1$.\smallskip
    \item[(TS1)] $\Ext^1(T,T^{(I)})=0$ for every cardinal
      $I$.\smallskip
    \item[(TS2)] $\rperp{T}=0$, that is: if $X\in\QHh$ satisfies
      $\Hom(T,X)=0=\Ext^1(T,X)$, then $X=0$.
    \end{enumerate}
\end{lemma}
We will mostly consider hereditary categories $\QHh$ where (TS0) is
automatically satisfied.  In case $\QHh=\Qcoh\XX$ with $\XX$ of genus
zero, we will also consider the following condition, where
$T_{\can}\in\Hh$ is a tilting bundle such that
$\End(T_{\can})=\Lambda$ is a canonical algebra.
\begin{enumerate}
\item[(TS3)] There are an autoequivalence $\sigma$ on $\Hh$ and an
  exact sequence $$0\ra\sigma(T_{\can})\ra T_0\ra T_1\ra 0$$ such that
  $\Add(T_0\oplus T_1)=\Add(T)$; if this can be realized with the
  additional property $\Hom(T_1,T_0)=0$, then we say that $T$
  satisfies condition (TS3+).
\end{enumerate}
Since $\sigma (T_{\can})$ is a tilting bundle, (TS3) implies (TS2). As
it will turn out, in case of genus zero, all tilting sheaves we
construct will satisfy (TS3), and some will even satisfy (TS3+), see
Example~\ref{ex:filtration}, Corollary~\ref{TS3tub}, and
Section~\ref{sec:TS3}. \medskip

\emph{Let $\QHh$ additionally be locally coherent with
  $\Hh=\fp(\QHh)$.}
\begin{lemma}\label{lem:porperties-tilting}
  Let $T\in\QHh$ be tilting.
  \begin{enumerate}
  \item $\Gen(T)=\Pres(T)$.
  \item $\rperpe{T}\cap\lperpe{(\rperpe{T})}=\Add(T)$. 
  \item If $X\in\Hh$ is coherent having a local endomorphism ring and
    $X\in\Add(T)$, then $X$ is a direct summand of $T$.
  \end{enumerate}
\end{lemma}
\begin{proof}
  (1) The same proof as in~\cite[Lemma~1.2]{colpi:trlifaj:1995} works
  here.
  
  (2) Is an easy consequence of~(1).

  (3) Since $X$ is coherent, we get $X\in\add(T)$. Since $X$ has local
  endomorphism ring, the claim follows. 
\end{proof}
\begin{definition}
  Two tilting objects $T$, $T'\in\QHh$ are \emph{equivalent}, if they
  generate the same tilting class. This is equivalent to
  $\Add(T)=\Add(T')$. A tilting sheaf $T\in\QHh$ is called
  \emph{large} if it is not equivalent to a coherent tilting sheaf.
\end{definition}
\emph{For the rest of this section we assume that $\XX$ is of genus
  zero and $\QHh=\Qcoh\XX$ with a fixed special line bundle $L$.}
\subsection*{Tilting bundles and concealed-canonical algebras} 
We fix a tilting bundle $T_{\cc}\in\Hh$. Its endomorphism ring
$\Sigma$ is a concealed-canonical $k$-algebra. Every
concealed-canonical algebra arises in this way, we refer
to~\cite{lenzing:delapena:1999}. Especially for $T_{\cc}=T_{\can}$, a
so-called canonical configuration, we get a canonical algebra. We
remark that $T_{\cc}$ is in particular a noetherian tilting object in
$\QHh$. It is well-known that $T_{\cc}$ is a (compact) generator of
$\Dd$ inducing an equivalence
$$ \RHom_{\Dd}(T_{\cc},-)\colon\Derived{\Qcoh\XX}
\longrightarrow\Derived{\Mod\Sigma}$$
of triangulated categories (cf.~\cite[Prop.~1.5]{bondal:kapranov:1989}
and \cite[Thm.~8.5]{keller:2007}). Via this equivalence the module
category $\Mod\Sigma$ can be identified (like
in~\cite[Thm.~3.2]{lenzing:delapena:1997} and~\cite{lenzing:1997})
with the full subcategory $\Add(\Tt_{\cc}\vee\Ff_{\cc}[1])$ of $\Dd$,
where $(\Tt_{\cc},\Ff_{\cc})$ is the torsion pair in $\QHh$ given by
$\Tt_{\cc}=\Gen(T_{\cc})=\rperpe{{T_{\cc}}}$ and
$\Ff_{\cc}=\rperpo{{T_{\cc}}}$. This torsion pair induces a split
torsion pair $(\Qq,\Cc)=(\Ff_{\cc}[1],\Tt_{\cc})$ in
$\Mod\Sigma$. Moreover,
$\mod\Sigma=(\Tt_{\cc}\cap\Hh)\vee(\Ff_{\cc}\cap\Hh)[1]$.
\subsection*{Correspondences between tilting objects}
Following~\cite{bazzoni:herbera:2008}, we call a tilting sheaf
$T\in\QHh$ of \emph{finite type} if the tilting class $\rperpe{T}$ is
determined by a class of finitely presented objects $\vSs\subseteq\Hh$
such that $\rperpe{T}=\rperpe{\vSs}$. If $T$ is of finite type, then
$\vSs:=\lperpe{(\rperpe{T})}\cap\Hh$ is the largest such class. We are
now going to see that all tilting sheaves lying in $\Tt_{\cc}$ are of
finite type.

We call  an object $T$ in the triangulated category
$\bDd=\bDerived{\Qcoh\XX}$ a \emph{tilting complex} if the
following two conditions hold.
\begin{enumerate}
\item[(TC1)]
$\Hom_{\Dd}(T,T^{(I)}[n])=0$  for all cardinals $I$ and all $n\in\ZZ$,
$n\neq 0$.\smallskip 
\item[(TC2)] If $X\in\bDd$ satisfies $\Hom_{\Dd}(T,X[n])=0$ for all
  $n\in\ZZ$, then $X=0$.
\end{enumerate}
\begin{proposition}\label{prop:tilting-sheaf-complex-module}
  The following statements are equivalent for $T\in\Tt_{\cc}$ (viewed
  as a complex concentrated in degree zero).
  \begin{enumerate}
  \item[(1)] $T$ is a tilting sheaf in $\QHh$.
  \item[(2)] $T$ is a tilting complex in $\bDd$.
  \item[(3)] $T$ is a tilting module in $\Mod\Sigma$ (of projective
    dimension at most one).  
  \end{enumerate}
  In particular, every tilting sheaf $T\in\QHh$ lying in
  $\Tt_{\cc}$ is of finite type.
\end{proposition}
\begin{proof}
  Clearly (2) implies (1) and (3). We show that (1) implies (2). Since
  $\QHh$ is hereditary, $\Ext^1_{\QHh}(T,T^{(I)})=0$ is equivalent to
  $\Hom_{\Der}(T,T^{(I)}[n])=0$ for all $n\neq 0$. Let
  $X=\bigoplus_{i=-s}^{s}X_i\in\bDd$ be such that $X_i\in\QHh[i]$, and
  assume
  \begin{equation}
    \label{eq:vanishing-exts}
   \Hom_{\Der}(T,X_i[n])=0\ \text{for all}\ n\in\ZZ\ \text{and all}\ i.    
  \end{equation}
  Since $X_i[-i]\in\QHh$, this implies for $n=-i$ and $n=-i+1$
  the condition $$\Hom_{\QHh}(T,X_i
  [-i])=0=\Ext^1_{\QHh}(T,X_i[-i]).$$ By (1) we conclude $X_i[-i]=0$,
  and thus $X_i=0$. Finally, we conclude $X=0$. 

  The proof that (3) implies (2) is similar. We just have to observe
  that condition~\eqref{eq:vanishing-exts} yields
  $\Ext^1_{\QHh}(T,X_i[-i])=0$, that is,
  $X_i[-i]\in\Gen(T)\subseteq{\Tt_{\cc}}$, and thus $X_i$ is, up to
  shift in the derived category, a $\Sigma$-module.

  Assume that $T$ satisfies condition~(1). In order to show that $T$
  is of finite type, we set $\vSs=\lperpe{(\rperpe{T})}\cap\Hh$ and
  verify $\rperpe{\vSs}=\rperpe{T}$. The inclusion
  $\rperpe{\vSs}\supseteq\rperpe{T}$ is trivial. Further, since
  $T\in\Tt_{\cc}$, we have $T_{\cc}\in\vSs$, and thus
  $\rperpe{\vSs}\subseteq\Tt_{\cc}$ consists of $\Sigma$-modules.  We
  view $T$ as a tilting $\Sigma$-module and exploit the corresponding
  result in $\Mod\Sigma$ from~\cite{bazzoni:herbera:2008}. It states
  that the tilting class
  $\rperpe{{T_\Sigma}}=\{X\in\Mod\Sigma\mid\Ext^1_{\Sigma}(T,X)=0\}$
  is determined by a class
  $\widetilde{\vSs}=\lperpe{({\rperpe{{T_\Sigma}}})}\cap\mod\Sigma$ of
  finitely presented modules of projective dimension at most one, that
  is, $\rperpe{{T_\Sigma}}=\rperpe{{\widetilde{\vSs}}}$.
  
  Notice that $\widetilde{\vSs}\subseteq\Tt_{\cc}$.  Otherwise there
  would be an indecomposable $F\in\Ff_{\cc}$ with
  $F[1]\in\widetilde{\vSs}$. Then $\Ext^1_{\QHh}(T,\tau
  F)=\D\Hom_{\QHh}(F,T)=\D\Ext^1_{\Sigma}(F[1],T)=0$, that is, $\tau
  F\in\Gen(T)\subseteq\Tt_{\cc}$, and $\Ext^1_{\QHh}(T_{\cc},\tau
  F)=0$. But also $\Hom_{\QHh}(T_{\cc},\tau F)=
  \D\Ext^1_{\QHh}(F,T_{\cc})=\D\Hom_{\Der}(F[1],T_{\cc}[2])=\D\Ext^2_{\Sigma}(F[1],T_{\cc})=
  0$ since $\pdim_\Sigma F[1]\le 1$, and so $F[1]=0$, a contradiction.

  Now any object $X$ in $\Tt_{\cc}$ can be viewed both in $\Mod\Sigma$
  and ${\QHh}$, and the functors $\Ext^1_{\Sigma}(X,-)$ and
  $\Ext^1_{\QHh}(X,-)$ coincide on $\Tt_{\cc}$. In particular,
  $\widetilde{\vSs}\subseteq{\vSs}$, and if $X$ is a sheaf in
  $\rperpe{\vSs}$, then $X$ is a $\Sigma$-module with
  $\Ext^1_{\Sigma}(S,X)=0$ for all $S\in\widetilde{\vSs}$, hence
  $\Ext^1_{\QHh}(T,X)=\Ext^1_{\Sigma}(T,X)=0$, that is,
  $X\in\rperpe{T}$. This finishes the proof.
\end{proof}
We will construct and classify a certain class of large tilting
sheaves independently from the representation type, even independently
of the genus, namely the tilting sheaves with a large torsion part. A
complete classification of all large tilting sheaves will be obtained
in the domestic and the tubular (that is: in the non-wild) genus zero
cases.\bigskip

The domestic case is akin to the tame hereditary case:
\subsection*{Tame hereditary algebras}
There is a tilting bundle $T_{\cc}$ such that $H=\End(T_{\cc})$ is a
tame hereditary algebra if and only if $\XX$ is of domestic type. In
this case it follows from
Proposition~\ref{prop:tilting-sheaf-complex-module} that the large
tilting $H$-modules (of projective dimension at most one), as
classified in~\cite{angeleri:sanchez:2013}, correspond (up to
equivalence) to the large tilting sheaves in $\Qcoh\XX$. Indeed,
recall that $T_{\cc}$ induces a torsion pair $(\Tt_{\cc},\Ff_{\cc})$
in $\Qcoh\XX$ and a split torsion pair $(\mathcal{Q},\mathcal{C})$ in
$\Mod H$. By~\cite[Thm.~2.7]{angeleri:sanchez:2013} every large tilting
$H$-module lies in the class $\Cc\subseteq\Mod H$, and it will be
shown in Proposition~\ref{prop:tilting-sheaves-finite-type} below that
every large tilting sheaf lies in $\Tt_{\cc}$.

\section{Torsion, torsionfree, and divisible sheaves}
\emph{In this section let $\QHh=\Qcoh\XX$, where $\XX$ is a weighted
  noncommutative regular projective curve over a field $k$.} Our main
aim is to prove that every tilting sheaf splits into a direct sum of
indecomposable sheaves of finite length, Pr\"ufer sheaves, and a
torsionfree sheaf.
\begin{definition}\label{def:tordiv}
  Let $V\subseteq\XX$ be a subset. A quasicoherent sheaf $F$ is called
  \mbox{$V$-\emph{torsionfree}} if $\Hom(S_x,F)=0$ for all $x\in V$
  and all simple sheaves $S_x\in\Uu_x$. In case $V=\XX$ the sheaf $F$
  is \emph{torsionfree}.  We set $$\vSs_V=\coprod_{x\in V}\Uu_x$$ and
  denote by $$\mathcal{F}_V =\rperpo{{\vSs_V}}$$ the class of
  $V$-torsionfree sheaves.

  Similarly, a quasicoherent sheaf $D$ is called $V$-\emph{divisible}
  if $\Ext^1(S_x,D)=0$ for all $x\in V$ and for all simple sheaves
  $S_x\in\Uu_x$. In case $V=\XX$ we call $D$ just \emph{divisible}.
  We denote by
   $$\Dd_V=\rperpe{{\vSs_V}}$$ the class of $V$-divisible sheaves. It
   is closed under direct summands, set-indexed direct sums,
   extensions and epimorphic images. Furthermore, we call $D$
   \emph{precisely $V$-divisible} if $D$ is $V$-divisible, and if
   $\Ext^1(S,D)\neq 0$ for every simple sheaf $S\in\vSs_{\XX\setminus
     V}$.
\end{definition}
\begin{remark}\label{rem:rperp} 
  The class $\vSs_V$ is a Serre subcategory in $\Hh=\fp(\QHh)$, its
  direct limit closure $\Tt_V=\vec{\vSs_V}$ is a localizing
  subcategory in $\QHh$ of finite type, and $(\Tt_V,\Ff_V)$ is a
  hereditary torsion pair in $\QHh$. In particular, the canonical
  quotient functor $\pi\colon\QHh\ra\QHh/\Tt_V$ has a right-adjoint
  $s\colon\QHh/\Tt_V\ra\QHh$ which commutes with direct limits. The
  class of $V$-torsionfree and $V$-divisible sheaves
  \begin{equation}
    \label{eq:rperp=Mod-general}
    \rperp{{\vSs_V}}=\rperp{{\Tt_V}}\simeq\QHh/\Tt_V
  \end{equation}
  is a full exact subcategory of $\QHh$, that is, the inclusion
  functor $j\colon\rperp{{\vSs_V}}\ra\QHh$ is exact and induces an
  isomorphism $\Ext^1_{\rperp{{\vSs_V}}}(A,B)\simeq\Ext^1_{\QHh}(A,B)$
  for all $A,\,B\in\rperp{{\vSs_V}}$. In particular,
  $\Ext^1_{\rperp{{\vSs_V}}}$ is right exact, so that the category
  $\QHh/\Tt_V\simeq\rperp{{\vSs_V}}$ is hereditary.  For details we
  refer to \cite[Prop.~1.1, Prop.~2.2, Cor.~2.4]{geigle:lenzing:1991},
  \cite[Thm.~2.8]{herzog:1997}, \cite[Lem.~2.2, Thm.~2.6, Thm.~2.8,
  Cor.~2.11]{krause:1997}.
\end{remark}
We note that in case $V=\XX$ the subclass $\vSs_\XX=\Hh_0$ of $\Hh$ is
the class of finite length sheaves, $\Tt=\Tt_\XX$ in $\QHh$ forms the
class of torsion sheaves, $\Ff=\Ff_\XX$ the class of torsionfree
sheaves, and $\Ff\cap\Hh=\vect\XX$ the class of vector bundles.
\begin{remark}\label{rem:basic-properties}
  Let $X\in\QHh$. Let $tX$ be the largest subobject of $X$ which lies
  in $\Tt$, the \emph{torsion subsheaf} of $X$. Then the quotient
  $X/tX$ is torsionfree. \emph{The canonical sequence
  $$\eta\colon 0\ra tX\ra X\ra X/tX\ra 0$$ is pure-exact}, that is,
$\Hom(E,\eta)$ is exact for every $E\in\Hh=\fp(\QHh)$. Indeed, we have
$E=E_+\oplus E_0$, where $E_+$ is a vector bundle and $E_0$ is of
finite length. It follows that
$\Ext^1(E,tX)=\Ext^1(E_+,tX)\oplus\Ext^1(E_0,tX)$. The left summand is
zero by Serre duality, since every vector bundle is
torsionfree. Moreover, $\Hom(E_0,X/tX)=0$, so
$\Hom(E,X)\ra\Hom(E,X/tX)$ is surjective.
\end{remark}
\begin{lemma}
  A quasicoherent sheaf is injective if and only if it is divisible.
\end{lemma}
\begin{proof}
  Trivially every injective sheaf is divisible. Conversely, every
  divisible sheaf $Q$ is $L'$-injective for every line bundle $L'$:
  this means that if $L''\subseteq L'$ is a sub line bundle of $L'$,
  then every morphism $f\in\Hom(L'',Q)$ can be extended to
  $L'$. Indeed, there is commutative diagram with exact
  sequences $$\xymatrix{ 0 \ar @{->}[r] & L''\ar @{->}[r] \ar
    @{->}[d]_-{f} & L'\ar @{->}[r] \ar
    @{->}[d] & E \ar @{->}[r] \ar @{=}[d] & 0\\
    0 \ar @{->}[r] & Q\ar @{->}[r] & X\ar @{->}[r] & E\ar @{->}[r] &
    0}$$ with $E$ of finite length. Since $Q$ is divisible, the lower
  sequence splits, and it follows that $f$ lifts to $L'$. This shows
  that $Q$ is $L'$-injective. Since the line bundles form a
  system of generators of $\QHh$, we obtain by the
  version~\cite[V.~Prop.~2.9]{stenstroem:1975} of Baer's criterion
  that $Q$ is injective in $\QHh$.
\end{proof}
\begin{remark}\label{rem:divisible-reduced}
  By the closure properties mentioned above, the class $\Dd$ of
  divisible sheaves is a torsion class.  Given an object $X\in\QHh$,
  we denote by $dX$ the largest divisible subsheaf of $X$. Since $dX$
  is injective, $$X\simeq dX\oplus X/dX.$$ The sheaves with $dX=0$,
  called \emph{reduced}, form the torsion-free class corresponding to
  the torsion class $\Dd$.
\end{remark}
\begin{proposition}\label{prop:main-properties-torsion-sheaves}
{\  }
  \begin{enumerate}
  \item[(1)] The indecomposable injective sheaves are (up to
    isomorphism) the sheaf $\Kk$ of rational functions and the
    Pr\"ufer sheaves $S[\infty]$ ($S\in\Hh$ simple).\smallskip
  \item[(2)] Every torsion sheaf $F$ is of the form 
    \begin{equation}
    \label{eq:torsion-x-parts}
    F=\bigoplus_{x\in\XX}F_x\quad\text{with}\quad
    F_x\in\vec{\Uu_x}\ \text{unique}, 
  \end{equation}
  and there are pure-exact sequences
  \begin{equation}
    \label{eq:x-pure-exact}
    0\ra E_x\ra F_x\ra P_x\ra 0
  \end{equation}
  in $\vec{\Uu_x}$ with $E_x$ a direct sum of indecomposable finite
  length sheaves and $P_x$ a direct sum of Pr\"ufer sheaves (for all
  $x\in\XX$).\smallskip
\item[(3)] Every sheaf of finite length $F$ is $\Sigma$-pure-injective,
  that is, $F^{(I)}$ is pure-injective for every cardinal $I$.
  \end{enumerate}
\end{proposition}
\begin{proof}
  (1) It is well-known that in a locally noetherian category every
  injective object is a direct sum of indecomposable injective
  objects. Every indecomposable injective object has a local
  endomorphism ring and is the injective envelope of each of its
  non-zero subobjects. For details we refer to~\cite{gabriel:1962}.

  Let $E$ be an indecomposable injective sheaf. We consider its
  torsion part  $tE$. If $tE\neq 0$,
  then $E$ has a simple subsheaf $S$. It follows that $E$ is 
  injective envelope of $S$, and thus it contains the direct
  family $S[n]$ $(n\ge 1)$ and its  union $S[\infty]$.
  We claim that $E=S[\infty]$.  Indeed, it is easy to
  see that $S[\infty]$ is uniserial, with each proper subobject of the
  form $S[n]$ for some $n\geq 1$. If there were a simple object $U$ with
  $0\neq\Ext^1(U,S[\infty])=\D\Hom(S[\infty],\tau U)$, then there would be a
  surjective map $S[\infty]\ra\tau U$, whose kernel would have to be a (maximal)
  subobject of $S[\infty]$, hence of the form $S[n]$, which is
  impossible since $S[\infty]$ has infinite length. It follows that
  $S[\infty]$ is divisible, thus injective, and we
  conclude $E=S[\infty]$. 
  
  If, on the other hand, $tE=0$, then $E$ is
  torsionfree and contains a line bundle $L'$ as a subobject. Then $E$
  is the injective envelope of $L'$. In the quotient category
  $\Hh/\Hh_0$ the structure sheaf $L$ and $L'$ become isomorphic
  (\cite{lenzing:reiten:2006}), and thus (by definition of the
  morphism spaces in the quotient category) there is a third line
  bundle $L''$ which maps non-trivially to both, $L'$ and $L$. It
  follows that $L'$ has the same injective envelope as $L$, namely
  $\Kk$.

  (2) The torsion class $\Tt$ is a locally finite Grothendieck
  category with injective cogenerator given by the direct sum of all
  the Pr\"ufer sheaves. We have the coproduct of (locally finite)
  categories
  \begin{equation}
    \label{eq:torsion-coprod}
    \Tt=\coprod_{x\in\XX}\vec{\Uu}_x,
  \end{equation}
  from which we derive~\eqref{eq:torsion-x-parts}.\medskip

  Let $\Uu=\Uu_x$ be a tube of rank $p\geq 1$, with simple objects
  $S,\,\tau S,\dots,\tau^{p-1}S$, and $E$ the injective cogenerator
  of $\vec{\Uu}$ given by $\bigoplus_{j=0}^{p-1}\tau^j S[\infty]$.  Its
  endomorphism ring, $R=\End(E)$ is a pseudo-compact $k$-algebra, such
  that $\vec{\Uu}$ is equivalent to the full subcategory $\Dis(R)$ of
  $\Mod(R)$, given by the discrete topological modules. Indeed,
  by~\cite[IV.4.~Cor.~1]{gabriel:1962} the category $\vec{\Uu}$ is
  dual to $\PC(R^{\op})$, the category of left pseudo-compact
  $R$-modules; note that in~\cite{gabriel:1962} left modules are
  considered, whereas we consider right modules, like
  in~\cite{vandenbergh:2001}. Since
  $\soc(E)=\bigoplus_{i=0}^{p-1}\tau^i S$, we get
  $R/\rad(R)\simeq\End(\soc(E))\simeq D^p$ as $k$-algebras, with
  $D=\End(\tau^i S)$, by~\cite[IV.4.~Prop.~12]{gabriel:1962}. In
  particular, the simple left $R$-modules are finite dimensional. It
  follows that $R^{\op}$ is cofinite in the sense
  of~\cite{vandenbergh:2001}. From~\cite[Prop.~4.10]{vandenbergh:2001}
  we get $\PC(R^{\op})=\Dis(R)^{\op}$. Thus, $\vec{\Uu}$ is dual to
  $\Dis(R)^{\op}$, therefore equivalent to $\Dis(R)$, as
  claimed.\medskip

  Moreover, as in~\cite[p.~373]{ringel:1979} one shows that
  $R\simeq H_p(V,\mathfrak{m})$, given by matrices
  $(a_{ij})\in\matring_p(V)$ with $a_{ij}\in\mathfrak{m}$ for $j>i$;
  here $V=\End(\tau^iS[\infty])$ is a (noncommutative) complete local
  principal ideal domain with maximal ideal $\mathfrak{m}$, so that
  every non-zero one-sided ideal is a power of $\mathfrak{m}$; such
  rings are called complete discrete valuation rings
  in~\cite{ringel:1979}. In particular, $R$ is a semiperfect, bounded
  hereditary noetherian prime ring. It is easily shown that
  $M\in\Mod(R)$ is discrete (with respect to the topology given
  by~\cite[IV.4.~Prop.~13]{gabriel:1962}, which is bounded) if and
  only if each element from $M$ is annihilated by a non-zero ideal in
  $R$, that is, precisely when $M$ is torsion (in the sense
  of~\cite[p.~373]{ringel:1979}). So the objects of $\Dis(R)$ are just
  the torsion $R$-modules.\medskip

  Now, in the terminology of~\cite{singh:1978}, the
  sequence~\eqref{eq:x-pure-exact} expresses that $E_x$ is a
  \emph{basic submodule} of $F_x$, and the existence of such a pure
  submodule is given by~\cite[Thm.~1]{singh:1978}.\medskip

  (3) Each indecomposable $R$-module $F$ of finite length has finite
  endolength, since it is finite dimensional over $k$, by the argument
  from the preceding part. From~\cite[Beisp.~2.6 (1)]{zimmermann:1977}
  we obtain that $F$ is a $\Sigma$-pure-injective $R$-module. Since an
  object $M$ in a locally noetherian category is pure-injective if and
  only if the summation map $M^{(I)}\ra M$ factors through the
  canonical embedding $M^{(I)}\ra M^I$ for every $I$ (we refer
  to~\cite[Thm.~5.4]{prest:2011}), we conclude that $F$ is
  $\Sigma$-pure-injective also in $\QHh$.
\end{proof}
If $F$ is a torsion sheaf like in~\eqref{eq:torsion-x-parts}, we call
the set of those $x\in\XX$ with $F_x\neq 0$ the \emph{support} of
$F$. If the support of $F$ is of the form $\{x\}$, we call $F$
\emph{concentrated} in $x$.
\begin{corollary}\label{prop:basic-subsheaf}
  Let $F\in\QHh$ be a torsion sheaf.
  \begin{enumerate}
  \item[(1)] There is a pure-exact sequence
  \begin{equation}
    \label{eq:basic-subsheaf}
    0\ra E\stackrel{\subseteq}\lra F\ra F/E\ra 0
  \end{equation}
  such that $E$ is a direct sum of finite length sheaves and $F/E$ is
  injective.\smallskip
  \item[(2)] If $F$ has no non-zero direct summand of finite length,
  then $F$ is a direct sum of Pr\"ufer sheaves.
  \item[(3)] If $F$ is a reduced torsion sheaf and 
    $E_1,\dots,E_n$ are the only indecomposable direct summands of $F$
  of finite length, then $F$ is  pure-injective and isomorphic to
  $\bigoplus_{j=1}^n{E_j}^{(I_j)}$ for suitable sets $I_j$.
    \end{enumerate}
\end{corollary}
\begin{proof}
  (1) The direct sum of all pure-exact sequences~\eqref{eq:x-pure-exact}
  ($x\in\XX$) is pure-exact.

  (2) This follows from~(1) by purity. (Locally, in $x$, we can also
  refer to~\cite[Thm.~10]{singh:1975}.)

(3)
  We consider the pure-exact sequence~\eqref{eq:basic-subsheaf}. By
  assumption, $E$ must be of the form $\bigoplus_{j=1}^n{E_j}^{(I_j)}$
  (indeed, since $E$ is pure in $F$,  its direct summands of
  finite length, being  pure-injective, are also direct summands of
  $F$). Now $E$ is, by part~(3) of
  Proposition~\ref{prop:main-properties-torsion-sheaves},
  pure-injective, and thus $F\simeq E\oplus F/E$. Since $F$ is
  reduced, we conclude $F\simeq E$.
\end{proof}
The following basic splitting property will be crucial for our
treatment of large tilting sheaves.
\begin{theorem}\label{thm:torsion-splitting}
  Let $T\in\QHh$ be a sheaf such that $\Ext^1(T,T)=0$ holds.
  \begin{enumerate}
  \item[(1)] The torsion part $tT$ is a direct sum of indecomposable
    sheaves of finite length and Pr\"ufer sheaves. In particular, it
    is pure-injective.\smallskip
  \item[(2)] The canonical exact sequence $0\ra tT\ra T\ra T/tT\ra 0$
    splits.
  \end{enumerate}
\end{theorem}
\begin{proof}
  By Remark~\ref{rem:basic-properties} it suffices to prove
  part~(1). By Remark~\ref{rem:pure-injective} the assertion is true
  in case $tT$ is coherent. If $tT$ does not admit any non-zero
  summand of finite length, then we conclude from
  Corollary~\ref{prop:basic-subsheaf}~(2) that $tT$ is a direct sum
  of Pr\"ufer sheaves, and then $tT$ is in particular
  pure-injective. Let now $E$ be an indecomposable summand of $tT$ of
  finite length. The composition of embeddings $E\ra tT\ra T$ gives a
  surjection $\Ext^1(T,T)\ra\Ext^1(E,T)$, showing that
  $\Ext^1(E,T)=0$. Forming the push-out, the projection $tT\ra E$
  yields the following commutative exact diagram. $$\xymatrix{ 0\ar
    @{->}[r]& tT\ar @{->}[r]\ar @{->>}[d] & T\ar @{->}[r]\ar @{->}[d]&
    T/tT\ar @{->}[r]\ar @{=}[d]& 0 \\ 0\ar @{->}[r]& E\ar @{->}[r]&
    T'\ar @{->}[r]& T/tT\ar @{->}[r]& 0.}$$ Using Serre duality
  $\Ext^1(T/tT,E)=\D\Hom(\tau^- E,T/tT)=0$, the lower sequence splits,
  showing that there is an epimorphism $T\ra E$. This gives a
  surjective map $\Ext^1(E,T)\ra\Ext^1(E,E)$, showing that
  $\Ext^1(E,E)=0$. Therefore $E$ must belong to an exceptional tube of
  some rank $p>1$, and has length $<p$. Thus there are only finitely
  many such $E$. From Corollary~\ref{prop:basic-subsheaf}
  and Remark~\ref{rem:divisible-reduced} we conlude that $tT$ is a
  direct sum of copies of these finitely many indecomposables of
  finite length and of Pr\"ufer sheaves. This proves the theorem.
\end{proof}
 Given a  tilting sheaf $T\in\QHh$, we will
often write $$T=T_+\oplus T_0$$ with $T_0=tT$ the \emph{torsion} and
$T_+\simeq T/tT$ the \emph{torsionfree part} of $T$.
We will say that $T$ has a \emph{large torsion
  part} if $tT$ is large in the sense that
there is no coherent sheaf $E$ such that $\Add(tT)=\Add(E)$.

\section{Tilting sheaves induced by resolving classes}
In this section we introduce the notion of a resolving class, and we
employ it to construct the torsionfree Lukas tilting sheaf $\bL$ and
the tilting sheaves $T_{(B,V)}$. We further classify all tilting
sheaves with large torsion part, and we establish a bijection between
resolving classes and tilting classes of finite type.
\begin{numb}\label{nr:pre-def-weakly-resolving}
  Let $\QHh$ be a locally coherent Grothendieck category with
  $\Hh=\fp(\QHh)$. Let $T$ be a tilting object of finite type in
  $\QHh$, that is,
$$\Bb:=\Gen(T)=\rperpe{T}=\rperpe{\vSs}$$ for some $\vSs\subseteq\Hh$,
which we choose to be the largest class with this
property $$\vSs=\lperpe{\Bb}\cap\Hh.$$ Applying $\Ext^1(S,-)$ to the
sequence
\begin{equation}
  \label{eq:inj-envelope}
  0\ra X\ra E(X)\ra E(X)/X\ra 0
\end{equation}
where $X\in\QHh$ is arbitrary and
$E(X)$ is its injective envelope, we see that
\begin{enumerate}
\item[(o)] $\vSs$  consists of objects $S$ with 
$\pd_{\QHh}(S)\leq 1$.
\end{enumerate} 
 We list further properties of $\vSs$ that can be verified by the reader:
\begin{enumerate}
\item[(i)] $\vSs$ is closed
under extensions; 
\item[(ii)] $\vSs$ is closed under direct summands;
\item[(iii)] $S'\in\vSs$ whenever $0\ra S'\ra S\ra S''\ra 0$ is exact with
$S,\,S''\in\vSs$.
\end{enumerate}
\end{numb}
\begin{definition}
  Let $\QHh$ be a locally coherent Grothendieck category. We call a
  class $\vSs\subseteq\Hh=\fp(\QHh)$ \emph{resolving} if it satisfies
  (i), (ii), (iii), and generates $\QHh$.
\end{definition}
\begin{remark}
  A generating system $\vSs\subseteq\Hh$ is resolving whenever it is
  closed under extensions and subobjects. In case $\QHh=\Qcoh\XX$ the
  converse also holds true; we refer to
  Corollary~\ref{cor:Qcoh-resolving} below.
\end{remark}
\begin{theorem}\label{thm:tilting-from-resolving}
  Let $\QHh$ be locally coherent and $\vSs$ a resolving class such
  that $\pd_{\QHh}(S)\leq 1$ for all $S\in\vSs$. Then there is a
  tilting object $T$ in $\QHh$ with $\rperpe{T}=\rperpe{\vSs}$.
\end{theorem}
\begin{proof}
  The class $\Bb=\rperpe{\vSs}$ is pretorsion, that is, it is closed
  under direct sums (recall that $\vSs \subseteq\Hh$ consists of
  finitely presented objects) and epimorphic images (here we need the
  assumption on the projective dimension). Further, it is special
  preenveloping by~\cite[Cor.~2.15]{saorin:stovicek:2011}. By
  assumption, $\vSs$ contains a system of generators $(G_i, i\in I)$
  for $\QHh$. Set $G=\bigoplus_{i\in I}G_i$, and let
  \begin{equation}
    \label{eq:preenvelope-0}
    0\ra G\ra T_0\ra T_1\ra 0
  \end{equation}
  be a special $\Bb$-preenvelope of $G$ with $T_0\in\Bb$ and
  $T_1\in\lperpe{\Bb}$. We claim that $T=T_0\oplus T_1$ is the desired
  tilting object. Since $\Bb$ is pretorsion we have
  $\Gen(T)\subseteq\Bb$.  Moreover, for every $X\in\QHh$
  there is a natural isomorphism
  \begin{equation}
    \label{eq:Ext1-coprod-prod}
    \Ext^1\Bigl(\bigoplus_{i\in I}G_i,X\Bigr)\simeq\prod_{i\in I}\Ext^1(G_i,X).
  \end{equation}
  (This we get from the natural isomorphism
  $\Hom(\bigoplus_{i\in I}G_i,X)\simeq\prod_{i\in I}\Hom(G_i,X)$ by
  applying $\Hom(G_i,-)$ and $\Hom(\bigoplus_{i\in I}G_i,-)$ to the
  exact sequence~\eqref{eq:inj-envelope}.)  Since $G_i\in\vSs$ for all
  $i\in I$, we deduce
  \begin{equation}
    \label{eq:Ext-Lambda-X=0}
    \Ext^1(G,X)=0\quad\text{for all}\ X\in\Bb.
  \end{equation}
  Hence $G\in\lperpe{\Bb}$, and (\ref{eq:preenvelope-0}) shows that
  $T_0$ and $T$ belong to $\lperpe{\Bb}$ as
  well. So $$\Gen(T)\subseteq\Bb\subseteq\rperpe{T}.$$ Let now
  $X\in\rperpe{T}$. Since $G$ is a generator, there is an epimorphism
  $G^{(J)}\ra X$ and a commutative exact diagram
  $$\xymatrix{ 0\ar @{->}[r]& G^{(J)}\ar @{->}[r]\ar @{->}[d] &
    (T_0)^{(J)}\ar @{->}[r]\ar @{->}[d]& (T_1)^{(J)}\ar @{->}[r]\ar
    @{=}[d]& 0 \\ 0\ar @{->}[r]& X\ar @{->}[r]& X'\ar @{->}[r]&
    (T_1)^{(J)}\ar @{->}[r]& 0.}$$ Since $X\in\rperpe{{T_1}}$ and thus
  by~(\ref{eq:Ext1-coprod-prod}) also
  $X\in\rperpe{({T_1}^{(J)})}$, the lower sequence splits. Therefore
  we get an epimorphism ${T_0}^{(J)}\ra X$, showing that
  $X\in\Gen(T)$.  We conclude that $T$ is a tilting object with
  $\Gen(T)=\Bb$.
\end{proof}
\subsection*{Applications}
\emph{Let now $\QHh=\Qcoh\XX$, where $\XX$ is a weighted
  noncommutative regular projective curve over a field $k$.} We
exhibit two applications of the theorem. The first one is quite easy.
\begin{proposition}\label{prop:large-tilting-torsionfree-in-general}
  Let $\QHh=\Qcoh\XX$, where $\XX$ is a weighted noncommutative
  regular projective curve. There is a torsionfree large tilting sheaf
  $\bL$, called \emph{Lukas tilting sheaf}, such that
  $\rperpe{\bL}=\rperpe{(\vect\XX)}$. \end{proposition}
\begin{proof}
  The class $\vSs=\vect\XX$ is  resolving. By
  Theorem~\ref{thm:tilting-from-resolving} there is a tilting sheaf
  $\bL$ with $\rperpe{(\vect\XX)}=\rperpe{\bL}$. We show that $\bL$ is
  torsionfree. Assume that $\bL$ has a non-zero torsion part $T_0$. By
  Theorem~\ref{thm:torsion-splitting} this is a direct summand of
  $\bL$. Then $$\rperpe{(\vect\XX)}=
  \rperpe{{\bL}}\subseteq\rperpe{{T_0}}\cap
  \rperpe{(\vect\XX)}\subsetneq\rperpe{(\vect\XX)},$$ where the last
  inclusion is proper because there exists a simple sheaf
  $S$ with $\Hom(S,T_0)\neq 0$ and thus $\tau
  S\in\rperpe{(\vect\XX)}\setminus\rperpe{{T_0}}$. Thus we get a
  contradiction. We conclude that $T_0=0$. Clearly, $\bL$ is then also
  large.
\end{proof}
We record the following observation for later reference.
\begin{lemma}\label{lem:V-divisibility-lemma}
  $\rperpe{\bL}$ contains the class $\mathcal{D}_V$ of $V$-divisible
  sheaves for any $\emptyset\neq V\subseteq\XX$.
  \end{lemma}
\begin{proof}
  With the notation of Definition~\ref{def:tordiv}, we have
  $\rperpe{{\vSs_V}}=\lperpo{{\vSs_V}}$ and
  $\rperpe{(\vect\XX)}=\lperpo{\vect\XX}$ by Serre duality. Let $F$ be
  a sheaf such that there is a non-zero morphism to a vector bundle,
  and consequently also to a line bundle. Since every non-zero
  subsheaf of a line bundle is a line bundle again, there is even an
  epimorphism from $F$ to a line bundle. This line bundle maps onto a
  simple sheaf concentrated in $x\in V$. We conclude that $F$ is not
  $V$-divisible.
\end{proof}
The second application is the classification of all tilting sheaves
having a large torsion part. We first introduce some terminology. 
\begin{numb}\label{wings} Let
$\Uu=\Uu_x$ be a tube of rank $p>1$. We recall that an indecomposable
sheaf $E\in\Uu$ is exceptional (that is, $\Ext^1(E,E)=0$) if and only
if its length is $\leq p-1$; in particular, there are only finitely
many such $E$. If $E$ is exceptional in $\Uu$, then we call the
collection $\Ww$ of all the subquotients of $E$ the \emph{wing rooted
  in} $E$. The set of all simple sheaves in $\Ww$ is called the
\emph{basis} of $\Ww$. It is of the form
$S,\,\tau^- S,\dots,\tau^{-(r-1)}S$ for an exceptional simple sheaf
$S$ and an integer $r$ with $1\leq r\leq p-1$ which equals the length
of the \emph{root} $E$; we call such a set of simples a \emph{segment}
in $\Uu$, and we say that two wings (or segments) in $\Uu$ are
\emph{non-adjacent} if the segments of their bases (or the segments)
are disjoint and their union consists of $<p$ simples and is not a
segment~\cite[Ch.~3]{lenzing:meltzer:1996}.\medskip

We remark that the full subcategory $\add\Ww$ of $\Hh$ is equivalent
to the category of finite-dimensional representations of the linearly
oriented Dynkin quiver $\vec{\AAA}_{r}$. It is well-known that any
tilting object $B$ in the category $\add\Ww$ has precisely $r$
non-isomorphic indecomposable summands $B_1,\dots,B_r$ forming a
so-called \emph{connected branch} in $\Ww$: one of the $B_i$ is
isomorphic to the root $E$, and for every $j$ the wing rooted in $B_j$
contains precisely $\ell_j$ indecomposable summands of $B$, where
$\ell_j$ is the length of $B_j$. In particular, for every $j$ we have
the (full) \emph{subbranch} rooted in $B_j$; if $B_j$ is different
from the root of $\Ww$, we call this subbranch \emph{proper}.\medskip

Following~\cite[Ch.~3]{lenzing:meltzer:1996}, we call a sheaf $B$ of
finite length a \emph{branch sheaf} if it is a multiplicity free
direct sum of connected branches in pairwise non-adjacent wings; it
then follows that $\Ext^1(B,B)=0$.  Clearly, there are only finitely
many isomorphism classes of branch sheaves.
\end{numb}

\begin{theorem}\label{thm:large-tilting-sheaves-in-general}
  Let $\QHh=\Qcoh\XX$, where $\XX$ is a weighted noncommutative
  regular projective curve.
  \begin{enumerate}
  \item[(1)] Let $\emptyset\neq V\subseteq\XX$ and $B\in\Hh_0$ be a
    branch sheaf. There is, up to equivalence, a unique large tilting
    sheaf $T=T_+\oplus T_0$ whose torsion part is given by
    \begin{equation}
      \label{eq:large-torsion-part}
      T_0=B\oplus\bigoplus_{x\in V}\bigoplus_{j\in \Rr_x}\tau^j
      S_x[\infty],
    \end{equation}
    and whose torsionfree part $T_+$ is $V$-divisible; the non-empty
    sets $\Rr_x\subseteq\{0,\dots,p(x)-1)\}$ are uniquely determined
    by $B$, see~\eqref{eq:Def_R_x}.\smallskip
  \item[(2)] Every tilting sheaf with large torsion part is, up to
    equivalence, as in~(1).
  \end{enumerate}
\end{theorem}
For the proof we need several preparations. We start by describing the
torsion part of a tilting sheaf.
\begin{lemma}\label{lem:exceptional-tube}
  Let $T$ be a tilting sheaf and $x$ an exceptional point of weight
  $p=p(x)>1$ such that $(tT)_x\neq 0$. 
  There are two possible cases:
  \begin{enumerate}
  \item[(1)] $(tT)_x$ contains no Pr\"ufer sheaf, but at most $p-1$
    indecomposable summands of finite length, which are arranged in
    connected branches in pairwise non-adjacent wings.\smallskip
  \item[(2)] $(tT)_x$ contains precisely $s$ Pr\"ufer sheaves, where  $1\leq
    s\leq p$, and precisely $p-s$ indecomposable summands of finite
    length. The latter lie in wings of the following form: if
    $S[\infty]$, $\tau^{-r}S[\infty]$ are summands of $T$ with $2\leq
    r\leq p$, but the Pr\"ufer sheaves
    $\tau^{-}S[\infty],\dots,\tau^{-{(r-1)}}S[\infty]$ in between are
    not, then there is a (unique) connected branch in the wing $\Ww$ rooted in
    $S[r-1]$ 
    that occurs as a summand of $T$.
  \end{enumerate}
\end{lemma}
\begin{proof}
  Given a simple object $S\in\Uu_x$, the corresponding Pr\"ufer sheaf
  $S[\infty]$ is $S[p]$-filtered, and thus
  by~\cite[Prop.~2.12]{saorin:stovicek:2011} we have
  \begin{equation}
    \label{eq:Pruefer-criterion}
    S[\infty]\ \text{is a
      summand of}\ T\;\Leftrightarrow\;\lperpe{(\rperpe{T})}\ \text{contains the ray}\ 
    \{S[n]\mid\,n\geq 1\}.  
  \end{equation}
  If no such ray exists, then $(tT)_x$ has at least one indecomposable
  summand of finite length, and it is well-known that all such
  summands are arranged in branches in pairwise non-adjacent wings,
  compare~\cite[Ch.~3]{lenzing:meltzer:1996}.\medskip

  Assume now that, say, $S[\infty]$ and $\tau^{-r} S[\infty]$ are
  summands of $T$, but no Pr\"ufer sheaf ``in between'' is a summand,
  where $2\leq r\leq p$ (when $r=p$, there is precisely one Pr\"ufer
  summand). We show that $S[r-1]$ is a summand of $T$. By
  (\ref{eq:Pruefer-criterion}) this is equivalent to show
  $\Ext^1(T,S[r-1])=0$. If this is not the case, then
  $\Hom(\tau^-S[r-1],T)\neq 0$, and thus there exists an
  indecomposable summand $E$ of $T$ lying on a ray starting in $\tau^-
  S[r-1],\dots,\tau^{-(r-2)}S[2]$ or $\tau^{-(r-1)}S$. But for such
  an $E$ we have
  $0\neq\D\Hom(\tau^-E,\tau^{-r}S[\infty])=\Ext^1(\tau^{-r}S[\infty],E)$,
  contradicting the fact that $T$ has no self-extension. Thus $S[r-1]$
  is a direct summand of $T$. The latter argument also shows that
  every indecomposable summand of $T$ of finite length and lying on a
  ray starting in $S,\tau S,\dots,\tau^{-(r-1)}S$ actually lies in the
  wing $\Ww$ rooted in $S[r-1]$.\medskip

  We claim that the direct sum $B$ of all indecomposable summands of
  $T$ lying in $\Ww$ forms a tilting object in $\add\Ww$. We have
  $\Ext^1(B,B)=0$. Assume that $B$ is not a tilting object in
  $\Ww$. Then there is an indecomposable $E\in\Ww$, not a direct
  summand of $B$, such that $\Ext^1(E\oplus B,E\oplus B)=0$. Let $E'$
  be the indecomposable quotient of $S[r-1]$ such that $E$ embedds
  into $E'$. We have a short exact sequence
  $0\ra F\ra S[r-1]\ra E'\ra 0$ with indecomposable $F\in\Ww$. Let
  $T_+$ be the torsionfree part of $T$. Then exactness of
  $0=\Hom(F,T_+)\ra\Ext^1(E',T_+)\ra\Ext^1(S[r-1],T_+)=0$ shows
  $\Ext^1(E',T_+)=0$, and then also $\Ext^1(E,T_+)=0$. Moreover
  $\Ext^1(T_+,E)=\D\Hom(\tau^- E,T_+)=0$, and since $E\in\Ww$, there
  are no extensions between $E$ and Pr\"ufer summands of $T$. We
  conclude that $E\in\rperpe{T}\cap\lperpe{(\rperpe{T})}=\Add(T)$, a
  contradiction. Thus $B$ is tilting, and it forms a connected
  branch.\medskip

  Doing this with every ``gap'' between Pr\"ufer sheaves in $(tT)_x$,
  one sees that $(tT)_x$ contains precisely $p-s$ indecomposable
  summands of finite length.
\end{proof}
\begin{lemma}\label{lem:T+-divisibility-lemma}
  In  the preceding lemma, the torsionfree part $T_+$ of $T$ belongs to $\Ww^{\perp_1}$
  for every wing $\Ww$ occurring in (1) or (2), and it is even
   $x$-divisible in case (2).
\end{lemma}
\begin{proof}
  The first part of the statement is shown as in the preceding proof.
  In case (2) it then remains to check that $T_+$ has no extensions
  with the simple objects in $\Uu_x$ which do not belong to the wings
  defined by the Pr\"ufer summands of $T$. Let $\Ww$ be such wing and
  $E$ such simple object, that is, $E\not\in\Ww$, but $\tau
  E\in\Ww$.
  Assume $0\neq \Ext^1(E,T_+)\simeq\D\Hom(T_+,\tau E)$. Since
  $\Hom(T_+,\tau\Ww)=0$, repeated application of the almost split
  property yields an indecomposable object $U$ on the ray starting in
  $S$ such that $\Hom(T_+,\tau U)\neq 0$. By Serre duality
  $\Ext^1(U,T_+)\neq 0$, and since $U$ embeds in $S[\infty]$, also
  $\Ext^1(S[\infty],T_+)\neq 0$, a contradiction.
\end{proof}
The branch sheaves and the Pr\"ufer sheaves occurring in the torsion
part of a tilting sheaf are interrelated.  In the situation of
Lemma~\ref{lem:exceptional-tube}~(2), we denote by $\Rr_x$ the set of
cardinality $s$ of all $j\in\{0,\dots,p(x)-1\}$ such that the Pr\"ufer
sheaf $\tau^j S[\infty]$ is a direct summand of $T$. Each such set
defines a unique collection
$$\Ww=\rperpe{\{\tau^j S[\infty]\mid j\in \Rr_x\}}\cap\Uu_x$$ of
pairwise non-adjacent wings in the exceptional tube $\Uu_x$, whereas
the branch $B$, viewed as collection of indecomposable sheaves, is
given as $$B=\Add(T)\cap\Uu_x.$$ In particular, this shows that a
tilting sheaf $T'$ with a different branch $B'\neq B$ in $\Uu_x$ will
have $\rperpe{{T'}}\neq\rperpe{T}$, that is, $T$ and $T'$ cannot be
equivalent.

Conversely, every branch sheaf in $\Uu_x$ -- which we will often
identify with the set of its indecomposable summands -- defines a
unique collection $\Ww$ of pairwise non-adjacent wings in $\Uu_x$, and
this defines uniquely the set $\Rr_x$; namely
\begin{equation}
    \label{eq:Def_R_x}
    \Rr_x=\{j=0,\dots,p(x)-1\mid \tau^{j+1}S\not\in\Ww\}.
\end{equation}
We now consider a pair $(B,V)$ given by a branch sheaf $B\in\Hh$ and a
subset $V\subseteq\XX$, and we associate to it a resolving class. For the
moment $V=\emptyset$ is permitted. In case $V\neq\emptyset$, the
corresponding tilting sheaf will have the properties required by
Theorem~\ref{thm:large-tilting-sheaves-in-general}.

\medskip

Let us fix some notation. The branch sheaf $B$ decomposes into
$B=\bigoplus_{x\in\XX}B_x$; of course $B_x\neq 0$ only if $x$ is one
of the finitely many exceptional points $x_1,\dots,x_t$. For every $x$
denote by $\Ww_x$ the collection of pairwise non-adjacent wings in
$\Uu_x$ defined by $B_x$, and for every $x\in V$ let $\Rr_x$ be the
associated non-empty subset of $\{0,\dots,p(x)-1\}$ defined
by~\eqref{eq:Def_R_x}.  We will write
$$B=B_{\mathfrak i}\oplus B_{\mathfrak e}$$ where $B_{\mathfrak e}$ is
supported in $\XX\setminus V$ and $B_{\mathfrak i}$ in $V$, and we
will say that $B_{\mathfrak e}$ is \emph{exterior} and
$B_{\mathfrak i}$ is \emph{interior} (with respect to $V$).

Recall that each $B_x$ is a direct sum of connected branches in
pairwise non-adjacent wings.  For a connected branch $C$ with
associated wing $\Ww_C$, we call the set
$$C^{>}:=
\begin{cases}
  \rperpo{C}\cap\Ww_C & \text{if}\ C\ \text{is interior},\\
  \rperpo{C}\cap\tau\Ww_C & \text{if}\ C\ \text{is exterior},
\end{cases}$$
the \emph{undercut} of $C$. The undercut $B^{>}$ of $B$ is the union
of the undercuts of all its connected branch components. For an
example, we refer to~\ref{examp:tube-11}.
\begin{lemma}\label{lemma:undercut}
  Let $V\subseteq\XX$ and $B=B_{\mathfrak i}\oplus B_{\mathfrak e}$ be
  a branch sheaf.
  \begin{enumerate}
  \item[(1)] With the notation above, the class
    \begin{equation}
      \label{eq:undercut-formula}
      \vSs=\add\Bigl(\,\vect\XX\cup\tau^- (B^{>})\cup\bigcup_{x\in
        V}\{\tau^j S_x[n]\mid j\in \Rr_x,\,n\in\NN\}\,\Bigr)
    \end{equation}
    is resolving. \smallskip
  \item[(2)] If $T$ is a tilting sheaf with
    $\rperpe{T}=\rperpe{\vSs}$, then
    $\vSs=\lperpe{(\rperpe{T})}\cap\Hh$, the torsionfree part $T_+$ is
    $V$-divisible, and the torsion part is given
    by
    $$T_0=B\oplus\bigoplus_{x\in V}\bigoplus_{j\in \Rr_x}\tau^j
    S_x[\infty].$$
  \end{enumerate}
\end{lemma}
\begin{proof}
  (1) The class $\vSs$ is clearly closed under subobjects.  A simple
  case by case analysis shows that $\vSs$ is also closed under
  extensions. For instance, if $0\ra A\ra E\ra C\ra 0$ is a short
  exact sequence with $A$ a vector bundle and $C\in\vSs$
  indecomposable of finite length, then $E=E_+\oplus E_0$, with $E_+$
  a vector bundle and $E_0$ of finite length; it follows that $E_0$ is
  isomorphic to a subobject of $C$, and thus $E_0\in\vSs$, and then
  $E\in\vSs$.  Compare also \cite[p.~36 from line -19]{angeleri:sanchez:2013}.  
 Since $\vSs$ contains the system of
  generators $\vect\XX$, we conclude that  it is  resolving.

  (2) 
  By Serre duality, an indecomposable coherent sheaf $E\in\Hh$ belongs
  to $\lperpe{(\rperpe{T})}$ if and only if
  $\tau
  E\in\rperpo{(\rperpe{T})}=\rperpo{(\rperpe{\vSs})}=\rperpo{(\lperpo{\tau\vSs})}$.
  We claim that this is further equivalent to $\tau E\in\tau\vSs$,
  that is, $E\in\vSs$.  Indeed, the claim is shown by arguing inside
  the abelian category $\Hh$ as
  in~\cite[Lem.~1.3]{reiten:ringel:2006}, keeping in mind that
  $\tau\vSs$ is closed under subobjects and extensions by part~(1).
  
  We thus have $\vSs=\lperpe{(\rperpe{T})}\cap\Hh$.  It follows
  from~\eqref{eq:Pruefer-criterion} that $T$ has precisely the
  Pr\"ufer summands $\tau^j S_x[\infty]$ with $x\in V$ and
  $j\in\Rr_x$. In particular, $T_+$ is $V$-divisible by
  Lemma~\ref{lem:T+-divisibility-lemma}.  Furthermore,
  \begin{equation}
    \label{eq:vSs-cap-rperp-vSs}
    \rperpe{\vSs}\cap\vSs=\rperpe{T}\cap\lperpe{(\rperpe{T})}\cap{\Hh}=\Add(T)\cap\Hh,
  \end{equation}
  and we now show that  this class further coincides with $\add(B)$.

  Let $\Ww$ be the union of non-adjacent wings associated to $B$, and
  let $B_1$ and $B_2$ be two indecomposable summands of $B$. Then
  $0=\Ext^1(B_1,B_2)=\D\Hom(B_2,\tau B_1)$. Thus
  $\tau B_1\in\rperpo{B}$.  If $B_1$ is either exterior, or interior
  with $\tau B_1\in\Ww$, then $\tau B_1\in B^{>}$, that is,
  $B_1 \in\tau^- (B^{>})\subseteq\vSs$. If, on the other hand, $B_1$
  is interior with $\tau B_1\not\in\Ww$, then $B_1\in\vSs$ by
  definition of $\Rr_x$. Moreover, we have
  $\Ext^1(\tau^-(B^{>}),B_1)=\D\Hom(B_1,B^{>})=0$, and then
  $\Ext^1 (\tau^j S_x[n],B_1)=\D\Hom(B_1,\tau^{j+1}S_x[n])=0$, for any
  $x\in V$ and $j\in\Rr_x$, shows that $B_1\in\rperpe{\vSs}$.
  
  Conversely, let $E\in\vSs\cap\rperpe{\vSs}$ be indecomposable.
  By~\eqref{eq:vSs-cap-rperp-vSs} we have that $E$ is a summand of
  $T$, in particular $E$ is exceptional and belongs to an exceptional
  tube. If $E$ is supported in $V$, then it is a summand of
  $B_{\mathfrak i}$ by Lemma~\ref{lem:exceptional-tube} and the fact
  that the connected parts of $B$ form tilting objects in the
  corresponding wings. If $E$ is not supported in $V$, then it belongs
  to $\tau^-(C^>)$ for a connected branch component $C$ of
  $B_{\mathfrak e}$. Since $\tau^-(C^>)=\lperpe{C}\cap\Ww_C$ where
  $\Ww_C$ is the wing associated to $C$, we infer again that $E$ is a
  summand of $B_{\mathfrak e}$.
  
  We  conclude that $T_0$ is given by
  $B\oplus\bigoplus_{x\in V}\bigoplus_{j\in \Rr_x}\tau^j S_x[\infty]$, as desired.
\end{proof}

We are now ready for the

\begin{proof}[Proof of Theorem~\ref{thm:large-tilting-sheaves-in-general}]
  By the preceding lemma there exists a (large) tilting sheaf with the
  claimed properties. 
  
  Let now $T=T_+\oplus T_0$ be any tilting sheaf with a non-coherent
  torsion part $T_0$. From Lemma~\ref{lem:exceptional-tube} we infer
  that $T_0$ is of the form
  $B\oplus\bigoplus_{x\in V}\bigoplus_{j\in \Rr_x}\tau^j S_x[\infty]$
  for some branch sheaf $B$, some non-empty subset $V\subseteq\XX$ and
  some sets $\Rr_x$.  It is sufficient to show that the class $\vSs$
  from~\eqref{eq:undercut-formula} satisfies
  $\vSs=\lperpe{(\rperpe{T})}\cap\Hh$, since this will imply
  $\rperpe{T}=\rperpe{\vSs}$, as desired.

  By Lemma~\ref{lem:T+-divisibility-lemma} the torsionfree part $T_+$
  of $T$ is $V$-divisible.  From Lemma~\ref{lem:V-divisibility-lemma}
  we infer $T_+\in\rperpe{(\vect\XX)}$. Since also
  $T_0 \in\rperpe{(\vect\XX)}$ by Serre duality, we conclude
  $\Ext^1(X,T)=0$ for any vector bundle $X$, hence
  $\vect\XX\subseteq\lperpe{(\rperpe{T})}$.


  Next, we show $\tau^-(B^{>})\subseteq\lperpe{(\rperpe{T})}$. If
  $E\in\tau^-({B_{\mathfrak i}}^{>})$, then
  $\Ext^1(E,B)=\D\Hom(B,\tau E)=0$ by definition of the
  undercut. Since $T_+$ and the Pr\"ufer sheaves are $V$-divisible, we
  get $\Ext^1(E,T)=0$ and $E\in \lperpe{(\rperpe{T})}$.  If
  $E\in\tau^-({B_{\mathfrak e}}^{>})$, then it belongs to
  $\tau^-(C^>)=\lperpe{C}\,\cap\,\Ww_C$ for a connected branch
  component $C$ of $B_{\mathfrak e}$ with associated wing $\Ww_C$. It
  follows $\Ext^1(E,B)=\D\Hom(B,\tau E)=0$, and $\Ext^1(E,T_+)=0$ by
  Lemma~\ref{lem:T+-divisibility-lemma}, so again
  $E\in \lperpe{(\rperpe{T})}$.

  Finally, if $E$ belongs to a ray $\{\tau^jS_x[n]\mid n\ge 1\}$ with $x\in V$ and $j\in
  \Rr_x$, then $E\in\lperpe{(\rperpe{T})}$ by~\eqref{eq:Pruefer-criterion}. 
  
  Altogether we have shown
  $\vSs\subseteq\lperpe{(\rperpe{T})}\cap\Hh$.
In order to prove the reverse inclusion, let $E\in\Hh$ be
  indecomposable with $E\in\lperpe{(\rperpe{T})}$.
  By definition of   $\vSs$, we can assume that $E$ is of
  finite length, and
  further, if 
  concentrated in a point $x\in V$,  that it  has the form
  $\tau^j S_x[n]$ with $j\not\in\Rr_x$. This means   $\tau^j S_x\in\tau^-\Ww$ 
  by~\eqref{eq:Def_R_x}, so  there is a  connected branch
  component $C$ of $B_{\mathfrak i}$  with associated wing $\Ww_C$
  such that $\tau^j S_x\in\tau^-\Ww_C$. Since $C$ is a summand of $T$,
  we have $E\in\lperpe{C}\cap\tau^-\Ww_C=
  \tau^-({C}^{>})\subseteq\vSs$.
  
  It remains to check the case when $E$ is concentrated in a point
  $x\not\in V$.  Notice that $\Hom(T,\tau E)\simeq\D\Ext^1(E,T)=0$
  implies $\Ext^1(T,\tau E)\neq 0$ by condition (TS2). But the latter
  amounts to $\Ext^1(B_{\mathfrak e},\tau E)\neq 0$, or equivalently,
  $\Hom(E,B_{\mathfrak e})\neq 0$. Let
  $0\neq f: E\to B_{\mathfrak e}$. If $E$ is simple, $f$ is a
  monomorphism, and $E\in\vSs$ because
  $B_{\mathfrak e}\in\tau^- ({B_{\mathfrak e}}^{>})\subseteq\vSs$ and
  $\vSs$ is closed under subjects.  If $E$ has length $\ell>1$, we
  consider the short exact sequence
  $0\to \Ker f\to E\to \text{Im}\,f\to 0$ where $\text{Im}\, f$
  belongs to $\vSs\subseteq \lperpe{(\rperpe{T})}$ and
  $\Ker f\in\lperpe{(\rperpe{T})}$. Proceeding by induction on $\ell$
  and using that $\vSs$ is closed under extensions, we conclude that
  $E\in\vSs$, which completes the proof.
\end{proof}
\begin{corollary} \label{cor:nctp} 
  Let $\QHh=\Qcoh\XX$ with $\XX$ a
  weighted noncommutative regular projective curve.  There is a
  bijection between the equivalence classes of tilting sheaves in
  $\QHh$ having a large torsion part, and the set of pairs $(B,V)$
  given by a branch sheaf $B\in\Hh$ and a subset
  $\emptyset\neq V\subseteq\XX$.  \qed
\end{corollary}
\begin{notation}
  Let $\emptyset\neq V\subseteq\XX$ and $B=B_{\mathfrak i}\oplus
  B_{\mathfrak e}$ be a branch sheaf with interior and exterior part with respect
  to $V$ given by $B_{\mathfrak i}$ and $B_{\mathfrak e}$, respectively. The large tilting
  sheaf constructed above will be denoted by
  \begin{equation}
  \label{eq:def_full-T_V,B}
  T_{(B,V)}=T_{(B_{\mathfrak i},V)}\oplus B_{\mathfrak e}.
\end{equation}
\end{notation}
\medskip

\subsection*{A correspondence.} Next, we establish an analogue of
\cite[Thm.~2.2]{angeleri:herbera:trlifaj:2006} stating that the
resolving subclasses of $\Hh$ correspond bijectively to tilting classes
of finite type. As we will see below, in the domestic and in the
tubular cases every tilting class is of finite type.
\begin{theorem}\label{thm:class-correspondence}
  Let $\XX$ be a weighted noncommutative regular projective curve and
  $\QHh=\Qcoh\XX$. The assignments
  $\Phi\colon\vSs\mapsto\rperpe{\vSs}$ and
  $\Psi\colon\Bb\mapsto\lperpe{\Bb}\cap\Hh$ define mutually inverse
  bijections between
  \begin{itemize}
  \item   resolving classes $\vSs$ in $\Hh$, and
  \item   tilting classes $\Bb=\rperpe{T}$ with
    $T\in\QHh$ tilting of finite type.
  \end{itemize}
\end{theorem}
For the proof of the Theorem, we need the following observations.
\begin{remark}
  In the situation of Lemma~\ref{lem:exceptional-tube}~(2), the right
  perpendicular category $\rperp{\Ww}$ of a wing $\Ww$ rooted in
  $S[r-1]$ coincides with the right perpendicular category to its
  basis $S,\tau^-S,\dots,\tau^{-(r-2)}S$. If $B$ forms a (connected)
  branch in $\Ww$, then also $\rperp{B}=\rperp{\Ww}$, and when forming
  this perpendicular category, the $r$ rays starting in the simple
  objects $S,\tau^-S,\dots,\tau^{-(r-2)}S,\tau^{-(r-1)}S$ and the
  corresponding Pr\"ufer sheaves are turned into a single ray
  $\tau^{-(r-1)}S[rn]$, $n\geq 1$, and a single Pr\"ufer sheaf
  $S[\infty]$.
\end{remark}
\begin{lemma}[Perpendicular Lemma]\label{lem:exterior-branch-reduction}
  Let $B\in\Hh$ be a branch sheaf. Let $T\in\QHh$ be a sheaf such that
  $T\in\rperp{B}$.
    \begin{enumerate}
    \item[(1)] We have $\rperp{B}\simeq\Qcoh\XX'$, where $\XX'$ is a
      noncommutative regular projective curve $\XX'$ with reduced
      weights $1\leq p'_i\leq p_i$.\smallskip
    \item[(2)] $T\oplus B$ is a (large) tilting sheaf in $\QHh$ if and
      only if $T$ is a (large) tilting sheaf in $\QHh'=\Qcoh\XX'$.
    \end{enumerate}
\end{lemma}
\begin{proof}
  (1) This follows from the preceding remark.

  (2) It is clear that $T\oplus B$ satisfies~(TS1) if and only if so
  does $T$. We assume that $T\oplus B$ satisfies~(TS2). Let
  $X\in\QHh'$ such that $\Hom(T,X)=0=\Ext^1(T,X)$. Since
  $\QHh'=\rperp{B}$ we get $\Hom(T\oplus B,X)=0=\Ext^1(T\oplus B,X)$,
  and hence $X=0$ follows, and $T$ satisfies~(TS2). Conversely, let
  $T$ satisfy~(TS2). Let $X\in\QHh$ with $\Hom(T\oplus
  B,X)=0=\Ext^1(T\oplus B,X)$. Then in particular
  $X\in\rperp{B}=\QHh'$, and also $\Hom(T,X)=0=\Ext^1(T,X)$. Then
  $X=0$, so that $T\oplus B$ satisfies~(TS2).
\end{proof}

\begin{proof}[Proof of Theorem~\ref{thm:class-correspondence}]
  $\Phi(\vSs)=\rperpe{\vSs}$ defines a map between the named sets by
  Theorem~\ref{thm:tilting-from-resolving}. By the discussion
  in~\ref{nr:pre-def-weakly-resolving} we see that
  $\vSs:=\Psi(\Bb)=\lperpe{\Bb}\cap\Hh$ satisfies conditions~(i), (ii)
  and~(iii) for resolving. Notice that $\vSs$ is even closed under
  subobjects since $\Qcoh\XX$ is hereditary. We show that $\vSs$ also
  generates $\QHh$. \medskip

  First we show that $\vSs$ contains a non-zero vector bundle. Let
  $\vSs'\subseteq\Hh$ with $\Bb=\rperpe{{\vSs'}}$. Then
  \begin{equation}
    \label{eq:S-prime-in-S}
    \vSs'\subseteq\lperpe{(\rperpe{{\vSs'}})}\cap\Hh=\vSs.
  \end{equation}
  We assume that $\vSs$ does not contain any non-zero vector bundle,
  which we will lead to a contradiction. Then
  $\vSs'\subseteq\Hh_0$. Let $T$ be tilting with
  $\Bb=\rperpe{T}$. Since a coherent $X$ lies in $\lperpe{\Bb}$ if and
  only if $\Ext^1(X,T)=0$, we get $\Hom(T,E)\neq 0$ for every non-zero
  vector bundle $E$. If $T$ is additionally torsionfree, then we infer
  $\Ext^1(T,F)=0$ for all finite length sheaves $F$. It follows from
  (TS2) that $T$ is a generator for $\QHh$, and then also
  projective. From Serre duality we conclude that there is no non-zero
  morphism from a vector bundle to $T$, which is impossible. If on the
  other hand, $T$ has a large torsion part, then by
  Lemma~\ref{lem:T+-divisibility-lemma} the torsionfree part $T_+$ is
  $x$-divisible for (at least) one point $x$. But $T$, and then also
  $T_+$, maps epimorphic to some line bundle $L'$, and $L'$ maps
  non-trivially to a simple sheaf $S_x$ concentrated in $x$, thus
  $\Hom(T_+,S_x)\neq 0$, contradicting the $x$-divisibility. The final
  case to consider is that the torsion part $T_0$ is a branch sheaf
  $B$. By Lemma~\ref{lem:exterior-branch-reduction} then $T_+$ is
  torsionfree tilting in $\rperp{B}=\Qcoh\XX'\subseteq\QHh$. Since
  $\vect\XX'=\vect\XX\cap\rperp{B}$ (the inclusion of the right
  perpendicular category is rank-preserving,
  by~\cite[Prop.~9.6]{geigle:lenzing:1991}), we infer that $T_+$ maps
  non-trivially to any non-zero vector bundle over $\XX'$, and we get
  a contradiction by the torsionfree case treated before. Thus in any
  case, $\vSs$ contains a non-zero vector bundle.\medskip

  Since $\vSs$ is closed under subobjects, it contains also a line
  bundle $L'$. By~\cite[Lem.~IV.4.1]{reiten:vandenbergh:2002},
  \cite[Rem.~3.8]{kussin:2014} there is a suitable product $\sigma$ of
  tubular shifts such that $(L',\sigma)$ forms an ample pair, and
  there is a monomorphism $\sigma^{-1}L'\ra L'$. We conclude that
  $\vSs$ contains the system of
  generators $\{\sigma^{-n}L' \mid n\geq 0\}$   for $\QHh$.\medskip

  We have thus shown that $\Phi$ and $\Psi$ define maps between the
  named sets. Now, from~\eqref{eq:S-prime-in-S} we infer
  $\Psi\Phi(\vSs)\supseteq\vSs$. The converse inclusion follows
  from~\cite[Lem.~1.3]{reiten:ringel:2006} as in the proof of
  Lemma~\ref{lemma:undercut}(2). Thus $\Psi\Phi(\vSs)=\vSs$. Moreover,
  $\Phi\Psi(\Bb)=\rperpe{(\lperpe{\Bb}\cap\Hh)}\supseteq\rperpe{(\lperpe{\Bb})}\supseteq\Bb$. Since
  $\Bb$ is of finite type, there is $\vSs'\subseteq \Hh$ such that
  $\Bb=\rperpe{{\vSs'}}$, and from~\eqref{eq:S-prime-in-S} we conclude
  $\vSs'\subseteq\Psi(\Bb)$, hence $\Phi\Psi(\Bb)=\Bb$. This completes
  the proof of the theorem.
\end{proof}
\begin{corollary}\label{cor:Qcoh-resolving}
  Let $\XX$ be a weighted noncommutative regular projective curve and
  $\QHh=\Qcoh\XX$. A generating system $\vSs\subseteq\Hh$ is resolving
  if and only if it is closed under extensions and subobjects. \qed
\end{corollary}
We further have the following immediate consequence of
Theorem~\ref{thm:tilting-from-resolving}.
\begin{corollary}\label{cor:nonzerovb}
 Let $\XX$ be a weighted noncommutative regular projective curve and
  $\QHh=\Qcoh\XX$. If $\vSs'\subseteq\Hh$ is a set
  containing at least one non-zero vector bundle, then there is a
  tilting sheaf $T\in\QHh$ with $\rperpe{T}=\rperpe{{\vSs'}}$.
\end{corollary}
\begin{proof}
  Let $\Bb=\rperpe{{\vSs'}}$. Then $\vSs:=\lperpe{\Bb}\cap\Hh$
  satisfies $\rperpe{\vSs}=\Bb$, it is closed under extensions and
  subobjects, and we see as in the proof of
  Theorem~\ref{thm:class-correspondence} that it contains
  a generating system. Thus $\vSs$ is resolving, and the claim follows
  from Theorem~\ref{thm:tilting-from-resolving}.
\end{proof}
\medskip

\subsection*{Genus zero}
\emph{For the rest of this section let $\XX$ be of genus zero and
$\QHh=\Qcoh\XX$.}
We refine the results above with the following notion.
\begin{definition}\label{resolving}
  Let $\vSs$ be a class of objects in $\Hh$. We call $\vSs$
  \emph{strongly resolving} if it is closed under extensions and
  subobjects, and if it contains a tilting bundle $T_{\cc}$.
\end{definition}
\begin{remark}\label{rem:filtration}
  Let $\vSs\subseteq\Hh$ be a strongly resolving class containing a
  tilting bundle $T_{\cc}$.  Then $\vSs$ is resolving (this is
  verified by using that
  $T_{\cc}(-nx)\subseteq T_{\cc}$ by~\eqref{eq:def-tubular-shift} for
  all $n\geq 0$ and all points $x\in\XX$, and that the system
  $(T_{\cc}(-nx), n\geq 0)$ is generating
  by~\cite[Prop.~6.2.1]{kussin:2009}).
 
  So we can apply Theorem \ref{thm:tilting-from-resolving} to obtain a
  tilting sheaf $T$ generating the class $\Bb=\rperpe{\vSs}$. More
  explicitly, any special $\Bb$-preenvelope
  \begin{equation}
    \label{eq:Tcc-preenvelope}
    0\ra T_{\cc}\ra T_0\ra T_1\ra 0
  \end{equation}
  of $T_{\cc}$ leads to a tilting sheaf of finite type
  $$T=T_0\oplus T_1$$ with
  $\rperpe{T}=\Bb$ {and} $T\in\Gen(T_{\cc}).$ 
   
   Indeed,
  the exact sequence
  $\Ext^1(T_1,X)\ra\Ext^1(T_0,X)\ra\Ext^1 (T_{\cc},X)\ra 0$ shows that
  $X\in\rperpe{T}$ implies
  $X\in\rperpe{{T_{\cc}}}=\Gen(T_{\cc})$, and  the claim
  follows replacing $G$ by $T_{\cc}$
  in the proof of Theorem~\ref{thm:tilting-from-resolving}.   

Notice that  the sheaves $T_0$ and
$T_1$ are $\vSs$-filtered in the sense
of~\cite[Def.~2.9]{saorin:stovicek:2011}, 
and the class $\lperpe{(\rperpe{T})}$ consists precisely of the direct
summands of the $\vSs$-filtered objects, see~\cite[Thm.~2.13 and
Cor.~2.15]{saorin:stovicek:2011}.
\end{remark}

\begin{example}\label{ex:filtration}
  (1) The system $\vSs=\vect\XX$ of all vector bundles is strongly
  resolving, and the Lukas tilting sheaf $\bL$ from
  Proposition~\ref{prop:large-tilting-torsionfree-in-general} with
  $\rperpe{\bL}=\rperpe{\vSs}$ is large, torsionfree and satisfies
  condition (TS3).

\medskip

(2) Let $T=T_{(B,V)}$ where $\emptyset\neq V\subseteq\XX$ and $B$ is a
branch sheaf. The class $\vSs=\lperpe{(\rperpe{T})}\cap\Hh$ is given
by~\eqref{eq:undercut-formula}, and it is strongly resolving as
$\vect\XX\subseteq\vSs$; we even have $T_{\can}\in\vSs$. By the
preceding discussion $\rperpe{T}=\rperpe{\vSs}$ and
$T\in\Gen(T_{\can})$. Sequence~\eqref{eq:Tcc-preenvelope} shows that
$T$ satisfies~(TS3). In fact, we will see in
Theorem~\ref{thm:axiom-TS3-for-T_VB} that $T$ even satisfies
condition~(TS3+).
\end{example}

\section{Tilting sheaves under perpendicular calculus}\label{calculus}
\emph{Throughout this section, $\QHh=\Qcoh\XX$ with $\XX$ a weighted
  noncommutative regular projective curve over a field $k$.} We use
perpendicular calculus (in particular
Lemma~\ref{lem:exterior-branch-reduction}) to reduce some
considerations to tilting sheaves $T_V=T_{(0,V)}$ with trivial branch
sheaf $B=0$. This will allow to obtain an explicit description of the
torsionfree part $T_+$ of any tilting sheaf $T_{(B,V)}$ and an
alternate method to determine the Pr\"ufer summands in the torsion
part.
\begin{remark}
  The Perpendicular Lemma~\ref{lem:exterior-branch-reduction} has several applications.\medskip

  (1) Let $B\in\Hh$ be a branch sheaf. Let $T\in\QHh$ be a sheaf such
  that $tT$ and $B$ have disjoint supports and $\Ext^1(B,T)=0$
  holds. Then $T\in\rperp{B}$. (This follows by applying $\Hom(B,-)$
  to the canonical exact sequence $0\ra tT\ra T\ra T/tT\ra 0$.) Thus
  we can use Lemma~\ref{lem:exterior-branch-reduction} to reduce our
  considerations to tilting sheaves with trivial exterior branch part
  $B_{\mathfrak e}$.\medskip
  
  (2) Let $\XX$ be a noncommutative regular projective curve of weight
  type $(p_1,\dots,p_t)$ (with $p_i\geq 2$), and assume that $\XX'$
  arises from $\XX$ by reduction of some weigths, so that $\XX'$ is of
  weight type $(p'_1,\dots,p'_t)$, with $1\leq p'_i\leq p_i$. Then the
  classification of (large) tilting sheaves in $\Qcoh\XX$ is at least
  as complicated as the classification in $\Qcoh\XX'$. Indeed, if $T'$
  is a (large) tilting sheaf in $\Qcoh\XX'$, then we find a branch
  sheaf $B\in\coh\XX$ such that $T=T'\oplus B$ is (large) tilting in
  $\Qcoh\XX$: namely, we have
  $\Qcoh\XX'\simeq\rperp{\vSs}\subseteq\Qcoh\XX$ for a finite set
  $\vSs$ of exceptional simple sheaves; we can then take any branch
  sheaf $B$ whose components lie in the wings whose bases belong to
  $\vSs$; then $\rperp{B}=\rperp{\vSs}$ and $T'\in\rperp{B}$. Clearly,
  if $T'_1$ and $T'_2$ are not equivalent, then $T'_1\oplus B$ and
  $T'_2\oplus B$ are also not equivalent.\medskip

  (3) In particular: if $\XX$ is a weighted projective line of wild
  type (in the sense of~\cite{geigle:lenzing:1987}), then $\Qcoh\XX$
  contains all large tilting sheaves coming from a suitable weighted
  projective line $\XX'$ of tubular type.
\end{remark}
\emph{Let us now assume} that $V\neq\emptyset$ and
$B_{\mathfrak e}=0$. Then all the branches of $B=B_{\mathfrak i}$ are
interrelated with Pr\"ufer summands of $T_{(B,V)}$ as described in
Lemma~\ref{lem:exceptional-tube}~(2). Let
$\QHh'=\rperp{(\tau^- B)}=\Qcoh\XX'$ and $i\colon\QHh'\ra\QHh$ the
inclusion. If we define, in analogy of Definition~\ref{def:tordiv},
the class $\vSs'_V$ and its direct limit closure
$\Tt'_V=\vec{\vSs'_V}$ in $\QHh'$, then it is easy to see that we have
$$\QHh'/\Tt'_V\simeq\rperp{{\vSs'_V}}=\rperp{(\tau^-B)}\cap\rperp{(i\vSs'_V)}
=\rperp{{\vSs_V}}\simeq\QHh/\Tt_V.$$
\begin{lemma}
  Let $T=T_{(B,V)}$ with $B_{\mathfrak e}=0$. Then
  \begin{equation}
    \label{eq:def-T_V}
    T_V:=T_{(0,V)}=T_+\oplus\bigoplus_{x\in
      V}\bigoplus_{j\in\Rr_x}\tau^jS_x[\infty]
  \end{equation}
  is a large tilting sheaf in $\QHh'$.
\end{lemma}
\begin{proof}
  It is sufficient to show that $T_{(0,V)}$ lies in the
  right-perpendicular category $\rperp{(\tau^- B)}$. By the definition
  of $\Rr_x$, and since the $\tau^j S_x[\infty]$ are injective, this is
  true for the direct sum of the Pr\"ufer summands. Since $T_+$ is
  $V$-divisible, this also holds for $T_+$.
\end{proof}
We conclude
\begin{corollary}
  $T_{(B,V)}=T_{(B_{\mathfrak i},V)}\oplus B_{\mathfrak e}$ and $T_{(0,V)}$ have the same
  torsionfree part. \qed
\end{corollary}
We will now deal with $T_V=T_{(0,V)}$. Its torsion part consists of
Pr\"ufer sheaves only. We consider $T_V$ as object in
$\QHh'=\Qcoh\XX'=\rperp{(\tau^-B)}$, and we exhibit the following
explicit construction. \medskip

Let $\Lambda'$ be a finite direct sum of indecomposable vector bundles
$F_j$ in $\QHh'=\Qcoh\XX'$ such that $\Lambda'$ maps onto each simple
sheaf in $\QHh'$. For instance,
\begin{itemize}
\item by~\cite[Prop.~1.1]{kussin:2014}, we can always find special
  line bundles $F_j$ with this property (by applying suitable tubular
  shifts to the structure sheaf $L$); or
\item in case $\XX$ is of genus zero, we can take alternatively
  $\Lambda'=T'_{\can}$, a canonical configuration in $\QHh'$. (See
  Remark~\ref{rem:canonical-subconfiguration}.)
\end{itemize}
We denote by $e(j,x)=e(j,x,\Lambda')$ the $\End(S_x)$-dimension of
$\Ext^1(\tau^j S_x,\Lambda')$, by $p'\!(x)$ the weight of $x$ in
$\XX'$, and consider the universal sequence in $\Hh'$
\begin{equation}
  \label{eq:x-universal-Lambda}
  0\ra\Lambda'\ra\Lambda'(x)\ra\bigoplus_{j=0}^{p'\!(x)-1}(\tau^j S_x)^{e(j,x)}\ra 0.
\end{equation}
Since the inclusion $S_x\ra S_x[\infty]$ yields a surjection
$\Ext^1(S_x[\infty],\Lambda')\ra\Ext^1(S_x,\Lambda')$, this induces a
short exact sequence in $\QHh'\subseteq\QHh$
\begin{equation}
  \label{eq:eta-x}
  \eta_x\colon 0\ra\Lambda'\ra
  \Lambda'_{x} \ra\bigoplus_{j=0}^{p'\!(x)-1}(\tau^j S_x[\infty])^{e(j,x)}\ra 0.
\end{equation}
(In~\eqref{eq:x-universal-Lambda} the $\tau^j S_x$ are the simple
sheaves in $\Hh'$ concentrated in $x$; since for the definition of
$\Lambda'_x$ we form the direct limit, we could also replace them by
the simple sheaves $\tau^j S_x$ in $\Hh$ (with $j\in\Rr_x$ and $\tau$
the Auslander-Reiten translation in $\Hh$).)  For $x\in V$ these short
exact sequences are spliced together via
\begin{equation}
  \label{eq:ext-identity}
  \Ext^1\bigl(\bigoplus_{y\in
  V}\tau^j S_y[\infty],\Lambda'\bigr)\simeq\prod_{y\in
  V}\Ext^1(\tau^j S_y[\infty],\Lambda'),
\end{equation}
which defines
\begin{equation}
  \label{eq:def-Lambda_V}
  \eta_V\colon 0\ra\Lambda'\ra
  \Lambda'_{V}\ra\bigoplus_{x\in V}\bigoplus_{j=0}^{p'\!(x)-1}(\tau^j
  S_x[\infty])^{e(j,x)}\ra 0.
\end{equation}
\begin{lemma}\label{lem:Lambda_V-tf-V-div}
  $\Lambda'_V$ is torsionfree and precisely $V$-divisible. 
\end{lemma}
\begin{proof}
  That $\Lambda'_V$ is torsionfree and $V$-divisible can be shown as
  in the proof of~\cite[Prop.~5.2]{ringel:1979}. Let $y\in\XX\setminus
  V$ and $S\in\Uu_y$ be simple. By applying $\Hom(S,-)$ to
  sequence~\eqref{eq:def-Lambda_V} we get
  $\Ext^1(S,\Lambda'_V)\simeq\Ext^1(S,\Lambda')\neq 0$. Thus
  $\Lambda'_V$ is precisely $V$-divisible.
\end{proof}
We now adopt the notation from Section 3 and interpret the sequence
$\eta_V$ in (\ref{eq:def-Lambda_V}) in terms of localization theory.
\begin{lemma}\label{lem:tilt-V-construct}
  Assume $V\neq\emptyset$ and $B_{\mathfrak e}=0$.  Let
  $\pi=\pi_V\colon\QHh\ra\QHh/\Tt_V$ be the canonical quotient
  functor.
  \begin{enumerate}
  \item In $\rperp{{\vSs_V}}\simeq\QHh/\Tt_V$ we have
    $\pi\Lambda'\simeq\pi(\Lambda'_V)$.\smallskip
  \item $\pi\Lambda'$ is a finitely presented projective generator in
    $\rperp{{\vSs_V}}\simeq\QHh/\Tt_V$.\smallskip
  \item The functor $X\mapsto\Hom_{\QHh/\Tt_V}(\pi\Lambda',X)$ yields
    an equivalence $$\QHh/\Tt_V\simeq\Mod(\End_{\QHh/\Tt_V}(\pi\Lambda')).$$
    In particular, $\rperp{{\vSs_V}}$ is locally noetherian. 
  \end{enumerate}
\end{lemma}
\begin{proof}
  (1) This is clear by the exact sequence~\eqref{eq:def-Lambda_V}.

  (2) Let $x\in V$. Then $\Lambda'$ and $\Lambda'(nx)$ become
  isomorphic in $\QHh/\Tt_V$ for all $n\in\mathbb{Z}$, which follows
  from~\eqref{eq:x-universal-Lambda}. We note that every short exact
  sequence in $\QHh/\Tt_V$ is isomorphic to the image of a short exact
  sequence in $\QHh$ under the quotient functor $\pi$. If $A\in\Hh$,
  then, by~\cite[0.4.6]{kussin:2009}, \cite{kussin:2014}, for
  sufficiently large $n>0$ we have $\Ext^1(\Lambda'(-nx),A)=0$, which
  shows that $\pi \Lambda'\simeq\pi(\Lambda'(-nx))$ is projective with
  respect to images of coherent objects. Since the class
  $\Ker\Ext^1(\pi\Lambda',-)$ is closed under direct limits, it
  follows that $\pi\Lambda'$ is projective. Since also, again
  by~\cite[0.4.6]{kussin:2009}, for sufficiently large $n>0$ we have
  $\Hom(\Lambda'(-nx),A)\neq 0$, we get
  $\Hom(\pi\Lambda',\pi A)\neq 0$ for every $A\in\Hh$, and it follows
  easily that $\pi\Lambda'$ is a generator in the quotient category.
  It is finitely presented because $\Hom(\Lambda',-)$ and hence
  $\Hom(\pi\Lambda',-)$ preserve direct limits (we refer to
  Remark~\ref{rem:rperp} and~\cite[Lem.~2.5]{krause:1997}).

  (3) This is a well-known result by Gabriel-Mitchell, we refer
  to~\cite[II.1]{bass:1968}. For the last statement, recall that
  $\Lambda'$ is noetherian, and so is $\End_{\QHh/\Tt_V}(\pi\Lambda')$.
\end{proof}
As an additional information on $\Lambda'_V$ we exhibit its minimal
injective resolution. We recall that the sheaf $\Kk$ of rational
functions is the injective envelope of the structure sheaf $L$.
\begin{proposition}
  Let $\emptyset\neq V\subseteq\XX$. There is a short exact sequence
  \begin{equation}
    \label{eq:ses-2}
    0\ra\Lambda'_V
    \ra\Lambda'_{\XX}\ra\bigoplus_{y\in\XX\setminus
      V} \bigoplus_{j=0}^{p'\!(y)-1}(\tau^j S_y[\infty])^{e(j,x)}\ra 0.
  \end{equation}
  This is the minimal injective resolution of $\Lambda'_V$. Moreover,
  $\Lambda'_{\XX}\simeq\Kk^n$ with $n=\rk(\Lambda')$.
\end{proposition}
\begin{proof}
  Via the identity~\eqref{eq:ext-identity} we have
  $\eta_V=(\eta_y)_{y\in V}$ and $\eta_\XX=(\eta_x)_{x\in\XX}$. Thus
  inclusion $\iota\colon\displaystyle\bigoplus_{y\in
    V}\bigoplus_{e(y)}S_y[\infty]\ra\bigoplus_{x\in
    \XX}\bigoplus_{e(x)}S_x[\infty]$ induces a map on the
  $\Ext^1$-spaces, which on the products induces projection onto the
  components of $V$, and thus maps $\eta_\XX$ to $\eta_V$. Thus there
  is a pull-back diagram
  $$\xy\xymatrixcolsep{1pc}\xymatrixrowsep{0.9pc}\xymatrix{ \eta_V\colon
    0\ar @{->}[r] & \Lambda' \ar @{->}[r]\ar @{=}[d] & \Lambda'_V \ar
    @{->}[r]\ar @{->}[d] & \displaystyle\bigoplus_{y\in
      V}\bigoplus_{j=0}^{p'\!(y)-1}(\tau^j S_y[\infty])^{e(j,x)}\ar
    @{->}[r] \ar
    @{->}[d]^-{\iota} & 0\\
    \eta_\XX\colon 0\ar @{->}[r] & \Lambda' \ar @{->}[r] &
    \Lambda'_{\XX} \ar @{->}[r] & \displaystyle\bigoplus_{x\in
      \XX}\bigoplus_{j=0}^{p'\!(x)-1}(\tau^j S_x[\infty])^{e(j,x)}\ar
    @{->}[r] & 0}\endxy,$$ that is, $\eta_V =\eta_\XX\cdot\iota$. Now we get
  sequence~\eqref{eq:ses-2} with the snake lemma. The
  sequence~\eqref{eq:def-Lambda_V} is, for $V=\XX$, the minimal
  injective resolution of $\Lambda'$; this follows from the
  construction of $\Lambda'_{\XX}$ like
  in~\cite[Thm.~4.1]{reiten:ringel:2006}. Therefore
  $\Lambda'_{\XX}\simeq\Kk^n$ with $n=\rk(\Lambda')$. From the
  monomorphisms $\Lambda'\ra\Lambda'_V\ra\Lambda'_{\XX}$ it is then
  clear that the sequence~\eqref{eq:ses-2} is the minimal injective
  resolution of $\Lambda'_V$.
\end{proof}
Since the sequence~\eqref{eq:ses-2} lies in
$\rperp{{\vSs_V}}=\Mod(\End_{\QHh/\Tt_V}(\pi\Lambda'))$, it is also the
minimal injective resolution of the projective generator
$\pi\Lambda'_V$.\medskip

The main result about the torsionfree part interprets $T_+$ as a
projective generator in the localization of $\QHh$ (or $\QHh'$) at
$V$.
\begin{proposition}
  $\Add(T_+)=\Add(\Lambda'_V)$.
\end{proposition}
\begin{proof}
  Invoking the uniqueness statement of
  Theorem~\ref{thm:large-tilting-sheaves-in-general} it is sufficient
  to show that $Q=Q_+\oplus Q_0$ with $Q_+=\Lambda'_V$ and
  $Q_0=T_0=\bigoplus_{x\in V}\bigoplus_{j=0}^{p'\!(x)-1}\tau^j
  S_x[\infty]$ is a tilting object in $\QHh'$. From
  Lemma~\ref{lem:tilt-V-construct} we deduce
  $\Ext^1(Q_+,{{Q_+}^{(I)}})=0$, and using the
  sequence~\eqref{eq:ses-2} we see that $\Ext^1(Q,Q^{(I)})=0$ for each
  set $I$. Let $X\in\QHh'$. We conclude that $X\in\Gen(Q)$ implies
  $X\in\rperpe{Q}$. We have to show that the converse also holds.  So,
  let now $X\in\rperpe{Q}$. In particular, $X\in\rperpe{{Q_0}}$. The
  embeddings $S_y\ra S_y[\infty]\ra Q_0$ give rise to epimorphisms
  $\Ext^1 (Q_0,X)\rightarrow\Ext^1 (S_y,X)$ for all $y\in V$, and
  hence $X$ is $V$-divisible. Consider the short exact sequences
  $0\ra K\ra {Q_+}^{(I)}\ra B\ra 0$ and $0\ra B\ra X\ra C\ra 0$, where
  $I=\Hom(Q_+,X)$, so that $B$ is the trace of $Q_+$ in $X$. It is
  sufficient to show that $C=0$. Since $X$ is $V$-divisible, the same
  holds for $C$. Moreover $\Hom(Q_+,C)=0$. We show, that $C$ is
  $V$-torsionfree. Assume, this is not the case. Then there is
  $y\in V$ such that $\Hom(S_y,C)\neq 0$. Since $C$ (and thus also
  $tC$ and $(tC)_y$) is $y$-divisible, we get
  $S_y[\infty]\subseteq (tC)_y\subseteq C$.  Since $S_y[\infty]$ is
  injective, there is a surjection
  $\Hom(Q_+,S_y[\infty])\ra\Hom(\Lambda',S_y[\infty])\neq 0$, and
  $\Hom(Q_+,C)\neq 0$ follows, a contradiction. Thus,
  $C\in\rperp{{\vSs_V}}$, and since $\Hom(Q_+,C)=0$, we get $C=0$ by
  Lemma~\ref{lem:tilt-V-construct}. This finishes the proof.
\end{proof}
The following is a reformulation of
Theorem~\ref{thm:large-tilting-sheaves-in-general}.
\begin{theorem}\label{thm:classif-tubular-torsion}
  Let $\XX$ be a weighted noncommutative regular projective curve.
  The tilting sheaves in $\QHh$ having a large torsion part are, up to
  equivalence, the sheaves of the form $$T_{(B,V)}=T_{V}\oplus B$$
  with a subset $\emptyset\neq V\subseteq\XX$, a branch sheaf
  $B=B_{\mathfrak i}\oplus B_{\mathfrak e}$ with interior and exterior
  part $B_{\mathfrak i}$ and $B_{\mathfrak e}$, respectively, and a
  tilting sheaf $T_{V}$ in the category
  $\Qcoh\XX'=\rperp{(B_{\mathfrak e}\oplus\tau^- B_{\mathfrak
      i})}\subseteq\QHh$,
  given as the direct sum of the middle term and the end term of the
  sequence~\eqref{eq:def-Lambda_V}. \qed
\end{theorem}
\begin{corollary}\label{cor:general-homogeneous-torsion}
  Let $\XX$ be a (non-weighted) noncommutative regular projective
  curve. The tilting sheaves in $\QHh$ having a large torsion part
  are, up to equivalence, the sheaves $T_V$ with
  $\emptyset\neq V\subseteq\XX$. \qed
\end{corollary}
\begin{remark}[Genus zero]\label{rem:canonical-subconfiguration}
  Let $\XX$ be of genus zero and consider the tilting bundle
  $T_{(B,V)}$ in $\QHh=\Qcoh\XX$. Then we can choose $\Lambda'$ as the
  canonical configuration $T'_{\can}$ in the category
  $\QHh'=\Qcoh\XX'$ as above.

  We recall from~\cite[Sec.~5]{lenzing:delapena:1999} that a canonical
  configuration $T_{\can}=\Lambda$ in $\Qcoh\XX$, considered as full
  subcategory of $\Hh$, has the following form:
  \begin{equation}
  \label{eqn:can-conf}
  \xy\xymatrixcolsep{1.7pc}\xymatrixrowsep{0.8pc}\xymatrix{ & L_1(1) \ar
    @{->}[r] & L_1(2) 
    \ar @{->}[r] & \cdots\ar @{->}[r] & L_1(p_1-2) \ar @{->}[r]
    & L_1(p_1-1) \ar@<+0.5ex>[ddr] & \\
    & L_2(1) \ar @{->}[r] & L_2(2) \ar @{->}[r]&\cdots\ar @{->}[r] & L_2(p_2-2)
    \ar @{->}[r] & L_2(p_2-1) \ar @{->}[dr] & \\
    L \ar @{->}[uur] \ar @{->}[ur] \ar
    @{->}[dr]& \vdots & \vdots &
    & & \vdots & \overline{L}\\
    & L_t(1) \ar @{->}[r] & L_t(2) \ar @{->}[r] & \cdots\ar @{->}[r] &
    L_t(p_t-2) \ar @{->}[r] & L_t(p_t-1) \ar @{->}[ur] & }\endxy
\end{equation}
If now a branch sheaf $B=B_{\mathfrak i}\oplus B_{\mathfrak e}$ is
given, we can assume, by applying suitable tubular shifts (associated
to the exceptional points) to $\Lambda$, that we have
$\Hom(L,B_{\mathfrak i})=0=\Hom(\overline{L},B_{\mathfrak i})$ and
$\Hom(L,\tau B_{\mathfrak e})=0=\Hom(\overline{L},\tau B_{\mathfrak
  e})$.
Then those direct summands of $\Lambda$ lying in
$\rperp{(B_{\mathfrak e}\oplus\tau^- B_{\mathfrak i})}\simeq\Qcoh\XX'$
form a canonical configuration $\Lambda'=T'_{\can}$ in $\Qcoh\XX'$,
containing $L$ and $\overline{L}$; it arises from $\Lambda$ by
removing some ``non-adjacent segments''
$L_i(j),\,L_i(j+1),\dots,L_i(j+r-2)$ from the inner parts of the arms.
(Compare also~\cite[Thm.~3.1]{lenzing:meltzer:1996}.)
\end{remark}

\section{The domestic case}
\emph{In this section let $\XX$ be a noncommutative curve of genus
  zero.} Assume that $\XX$ is of domestic type, that is, the orbifold
Euler characteristic $\chi'_{orb}(\XX)$ is positive. This means, for
the degree of the \emph{dualizing sheaf} $\tau L$ we have
$$\delta(\vom):=\deg(\tau L)=-\frac{2\ovp
  s^2}{\kappa\varepsilon}\cdot\chi'_{orb}(\XX)<0.$$
Here, $\ovp$ is the least common multiple of the weights
$p_1,\dots,p_t$, moreover $\kappa=\dim_k\End(L)$ and $s=s(\Hh)$. For
every indecomposable vector bundle $E$ one has the following slope
formula
$$\mu(\tau E)=\mu(E)+\delta(\vom).$$

We recall the main features of the domestic case:
\begin{itemize}
  \item[(D1)] All indecomposable vector bundles are stable and exceptional.
  \item[(D2)] If $E$ and $F$ are indecomposable vector bundles, then
    $\Hom(E,F)=0$ if $\mu(E)>\mu(F)$.
  \item[(D3)] If $E$ is an indecomposable vector bundle then
    $\mu(\tau E)<\mu(E)$.
  \item[(D4)] The collection $\Ff$ of indecomposable vector bundles
    $F$ such that $0\leq\mu(F)<-\delta(\vom)$ forms a slice in the
    sense of~\cite[4.2]{ringel:1984}, and
    $T_{\her}:=\bigoplus_{F\in\Ff}F$ is a tilting bundle having a tame
    hereditary algebra as endomorphism ring. We refer
    to~\cite[Prop.~6.5]{lenzing:reiten:2006} (the result there is in a
    more general context).
  \item[(D5)] There are only finitely many Auslander-Reiten orbits of
    vector bundles. (From~(D3) it follows that $\Ff$ contains
    precisely one indecomposable from each Auslander-Reiten orbit, the
    finiteness follows from~(D4).)
\end{itemize}
\begin{lemma}\label{lem:nomaps}
  Assume that $\XX$ is domestic, and that $T\in\QHh$ is a large
  tilting object which is torsionfree. Then there is no non-zero
  morphism from $T$ to a vector bundle.
\end{lemma}
\begin{proof}
  (Sketch.) We assume that there is a vector bundle $E$ with
  $\Hom(T,E)\neq 0$. Our aim is to get a contradiction. The category
  $\vect\XX$ of vector bundles consists of finitely many
  Auslander-Reiten orbits. For the slope of an indecomposable vector
  bundle $E$ we have $\mu(\tau E)<\mu(E)<\mu(\tau^- E)$. We will see
  that if there are line bundles of arbitrarily small slope to which
  $T$ maps non-trivially, then $T$ generates all line bundles; using
  that each vector bundle has a line bundle filtration, it follows
  that $T$ generates $\vect\XX$; by Serre duality it follows that no
  vector bundle maps non-trivially to $T$, which is impossible. Thus
  there is a bound $n\in\ZZ$ such that $T$ is not mapping
  non-trivially to any line bundle, nor (using line bundle filtrations
  again) to any vector bundle of slope $<n$. Then it follows
  from~\cite[Prop.~6.5]{lenzing:reiten:2006} that there is a tilting
  bundle $T_{\her}$ whose endomorphism ring is a tame hereditary
  algebra $H$ such that $\Ext^1(T_{\her},T)=0$, and thus (by
  Proposition~\ref{prop:tilting-sheaf-complex-module}) $T$ can be
  identified with an $H$-module. The details of all these arguments,
  which we just sketched only, will be postponed to the end of this
  section.
 
  We proceed with the proof of the lemma. By the above, there exists
  an indecomposable vector bundle $F$ such that $\Hom(T,F)\neq 0$ and
  $T$ does not map non-trivially to any predecessor of $F$ (since they
  have smaller slopes by stability). Then every non-zero morphism
  $T\ra F$ must be a split epimorphism, by the almost split property.
  Thus, $T$ is a tilting $H$-module having a finite dimensional
  indecomposable preprojective module $P$ (corresponding to $F$) as a
  direct summand, and then $T$ is equivalent to a finite dimensional
  tilting module $T'$,
  compare~\cite[Thm.~2.7]{angeleri:sanchez:2013}. In other words,
  $\Add(T)=\Add(T')$ in $\Mod H$, and then also in $\QHh$, where $T'$
  is a coherent tilting sheaf. Since $T$ is large, this gives the
  desired contradiction and proves the lemma.
\end{proof}
\begin{proposition}\label{prop:tilting-sheaves-finite-type}
  Let $\XX$ be a domestic curve and $T\in\QHh$ a large tilting
  sheaf. Then $T\in\Gen(T_{\cc})$ for every tilting bundle
  $T_{\cc}$. In particular, $T$ is of finite type.
\end{proposition}
\begin{proof}
  For $T=T_{(B,V)}$ this was already shown in
  Remark~\ref{rem:filtration}. Therefore we can assume that $T$ is
  torsionfree. By the preceding lemma we have
  $\Ext^1(T_{\cc},T)=\D\Hom(T,\tau T_{\cc})=0$, that is,
  $T\in\Gen(T_{\cc})$. The last statement then follows from
  Proposition~\ref{prop:tilting-sheaf-complex-module}.
\end{proof}
\begin{proposition}\label{prop:domestic-torsionfree-large-tilting}
  Assume that $\XX$ is domestic, and that $T\in\QHh$ is a large
  tilting sheaf which is torsionfree. Then $T$ is equivalent to the
  Lukas tilting sheaf $\bL$. 
\end{proposition}
\begin{proof}
  Since $T$ is torsionfree, $\rperpe{T}$ contains the class of torsion
  sheaves $\vec{\vSs}_\XX$ by Serre duality. Then
  $\lperpe{(\rperpe{T})}\cap\coh\XX\subseteq\vect\XX$, and by
  Lemma~\ref{lem:nomaps} we even have equality.  Now
  Proposition~\ref{prop:tilting-sheaves-finite-type} yields
  $\Gen(T)=\Gen(\bL)$, compare
  also Theorem~\ref{thm:class-correspondence}.
\end{proof}
The main result of this section summarizes the discussions above:
\begin{theorem}\label{thm:full-classif-domestic}
  Let $\XX$ be a domestic curve. 
  \begin{enumerate}
  \item The large tilting sheaves in $\QHh$  are, up to equivalence, the sheaves of the form
  $$T_{(B,V)}=T_{(B_{\mathfrak i},V)}\oplus B_{\mathfrak e}$$
   with a subset
    $V\subseteq\XX$, a branch sheaf $B=B_{\mathfrak i}\oplus B_{\mathfrak e}$  with
    interior and exterior part $B_{\mathfrak i}$ and $B_{\mathfrak e}$, respectively, and a  tilting sheaf
    $T_{(B_{\mathfrak i},V)}$   in the category
    $\rperp{{B_{\mathfrak e}}}=\Qcoh\XX'$; here $T_{(B_{\mathfrak i},V)}$ with
    $V\neq\emptyset$ is given by
    Theorems~\ref{thm:large-tilting-sheaves-in-general} and~\ref{thm:classif-tubular-torsion}, and
    $T_{(B_{\mathfrak i},\emptyset)}=T_{(0,\emptyset)}=\bL'$ is the
    Lukas tilting sheaf in $\rperp{{B_{\mathfrak e}}}$.\smallskip
  \item There is a bijection between the set of equivalence classes of
    large tilting sheaves in $\QHh$ and the set of  pairs $(B,V)$
     given by  a branch sheaf $B\in\Hh$ and a subset
    $V\subseteq\XX$. Moreover, every
    large tilting object is uniquely determined (up to equivalence) by
    its torsion part. \qed
  \end{enumerate}
\end{theorem}
As a special case we get:
\begin{corollary}
  Let $\XX$ be a domestic (non-weighted) noncommutative curve of genus zero.
  The large tilting sheaves in $\QHh$ are, up to equivalence, the
  sheaves of the form $T_V$ with $\emptyset\not=V\subseteq\XX$ defined
  in~\eqref{eq:def-T_V}, and the Lukas tilting sheaf $\bL$. \qed
\end{corollary}
For completeness, we record the corresponding classification of
resolving classes (compare Theorem~\ref{thm:class-correspondence} and
Lemma~\ref{lemma:undercut}).
\begin{corollary}\label{resdom}
  Let $\XX$ be a domestic curve. The complete list of the resolving
  classes $\vSs\subseteq\Hh$ containing $\vect\XX$ is given by
  $$\add\,(\,\vect\XX\cup\tau^- (B^{>})\cup\bigcup_{x\in
    V}\{\tau^j S_x[n]\mid j\in \Rr_x,\,n\in\NN\}\,)$$
  with $V\subseteq\XX$ and $B$ a branch sheaf. \qed
\end{corollary}
\subsection*{Slope arguments}\label{sec:slope-domestic}  
In this subsection we complement the arguments given in the proof of
Lemma~\ref{lem:nomaps} by more details. Most of them are
well-established for weighted projective lines, see for
example~\cite[Thm.~2.7]{lenzing:delapena:1997}. Here we see that in
general we have to be careful with the so-called special line bundles.
 
 \medskip
  
Let $\XX$ be an \emph{arbitrary} noncommutative curve of genus zero. Recall that a line
bundle $L'$ is called \emph{special} if for every exceptional point
$x_i$ there is \emph{precisely one} simple sheaf $S_i$ concentrated in
$x_i$ with $\Hom(L',S_i)\neq 0$. Every autoequivalence of $\Hh$ induces
an autoequivalence of $\Hh_0$ and thus of $\Hh/\Hh_0$, and is
therefore rank-preserving. Hence, if $L'$ is a special line bundle,
then so is $\sigma L'$. 
\begin{numb} {\bf Numerical invariants.}
\label{nr:numerical-invariants}
Each noncommutative curve of genus zero has a so-called underlying
tame bimodule, which is either of dimension type $(2,2)$ or
$(1,4)$. In the first case we have $\eps=1$, in the second $\eps=2$.
We recall that the structure sheaf $L$ has the property that for every
point $x\in\XX$ there is precisley one simple $S_x\in\Uu_x$ with
$\Hom(L,S_x)\neq 0$, and $\End(L)$ is a skew field. One then defines
$\kappa=[\End(L):k]$ and for every point $x$
  $$f(x)=\frac{1}{\eps}[\Hom(L,S_x):\End(L)],\
  e(x)=[\Hom(L,S_x):\End(S_x)].$$
   For an exceptional point $x_i$ one writes $f_i=f(x_i)$ and
  $e_i=e(x_i)$. We have $$\deg(S_x)=\frac{\ovp}{p(x)}f(x).$$
If $k$ is algebraically closed, then all the numbers $\eps$,
  $\kappa$, $e(x)$, $f(x)$ are equal to $1$. We refer
  to~\cite{lenzing:delapena:1999}, \cite{lenzing:1998}
  and~\cite{kussin:2009} for details.
\end{numb}
\begin{numb}\label{nr:canonical-sequences}
  The weight type is still arbitrary. Let $L$ be the structure sheaf,
  which is of degree $0$ and hence of slope $0$.  Let $S_1,\dots,S_t$
  be the simple exceptional sheaves such that $\Hom(L,S_i)\neq 0$.
  The exceptional vector bundles $L_i(j)$ are
  defined~\cite[Sec.~5]{lenzing:delapena:1999} as the middle terms of
  the $\add(L)$-couniversal sequences
\begin{equation}
  \label{eq:sesA0}
  0\ra L^{\eps f_i}\ra L_i(j)\ra\tau^- S_i[j]\ra 0,
\end{equation}
for $i=1,\dots,t$ and $j=1,\dots,p_i-1$.  Similarly, $\overline{L}$ is
defined as the middle term of the $\add(L)$-couniversal sequence
\begin{equation}
  \label{eq:sesA01}
  0\ra L^{\eps}\ra\overline{L} \ra S\ra 0,
\end{equation}
where $S$ is a simple sheaf concentrated in a point $x_0$ with
$p(x_0)=1$ and $f(x_0)=1$.  The vector bundle $\overline{L}$ is
exceptional, has rank $\eps\in\{1,2\}$ and slope
$\ovp/\varepsilon$. From~\eqref{eq:sesA01} we deduce that
$\overline{L}$, like $L$, satisfies
\begin{equation}
  \label{eq:ovl-special}
  \Hom(\overline{L},\tau^j S_i)\neq 0\quad\text{if and only if}\quad
  j\equiv0\mod p_i.
\end{equation}
The collection of all vector bundles
$L,\,\overline{L}$ and the $L_i(j)$ yields the canonical
configuration~\eqref{eqn:can-conf}, which we denote by
$T_{\can}$.\medskip

By~\cite[5.4 and 5.5]{lenzing:delapena:1999} there are short exact
sequences
\begin{gather}
\label{eq:sesA1}
 0\ra L^{\eps f_i}\ra L_i(1)\ra\tau^- S_i\ra 0\\
 \label{eq:sesA2}
 0\ra L_i(j-1)\ra L_i(j)\ra\tau^{-j} S_i\ra 0\\
 \label{eq:sesA3}
 0\ra L^{\eps}\ra\overline{L}\ra S\ra 0\\
\label{eq:sesA4}
0\ra L_i(j)\ra\overline{L}^{f_i}\ra\tau^{-j}S_i[p_i-j]\ra 0.  
\end{gather}
Concerning the cokernels, the only fact which will be of interest for
the following arguments is that they are of finite length, and hence
do not map to vector bundles.
\end{numb} 
For Geigle-Lenzing weighted projective
lines~\cite{geigle:lenzing:1987} it is well-known
(see~\cite[2.1]{lenzing:delapena:1997}) that for each degree $\vx$ we
have
$$\Hom(\Oo(\vx),\Oo(\vom+\vc))\neq 0\ \text{if}\
\Hom(\Oo,\Oo(\vx))=0.$$
Since $\Hom(\Oo(\vx),\Oo(\vom+\vc))=\D\Ext^1(\Oo(\vc),\Oo(\vx))$, when
we write $L=\Oo$ and $\overline{L}$ replaces $\Oo(\vc)$, the following
statement is the generalization of this to noncommutative curves of
genus zero (of arbitrary weight type).
\begin{lemma}\label{lem:first}
  Let $\XX$ be a noncommutative curve of genus zero and $X$ be an
  indecomposable vector bundle. Then $\Hom(L,X)\neq 0$ or
  $\Ext^1(\overline{L},X)\neq 0$ holds.
\end{lemma}
In the domestic case this is a special case
of~\cite[Prop.~4.1]{lenzing:reiten:2006}. 
\begin{proof}
  Assume that $\Hom(L,X)=0=\Ext^1(\overline{L},X)$. We now apply
  $\Hom(-,X)$ to several of the exact sequences above.
  Sequence~\eqref{eq:sesA3} gives
  \begin{gather*}
    0\ra\Hom(S,X)\ra\Hom(\overline{L},X)\ra\Hom(L^{\eps},X)\ra\\
     \ra\Ext^1(S,X)\ra\Ext^1(\overline{L},X)\ra\Ext^1(L^{\eps},X)\ra 0
  \end{gather*}
  and from the assumptions we conclude that all terms are
  zero. Applying then $\Hom(-,X)$ to \eqref{eq:sesA4} shows
  $\Ext^1(L_i(j),X)=0$. Similarly, \eqref{eq:sesA1} and
  \eqref{eq:sesA2} inductively yield $\Hom(L_i(j),X)=0$. Altogether
  this gives that $\Hom(T_{\can},X)=0=\Ext^1(T_{\can},X)$, and since
  $T_{\can}$ is a tilting object we get $X=0$, a contradiction.
\end{proof}
\begin{lemma}\label{lem:second}
  Let $\XX$ be domestic. Let $L$ be a special line bundle. Let $F$ be
  an indecomposable vector bundle of slope
  $\mu(F)-\mu(L)>\ovp/\eps+\delta(\vom)$. Then $\Hom(L,F)\neq 0$.
\end{lemma}
\begin{proof}
  For every special line bundle $L$ we can form a canonical
  configuration like in~\eqref{eqn:can-conf},
  see~\cite{lenzing:delapena:1999}. Then $L$ does not necessarily have
  degree zero, but still $\mu(\overline{L})-\mu(L)=\ovp/\eps$.

  We assume $\Hom(L,F)=0$. Then by Lemma~\ref{lem:first}
  $\Ext^1(\overline{L},F)\neq 0$, and by Serre duality,
  $\Hom(F,\tau\overline{L})\neq 0$. But by assumption
  $$\mu(F)>\ovp/\eps+\delta(\vom)+\mu(L)=\mu(\overline{L})+\delta(\vom)=
  \mu(\tau\overline{L}),$$
  which contradicts the stability of indecomposable vector bundles in
  the domestic case (Theorem~\ref{thm:stability}).
\end{proof}
\begin{remark}
  In the domestic case every indecomposable vector bundle is
  exceptional. In particular this is true for every line bundle. But
  there are domestic cases where not every line bundle is special. Take
  for example the domestic symbol $
  \begin{pmatrix}
    2\\
    2
  \end{pmatrix}$.
  In other words, there is one exceptional point $x$ with $p(x)=2$, $f(x)=1$
  and $e(x)=2$. Let $L$ be a special line bundle which maps onto the
  simple $S_x$. The kernel then is a line bundle $L'$. One computes
  that $[L']$ is a $1$-root in $\Knull(\XX)$ and that
  $\Hom(L',S_x)\neq 0$ and $\Hom(L',\tau S_x)\neq 0$. Hence $L'$ is
  not special.
\end{remark}
\begin{lemma}\label{lem:third}
  Let $\XX$ be domestic. Let $T$ be a torsionfree tilting
  sheaf. Assume that for every $n\in\ZZ$ there is a \emph{special} line
  bundle $L_n$ with $\mu(L_n)<n$ such that $\Hom(T,L_n)\neq 0$. If
  $L'$ is an arbitrary line bundle, then also $\Hom(T,L')\neq 0$.
\end{lemma}
\begin{proof}
  Let $L'$ be a line bundle. Choose $n\in\ZZ$ such that
  $n<\mu(L')-\ovp/\eps-\delta(\vom)$. Then
  $\mu(L')-\mu(L_n)>\ovp/\eps+\delta(\vom)$, and by
  Lemma~\ref{lem:second}, since $L_n$ is special, we have
  $\Hom(L_n,L')\neq 0$. Hence there is a monomorphism $L_n\ra
  L'$. Since $\Hom(T,L_n)\neq 0$ we get $\Hom(T,L')\neq 0$ as well.
\end{proof}
In order to remove the word ``special'' from the preceding lemma, we
use the Riemann-Roch formula (see~\cite{kussin:2014})\begin{equation}
  \label{eq:RR}
  \frac{1}{\ovp\kappa}\DLF{X}{Y}=-\frac{\eps}{2}\delta(\vom)\cdot\rk(X)\cdot
  \rk(Y)+\frac{\eps}{\ovp} 
\begin{vmatrix}
  \rk(X) & \rk(Y)\\
  \deg(X) & \deg(Y)
\end{vmatrix}
\end{equation}
where   $$\DLF{X}{Y}=\sum_{j=0}^{\ovp-1}\LF{X}{\tau^{-j}Y}$$ is the
average Euler
form.
In particular, if $L'$ and $L$ are line bundles with
$\mu(L)=\deg(L)\geq\deg(L')=\mu(L')$ then (by $\delta(\vom)<0$) we
have $\DLF{L'}{L}>0$. Since, by stability,
$\Ext^1(L',\tau^{-j}L)=\D\Hom(\tau^{-j-1}L,L')=0$, and thus
$\LF{L'}{\tau^{-j}L}=\dim_k \Hom(L',\tau^{-j}L)$, we obtain
$\Hom(L',\tau^{-j}L)\neq 0$ for some $j\in\{0,\dots,\ovp-1\}$. It
follows that there is a monomorphism $L'\ra\tau^{-j}L$, where
$\mu(\tau^{-j}L)=\mu(L)-j\cdot\delta(\vom)\leq\mu(L)-(\ovp-1)\cdot\delta(\vom)$. 

If $L$ is a special line bundle, then also $\tau^{n}L$ is special for
every integer $n$, and has slope
$\mu(\tau^{n}L)=\mu(L)+n\cdot\delta(\vom)$. So, if $L'$ is a given
line bundle, then there is a \emph{special} line bundle $L$ of slope
$\mu(L)$ in the interval $[\mu(L'),\mu(L')-\delta(\vom)[$. With the
preceding paragraph we obtain $j\in\{0,\dots,\ovp-1\}$ and a
monomorphism $L'\ra\tau^{-j}(L)$, for which
$\mu(\tau^{-j}L)\leq\mu(L)-(\ovp-1)\cdot\delta(\vom)<\mu(L')-\ovp\cdot\delta(\vom)$
holds. To summarize:
\begin{lemma} Let $\XX$ be domestic.
  For every line bundle $L'$ there is a \emph{special} line bundle $L$
  with a monomorphism $L'\ra L$, so that the distance of slopes
  $$0\leq\mu(L)-\mu(L')<-\ovp\cdot\delta(\vom)$$
  is bounded by a constant. \qed
\end{lemma}
We then have: if $L'$ is such that $\Hom(T,L')\neq 0$, then, since
there is a monomorphism $L'\ra L$, also $\Hom(T,L)\neq 0$. As a
consequence we get: if we find line bundles $L'$ of arbitrarily small
slope with $\Hom(T,L')\neq 0$, then we also find \emph{special} line
bundles $L$ of arbitrarily small slope with $\Hom(T,L)\neq
0$. Therefore we now have the stronger version of
Lemma~\ref{lem:third}.
\begin{lemma}\label{lem:fourth}
  Let $\XX$ be domestic. Let $T$ be a torsionfree tilting
  sheaf. Assume that for every $n\in\ZZ$ there is a line bundle $L_n$
  with $\mu(L_n)<n$ such that $\Hom(T,L_n)\neq 0$. If $L'$ is an
  arbitrary line bundle, then also $\Hom(T,L')\neq 0$. \qed
\end{lemma}
\begin{corollary}\label{cor:fourth}
  Let $\XX$ be domestic. Let $T$ be a torsionfree tilting sheaf. Then
  there is $n_0\in\ZZ$ such that $\Hom(T,L')=0$ for every line bundle
  $L'$ with $\mu(L')<n_0$.
\end{corollary}
\begin{proof}
  Otherwise we would have $\Hom(T,L')\neq 0$ for all line bundles $L'$
  by the preceding lemma. As in the proof of
  Lemma~\ref{lem:large-tilting-leftbounded}, we see that for a line
  bundle $L'$ the condition $\Hom(T,L')\neq 0$ amounts to
  $L'\in\Gen(T)$, and since every vector bundle has a line bundle
  filtration, we infer that all vector bundles lie in
  $\Gen(T)=\rperpe{T}$. By Serre duality we get that \emph{no} vector
  bundle (even no coherent sheaf) maps to the torsionfree sheaf $T$,
  which is a contradiction, since $\QHh$ is locally noetherian.
\end{proof}
\begin{lemma}
  Let $\XX$ be domestic. Let $T$ be a torsionfree tilting sheaf. Then
  there is $m\in\ZZ$ such that $\Hom(T,E)=0$ for every indecomposable
  vector bundle $E$ with $\mu(E)<m$.
\end{lemma}
\begin{proof}
  Let $\Ff$ be the set of indecomposable vector bundles $F$ with
  $0\leq\mu(F)<-\delta(\vom)$. This is a finite set by~(D5), and every
  indecomposable vector bundle is of the form $\tau^n F$ for some
  $F\in\Ff$ and some $n\in\ZZ$. For every $F\in\Ff$ we fix a line
  bundle filtration, which altogether form a finite collection $\Ll$
  of line bundles. We denote by $\alpha=\alpha(\Ff)$ the maximum of
  slopes of the objects in $\Ll$. Then
  $\alpha(\tau^m\Ff)=\alpha+m\delta(\vom)$. With the bound $n_0$ from
  Corollary~\ref{cor:fourth}, for all $m\in\ZZ$ such that
  $\alpha+m\delta(\vom)<n_0$, we get $\Hom(T,\tau^m\Ll)=0$, and thus
  $\Hom(T,\tau^m\Ff)=0$. It follows that $\Hom(T,E)=0$ for every
  indecomposable vector bundle $E$ with $\mu(E)<m\delta(\vom)$.
\end{proof}
As already mentioned, the fundamental domain $\Ff$ from the preceding
proof is known to form a tilting bundle $T_{\her}$ whose endomorphism
ring is a tame hereditary algebra $H$.
\begin{proposition}\label{prop:not-large}
  Let $\XX$ be domestic, and let $T$ be a torsionfree tilting
  sheaf. For $n\gg 0$ the tilting bundle $T'_{\her}=\tau^n (T_{\her})$
  has the additional property that $\Ext^1(T'_{\her},T)=0$. In other words,
  $T$ is a module over the tame hereditary algebra
  $H=\End(T'_{\her})$.
\end{proposition}
\begin{proof}
  By the above $\Hom(T,\tau^n (T_{\her}))=0$ for $n\gg 0$. Now apply
  Serre duality.
\end{proof}

\section{The tubular case}

\emph{Throughout this section let $\XX$ be a tubular noncommutative
  curve of genus zero and $\QHh=\Qcoh\XX$.}\medskip

We denote by $\ovp$ be the least common multiple of the weights
$p_1,\dots,p_t$. We recall that in the tubular case every
indecomposable coherent sheaf $E$ is semi-stable, having a
well-defined slope $\mu(E)=\widehat{\QQ}=\QQ\cup\{\infty\}$. Moreover,
for every $\alpha\in\widehat{\QQ}$ there is an indecomposable
$E\in\Hh$ with $\mu(E)=\alpha$; all these coherent sheaves form the
tubular family $\bt_\alpha$. In particular, $\bt_{\infty}$ consists of
the finite length sheaves. We can thus
write $$\Hh=\bigvee_{\alpha\in\widehat{\QQ}}\bt_{\alpha}.$$

We will need the following important application of the
Riemann-Roch formula from \cite{lenzing:delapena:1999}.
\begin{lemma}\label{lem:Hom-X-tauY}
  If $X$, $Y\in\mathcal{H}$ are indecomposable with $\mu (X)<\mu (Y)$,
  then there exists $j$ with $0\leq j\leq\ovp-1$ such that $\Hom
  (X,\tau^j Y)\neq 0$. \qed
\end{lemma}
For general information on the tubular case we refer
to~\cite{lenzing:meltzer:1992}, \cite{lenzing:1997},
\cite[Ch.~13]{reiten:ringel:2006} and~\cite[Ch.~8]{kussin:2009}.
\subsection*{Quasicoherent sheaves having a real slope}
For $w\in\widehat{\RR}=\RR\cup\{\infty\}$ we
define
$$\bp_w=\bigcup_{\alpha<w}\bt_\alpha\quad\quad\bq_w=\bigcup_{w<\beta}\bt_\beta,$$
where $\alpha,\,\beta\in\widehat{\QQ}$. Accordingly,
$\Hh=\bp_w \vee\bt_w\vee\bq_w$ if $w$ is rational, and
$\Hh=\bp_w\vee\bq_w$ if $w$ is irrational. Moreover, let
$$\Cc_w=\rperpo{{\bq_w}}=\lperpe{\bq_w}\quad\quad\Bb_w=\lperpo{\bp_w}=
\rperpe{{\bp_w}}$$
and $$\Mm(w)=\Bb_w\cap\Cc_w.$$ The sheaves in $\Mm(w)$ are said to
have \emph{slope} $w$. Clearly, for coherent sheaves this definition
of slope is equivalent to the former one, and for irrational $w$ there
are only non-coherent sheaves in $\Mm(w)$.

For $v\leq w\leq\infty$ we have $\Cc_v\subseteq\Cc_w$ and
$\Bb_v\supseteq\Bb_w$. Moreover, $$\bigcap_{w\in\widehat{\RR}}\Cc_w=0
\quad\text{and}\quad\bigcup_{w\in\widehat{\RR}}\Cc_w=\Cc_\infty=\QHh,$$
and
$$\bigcap_{w\in\widehat{\RR}}\Bb_w=\Bb_\infty=\lperpo{\vect\XX}
\quad\text{and}\quad\Hh\cap\bigcup_{w\in\widehat{\RR}}\Bb_w=\Hh.$$ We
note that for example $\bigoplus_{\alpha\in\widehat{\QQ}}S_\alpha$
with $S_\alpha\in\bt_\alpha$ quasisimple is not in
$\bigcup_{w\in\widehat{\RR}}\Bb_w$. Let $X\in\QHh$ be a non-zero
object. Let $v=\sup\{r\in\widehat{\RR}\mid
X\in\Bb_r\}\in\widehat{\RR}\cup\{-\infty\}$ and
$w=\inf\{r\in\widehat{\RR}\mid X\in\Cc_r\}\in\widehat{\RR}$. Since
$X\neq 0$ we have $v\leq w$.

In the special case, when $w=\infty$, a sheaf $X\in\QHh$ has slope
$\infty$ if and only if
$X\in\lperpo{\vect\XX}=\rperpe{(\vect\XX)}$. (This, as a definition,
makes also sense for other representation types; \emph{in the domestic
  case}, we have seen that \emph{every large tilting sheaf has slope}
$\infty$.)
\subsection*{Interval categories}
The following technique is very useful in the tubular setting. Let
$\alpha\in\widehat{\QQ}$. Denote by $\Hh\spitz{\alpha}$ the full
subcategory of $\bDerived{\Hh}$ defined
by 
$$\bigvee_{\beta>\alpha}\bt_{\beta}[-1]\vee
\bigvee_{\gamma\leq\alpha}\bt_{\gamma}.$$
By~\cite[Prop.~8.1.6]{kussin:2009} we have
$\Hh\spitz{\alpha}=\coh\XX_\alpha$ for some tubular curve
$\XX_\alpha$. (If $k$ is algebraically closed, then $\XX_\alpha$ is
isomorphic to $\XX$; but this is not true in general.) The rank
function on $\Hh\spitz{\alpha}$ defines a linear form
$\rk_\alpha\colon\Knull(\Hh)\ra\ZZ$. A sequence
$\eta\colon0\ra E'\stackrel{u}\lra E\stackrel{v}\lra E''\ra 0$ with
objects $E',E,E''$ in $\Hh\cap\Hh\spitz{\alpha}$ is exact in $\Hh$ if
and only if it is exact in $\Hh\spitz{\alpha}$; indeed, both
conditions are equivalent to
$E'\stackrel{u}\lra E\stackrel{v}\lra E''\stackrel{\eta}\lra E'[1]$
being a triangle in $\bDerived{\Hh}$.
\begin{lemma}[Reiten-Ringel]\label{lem:reiten-ringel}
  For every $w\in\widehat{\RR}$ the pair $(\Gen(\bq_w),\Cc_w)$ is a
  torsion pair, which is split in case $w\in\widehat{\QQ}$.
\end{lemma}
\begin{proof}
  As in~\cite[Lem.~1.4]{reiten:ringel:2006} one shows that
  $\Gen(\bq_w)$ is extension-closed; the same proof works in the
  locally noetherian category $\QHh$, replacing ``finite length'' by
  ``finitely presented''. Then
  $\Gen(\bq_w)=\lperpo{(\rperpo{{\bq_w}})}=\lperpo{\Cc_w}$ follows
  like in~\cite[Lem.~1.3]{reiten:ringel:2006}, and thus
  $(\Gen(\bq_w),\Cc_w)$ is a torsion pair. For the splitting property
  in case $w=\alpha\in\widehat{\QQ}$ we have to show that every short
  exact sequence $\eta\colon0\ra X\ra Y\ra Z\ra 0$ with
  $X\in\Gen(\bq_\alpha)$ and $Z\in\Cc_\alpha$ splits. We may assume
  that $X$ is a subobject of $Y$ and $Z=Y/X$. If $Z$ is finitely
  presented, this follows from Serre duality. For general
  $Z\in\Cc_\alpha$, we consider the set of subobjects $U$ of $Y$ such
  that $U\cap X=0$ and $Y/(X+U)\in\Cc_\alpha$. Like
  in~\cite[Prop.~1.5(b)]{reiten:ringel:2006} one has a maximal such
  $U$, and as in~\cite[Prop.~1.5(a)]{reiten:ringel:2006} one shows
  $Y=X\oplus U$, so that $\eta$ splits. (If one assumes that the
  inclusion $X+U\subsetneq Y$ is proper, then $\QHh$ being locally
  noetherian allows to find $Y'$ with $X+U\subsetneq Y'\subseteq Y$ with
  $Y'/(X+U)$ finitely presented. Then we proceed like
  in~\cite{reiten:ringel:2006}. We remark that an analogue of
  condition~(F) therein can be proved along the same lines by
  exploiting the fact that an indecomposable $E\in\Hh$ belongs to
  $\bq_\alpha$ if and only if $\delta(E)>0$, where
  $\delta=-\rk_{\alpha}$.)
\end{proof}
Let $\alpha\in\widehat{\QQ}$. By $\QHh\spitz{\alpha}$ we denote the
direct limit closure of $\Hh\spitz{\alpha}$ in $\bDerived{\QHh}$. We
have $\QHh\spitz{\alpha}=\Qcoh\XX_\alpha$. If $X\in\QHh$ has a
rational slope $\alpha$, then clearly $X\in\QHh\cap\QHh\spitz{\alpha}$
where the intersection is formed in
$\bDerived{\QHh}=\bDerived{\QHh\spitz{\alpha}}$; in
$\QHh\spitz{\alpha}$ then $T$ has slope $\infty$. Clearly,
$\Cc_{\alpha}=\QHh\spitz{\alpha}\cap\QHh$.
\begin{lemma}\label{lem:tilting-corr-tubular}
  Let $\alpha\in\widehat{\QQ}$. For an object $T$ in $\QHh$ lying in
  $\Cc_{\alpha}$, the following conditions are equivalent:
  \begin{enumerate}
  \item[(i)] $T$ is a tilting sheaf in $\QHh$;
  \item[(ii)] $T$ is a tilting complex in $\bDerived{\QHh}$;
  \item[(iii)] $T$ is a tilting sheaf in $\QHh\spitz{\alpha}$.
  \end{enumerate}
\end{lemma}
\begin{proof}
  This is shown like in
  Proposition~\ref{prop:tilting-sheaf-complex-module}.
\end{proof}
\begin{remark}
  There is an interesting class of locally coherent categories which
  are derived-equivalent to $\QHh$: If $w$ is irrational, then we
  define
  $\Hh\spitz{w}=\bigvee_{\beta>w}\bt_{\beta}[-1]\vee
  \bigvee_{\gamma<w}\bt_{\gamma}$ and $\QHh\spitz{w}$ similarly as
  above. It is easy to see that $\Hh\spitz{w}$ is hereditary and does
  not contain any simple object. Accordingly, $\QHh\spitz{w}$ is a
  Grothendieck category (we refer
  to~\cite[Sec.~2.4+2.5]{angeleri:kussin:2015}) which is locally
coherent but \emph{not} locally noetherian. Moreover, $\QHh\spitz{w}$
is derived-equivalent to $\QHh$ and contains a finitely presented
tilting object $T_{\can}$ whose endomorphism ring is a canonical
algebra. It is not difficult to show that there are only countably
many irrational $w'$ such that the category $\QHh\spitz{w'}$ (resp.\
$\Hh\spitz{w'}$) is equivalent to $\QHh\spitz{w}$ (resp.\
$\Hh\spitz{w}$). It would be of interest to get a better understanding
of the ``geometric meaning'' of these categories.
\end{remark}
\subsection*{Indecomposable quasicoherent sheaves}
The following statement reflects the importance of the concept of
slope in the tubular case, also for quasicoherent sheaves.
\begin{theorem}[Reiten-Ringel]\label{thm:reiten-ringel}
  \begin{enumerate}
  \item[(1)] $\Hom(\Mm(w'),\Mm(w))=0$ for $w<w'$.
  \item[(2)] Every indecomposable sheaf has a well-defined slope
    $w\in\widehat{\RR}$.
  \end{enumerate}
\end{theorem}
\begin{proof}
  (1) This follows like in~\cite[Thm.~13.1]{reiten:ringel:2006}.

  (2) We transfer the original proof for modules over a tubular
  algebra in~\cite[Thm.~13.1]{reiten:ringel:2006} to $\Qcoh\XX$; we
  need a slight modification. Let $X\in\QHh$ be indecomposable. Then
  $0\neq
  X\in\bigcup_{w\in\widehat{\RR}}\Cc_w\setminus\bigcap_{w\in\widehat{\RR}}\Cc_w$.
  Let $w\in\widehat{\RR}$ be the infimum of all
  $\alpha\in\widehat{\QQ}$ such that $X\in\Cc_\alpha$. Since
  $\bq_w=\bigcup_{w<\alpha}\bq_\alpha$, we have $\Hom(\bq_w,X)=0$,
  that is, $X\in\Cc_w$.

  We now show that $X\in\Bb_w=\lperpo{\bp_w}$. We observe
  that $$\Bb_w=\bigcap_{\alpha<w}\lperpo{\bt_\alpha}$$ and
  $\Gen(\bq_\alpha)\subseteq\lperpo{\bt_\alpha}$. Hence, if
  $X\not\in\Bb_w$, then there is a rational $\beta<w$ with
  $X\not\in\Gen(\bq_{\beta})$. But $(\Gen(\bq_{\beta}),\Cc_{\beta})$
  is a split torsion pair, and since $X$ is indecomposable, we get
  $X\in\Cc_{\beta}$. Since $\beta<w$ this gives a contradiction to the
  choice of $w$.
\end{proof}
\begin{remark}
  If $T$ is a \emph{noetherian} tilting object in $\QHh$ (that is,
  $T\in\Hh$), then $T$ does \emph{not} have any slope. In fact, if
  $T=T_1\oplus\ldots\oplus T_n$ with pairwise nonisomorphic
  indecomposable $T_i$, then $n$ coincides with the rank of the
  Grothendieck group $\Knull(\Hh)$. If $T$ would have a slope
  $\alpha$, then each summand $T_i$ would be exceptional of slope
  $\alpha$, hence lying in one of the finitely many exceptional tubes
  of slope $\alpha$. If such a tube has rank $p>1$, then there are at
  most $p-1$ indecomposable summands of $T$ lying in this tube. If
  $p_1,\dots,p_t$ are the weights of $\XX$, then
  $n=|\Knull(\Hh)|=\sum_{i=1}^t (p_i-1)+2>\sum_{i=1}^t (p_i-1)\geq n$,
  giving a contradiction.
\end{remark}
\begin{proposition}\label{prop:exist-Lw}
  Let $w\in\widehat{\RR}$.
  There is a large tilting sheaf $\bL_w$ of slope $w$.
\end{proposition}
\begin{proof}
  Applying Theorem~\ref{thm:tilting-from-resolving} to the strongly resolving
  subcategory $\add(\bp_w)$, we get a tilting sheaf $T$ with
  $\Gen(T)=\rperpe{\vSs}=\rperpe{{\bp_w}}=\Bb_w$. Moreover, by the
  tilting property clearly $T\in\lperpe{\Bb_w}$, which is a subclass
  of $\Cc_w$. By the preceding remark, $T$ is large.
\end{proof}
Let $T\in\QHh$ be a tilting object of rational slope $\alpha$. Then in
$\QHh\spitz{\alpha}$ the splitting property of
Theorem~\ref{thm:torsion-splitting} holds, that is, the canonical
exact sequence $0\ra t_{\alpha}(T)\ra T\ra T/t_{\alpha}(T)\ra 0$ in
$\QHh\spitz{\alpha}$ splits, where $t_{\alpha}(T)$ denotes the torsion
subsheaf of $T$ in $\QHh\spitz{\alpha}$.
\begin{definition}\label{def:torsionfree-tilting}
  Let $T$ be a tilting sheaf of slope $w$. We call $T$ a
  \emph{torsionfree} tilting sheaf, if either $w$ is irrational, or if
  $w=\alpha\in\widehat{\QQ}$ and $t_\alpha(T)=0$. 
\end{definition}
\begin{lemma}\label{lem:L_w-torsionfree}
  For every $w\in\widehat{\RR}$ the tilting sheaf $\bL_w$ is
  torsionfree.
\end{lemma}
\begin{proof}
  For irrational $w$ there is nothing to show. Switching to the
  category $\QHh\spitz{\alpha}$ if $w=\alpha$ is rational, we can
  assume without loss of generality that $w=\infty$. Then the claim
  follows from
  Proposition~\ref{prop:large-tilting-torsionfree-in-general}.
\end{proof}
\begin{lemma}\label{lem:torsion-implies-slope}
  Let $\alpha\in\widehat{\QQ}$. Let $T\in\QHh$ be a large tilting
  sheaf with $T\in\Cc_\alpha$ and $t_\alpha(T)\neq 0$. Then $T$ has
  slope $\alpha$.
\end{lemma}
\begin{proof}
  Switching to the category $\QHh\spitz{\alpha}=\Qcoh\XX_{\alpha}$, we
  can assume without loss of generality that $\alpha=\infty$. (This
  will just simplify the notation.) If $tT$ contains a non-coherent
  summand, then with Theorem~\ref{thm:classif-tubular-torsion} we get
  that $T$ has slope $\infty$, since $T\in\Bb_{\infty}$ follows
  from~\ref{lem:V-divisibility-lemma}. If, on the other hand, $tT$
  only consists of coherent summands (necessarily only finitely many
  indecomposables) then $T/tT$ is a torsionfree tilting sheaf in
  $$\Qcoh\XX'=\rperp{tT}\subseteq\QHh,$$
  where $\XX'$ is a curve with reduced weights, thus of domestic
  type. By~\cite[Prop.~9.6]{geigle:lenzing:1991} the induced inclusion
  $\coh\XX'\subseteq\Hh$ is rank-preserving. The torsionfree sheaf
  $T/tT$ is equivalent to the Lukas tilting sheaf $\bL'\in\Qcoh\XX'$
  by Proposition~\ref{prop:domestic-torsionfree-large-tilting}.  We
  show that $\bL'$, as object in $\QHh$, has slope $\infty$. We assume
  this is not the case. Then there is $\beta<\infty$ with
  $\Hom(\bL',\bt_\beta)\neq 0$. Since in $\QHh$ every vector bundle has
  a line bundle filtration, it follows that there is a line bundle $L$
  with $\Hom(\bL',L)\neq 0$. Since non-zero subobjects of line bundles
  are line bundles, we can assume without loss of generality that
  there is an epimorphism $f\colon\bL'\ra L$. Let $E$ be an
  indecomposable vector bundle over $\XX'$, considered as object in
  $\QHh$. Let $x_0\in\XX$ be a homogeneous point. The support of $tT$
  is disjoint from $x_0$, and thus the associated tubular shift
  automorphism $\sigma_0$ fixes $tT$. Then $E(nx_0)\in\vect\XX'$ for
  all $n>0$: indeed, by definition of the tubular shift there is an
  exact sequence $0\ra E\ra E(nx_0)\ra C\ra 0$ in $\QHh$ with $C$
  lying in the tube $\Uu_{x_0}$; then $E,\,C\in\rperp{tT}$ implies
  $E(nx_0)\in\rperp{tT}$, having the same rank as
  $E$. By~\cite[(S15)]{lenzing:delapena:1999}, for $n\gg 0$ we have
  $\Hom(L,E(nx_0))\neq 0$. We get $\Hom(\bL',E(nx_0))\neq 0$, which
  also holds in the full subcategory $\Qcoh\XX'$, and gives a
  contradiction since in $\Qcoh\XX'$ one has
  $\bL'\in\lperpo{\vect\XX'}$. Thus $\bL'$ has slope $\infty$, and so
  has $T$, which is equivalent to $\bL'\oplus tT$.
\end{proof}
In the tubular case, the tilting bundle $T_{cc}$ can be chosen such
that the indecomposable summands have arbitrarily small slopes. This
will imply that tilting sheaves have finite type. The following
statement is crucial.
\begin{lemma}\label{lem:large-tilting-leftbounded}
  For any large tilting sheaf $T\in\QHh$ there is
  $\alpha\in\widehat{\QQ}$ with $T\in\Bb_{\alpha}$. 
\end{lemma}
\begin{proof}
  If $T$ has a non-trivial torsion part, then $T$ has slope $\infty$
  by Lemma~\ref{lem:torsion-implies-slope}. Thus we will assume in the
  following that $T$ is torsionfree.

  Let $\Bb=\Gen(T)$ and $\vSs=\lperpe{\Bb}\cap\Hh$. Suppose there is
  no $\alpha$ with $T\in\Bb_{\alpha}$. We will lead this to a
  contradiction.  From the assumption it follows that there are
  infinitely many and arbitrarily small $\alpha$ with
  $\Hom(T,\bt_{\alpha})\neq 0$: otherwise there would be some $\alpha$
  with $\bp_{\alpha}\subseteq\vSs$, and
  then $$T\in\Bb\subseteq\rperpe{(\lperpe{\Bb})}\subseteq\rperpe{\vSs}\subseteq
  \rperpe{{\bp_{\alpha}}}=\Bb_{\alpha}.$$ We show that there is no
  $\alpha$ such that $\vSs\cap\bp_{\alpha}=\emptyset$. Indeed, if
  there were such $\alpha$, then $\Hom(T,X)\neq 0$ for all
  $X\in\bp_{\alpha}$. So, for any line bundle $L$ in $\bp_{\alpha}$,
  the trace $L'$ of $T$ in $L$ would be a non-zero line bundle
  again. Applying $\Ext^1(T,-)$ to the short exact sequence $0\ra
  L'\ra L\ra C\ra 0$ would give $\Ext^1(T,L)=0$, as $T$ is torsionfree
  and $C$ has finite length. Then, given a point $x\in\XX$ and an
  integer $n\geq 1$, we would infer $\Ext^1(T,L(nx))=0$ from the exact
  sequence $0\ra L\ra L(nx)\ra F\ra 0$ with $F$ of finite length. Now,
  since any line bundle $L$ in $\Hh$ satisfies $L(-nx)\in\bp_{\alpha}$
  for $n\gg 0$, we would conclude that $\Hom(L,T)=\D\Ext^1(T,\tau
  L)=0$ holds for all line bundles, and using line bundle filtrations,
  even for all vector bundles. But this is clearly impossible, since
  $T\neq 0$, as torsionfree object, is a direct limit of vector
  bundles.

  Thus $\vSs\cap\bt_{\alpha}\neq\emptyset$ for infinitely many and
  arbitrarily small $\alpha$. Let $X\in\vSs\cap\bt_{\alpha}$ be
  indecomposable, and let $\beta<\alpha$ with $\Hom(T,\bt_{\beta})\neq
  0$. Considering images, there is an indecomposable $B\in\Hh$ with
  $B\in\Gen(T)$ and slope $\mu(B)<\alpha$. By
  Lemma~\ref{lem:Hom-X-tauY} we have $\Hom(B,\tau^j X)\neq 0$ for some
  integer $j$. If we assume that $X$ is not exceptional, we can even
  show $\Hom(B,\tau X)\neq 0$. Indeed, this is clear if $X$ lies in a
  homogeneous tube, which means $\tau X=X$, while for $X$ lying in an
  exceptional tube of rank $p>1$ we know from
  Lemma~\ref{lem:Hom-X-tauY} that $B$ maps non-trivially into a
  quasisimple object of the tube, and by the almost split property it
  follows inductively that $B$ maps non-trivially into each object
  from the tube which has quasilength $\geq p$, so in particular into
  $\tau X$. Now we get $\Ext^1(X,T)=\D\Hom(T,\tau X)\neq 0$, which is
  a contradiction to $X\in\vSs\subseteq\lperpe{\Bb}$.  So we conclude
  that every indecomposable $X\in\vSs$ is exceptional.
  
  We now fix an indecomposable, exceptional $X\in\vSs\cap\bt_{\alpha}$
  lying in a tube of rank $p>1$, and we show that there is an
  indecomposable $Y$ in the same tube and of the same quasi-length
  such that $\Hom(T,Y)\neq 0$.  To this end, we start with an
  arbitrary object $Z$ of quasi-length $p$ in the tube. Since
  $\tau^{-}Z$ is not exceptional, and thus not in $\vSs$, we have
  $\Hom(T,Z)\neq 0$. Then, considering the almost split sequences, we
  get inductively that $T$ maps non-trivially to an object of
  quasi-length $\ell$ for any $\ell<p$: given $0\neq f\in\Hom(T,Z)$
  where $Z$ is indecomposable of quasilength $\ell$, with
  $2\leq\ell\leq p$, there is an irreducible monomorphism $\iota$
  ending in $Z$ and an irreducible epimorphism $\pi$ starting in $Z$,
  and either $\pi f\neq 0$, or $f$ factors through $\iota$; in both
  cases $T$ maps non-trivially to an object in the tube of quasilength
  $\ell-1$.

  Now, since $\Hom(T,\tau X)=0$, we can assume $\Hom(T,Y)\neq 0$ and
  $\Hom(T,\tau Y)=0$. We conclude $\Ext^1(Y,T)=\D\Hom(T,\tau Y)=0$,
  thus $Y\in\vSs$. Let $B$ be an indecomposable summand of the trace
  of $T$ in $Y$. Since $B\subseteq Y$, we conclude $\Ext^1(B,T)=0$,
  thus $B\in\vSs$. Thus $B\in\Bb\cap\lperpe{\Bb}$, and by
  Lemma~\ref{lem:porperties-tilting}, $B$ is a direct summand of $T$,
  of slope $\mu(B)\leq\alpha$. 

  Repeating the argument for slope smaller than $\mu(B)$ we get an
  infinite sequence of indecomposable coherent sheaves
  $B_1,\,B_2,\,B_3,\dots$ lying in $\add(T)$, and with slopes
  $\mu(B_1)>\mu(B_2)>\mu(B_3)>\dots$. We conclude $\Ext^1(B_i,B_j)=0$
  for all $i,\,j$ and $\Hom(B_i,B_j)=0$ for all $i<j$. So, for all
  $n$, the sequence $(B_n,\dots,B_2,B_1)$ is exceptional in
  $\Hh$. This gives our desired contradiction, since the length of
  exceptional sequences in $\Hh$ is bounded by the rank of the
  Grothendieck group $\Knull(\Hh)$.
\end{proof}
\begin{proposition}\label{prop:tilting-sheaves-finite-type-tub}
  Every tilting sheaf $T\in\QHh$ is of finite type. 
\end{proposition}
\begin{proof}
  By Lemma~\ref{lem:large-tilting-leftbounded} there is
  $\alpha\in\widehat{\QQ}$ with
  $T\in\Bb_{\alpha}=\rperpe{{\bp_{\alpha}}}$. Then
  $\Ext^1(\bp_{\alpha},T)=0$, and choosing a tilting bundle
  $T_{\cc}\in\bp_{\alpha}$, we get $\Ext^1(T_{\cc},T)=0$. Now we can
  apply Proposition~\ref{prop:tilting-sheaf-complex-module}.
\end{proof}
The proof above also shows that $\vSs=\lperpe{(\rperpe{T})}\cap\Hh$ is
a strongly resolving subcategory of $\Hh$ such that
$\Gen(T)=\rperpe{\vSs}$. Now let us start conversely with a strongly resolving
subcategory.
\begin{lemma}\label{lem:TinC}
  Let $\alpha\in\widehat{\QQ}$ and $\vSs\subseteq\Cc_\alpha\cap\Hh$ be
  strongly resolving.
  \begin{enumerate}
  \item There is a tilting sheaf
  $T\in\QHh$ with $T\in\Cc_\alpha$ and $\rperpe{T}=\rperpe{\vSs}$. Moreover:
  \item If $\vSs\cap\bt_\alpha\neq\emptyset$, then
    $t_\alpha(T)\neq 0$.
  \item If $\vSs\cap\bt_\alpha=\emptyset$, then
    $t_\alpha(T)=0$.
  \end{enumerate}
\end{lemma}
\begin{proof}
  (1) By Theorem~\ref{thm:tilting-from-resolving} there is a tilting
  sheaf $T\in\QHh$ with $\rperpe{\vSs}=\rperpe{T}$. Moreover, there is
  an exact sequence~\eqref{eq:Tcc-preenvelope} with $T=T_0\oplus T_1$,
  and by Remark~\ref{rem:filtration} the summands $T_0$ and $T_1$ are
  $\vSs$-filtered. Since $\Cc_\alpha=\lperpe{\bq_{\alpha}}$ is closed
  under filtered direct limits (which follows
  from~\cite[Prop.~2.12]{saorin:stovicek:2011}), we get
  $T_0,\,T_1\in\Cc_\alpha$, thus $T$ is in $\Cc_\alpha$.

  (2) Assume that $t_\alpha(T)=0$. Let $S\in\vSs\cap\bt_\alpha$ be
  indecomposable. Then $\Ext^1(T,S)=\D\Hom(\tau^- S,T)=0$, that is,
  $S\in\rperpe{T}$. For every $X\in\rperpe{T}=\rperpe{\vSs}$ we have
  $\Ext^1(S,X)=0$, and thus $S\in\lperpe{(\rperpe{T})}$. Since, by
  Lemma~\ref{lem:porperties-tilting},
  $\rperpe{T}\cap\lperpe{(\rperpe{T})}=\Add(T)$ we get $S\in\Add(T)$,
  and then $S$ is a summand of $t_\alpha(T)$, contradiction. Thus
  $t_\alpha(T)\neq 0$.

  (3) Assume that $t_\alpha(T)\neq 0$. By
  Lemma~\ref{lem:torsion-implies-slope} then $T$ has slope $\alpha$,
  so $\Gen(T)\subseteq\Bb_\alpha$, and we even have equality since
  $\vSs\subseteq\bp_{\alpha}$. So $T$ is torsionfree by
  Lemma~\ref{lem:L_w-torsionfree}, contradiction.
\end{proof}
The main result of this section is the following.
\begin{theorem}\label{thm:every-large-ts-slope}
  Every large tilting sheaf $T$ has a slope $w\in\widehat{\mathbb{R}}$.
\end{theorem}
\begin{proof}
  Let $\Bb=\Gen(T)=\rperpe{T}$ and $\vSs=\lperpe{\Bb}\cap\Hh$. Define
  $w=\inf\{r\in\widehat{\RR}\mid T\in\Cc_r\}\in\widehat{\RR}$. This is
  well-defined. We show that $T$ has slope $w$. By properties of the
  infimum we have $T\in\Cc_w$, but $T\not\in\Cc_v$ for all $v<w$. We
  have to show that $T\in\Bb_w$. By the preceding lemma $T$ is of
  finite type, in other words, $\rperpe{T}=\rperpe{\vSs}$. For every
  rational number $\alpha<w$ let $$\vSs_\alpha=\vSs\cap\Cc_\alpha.$$
  Since $\vSs$ is strongly resolving by
  Lemma~\ref{lem:large-tilting-leftbounded}, the same holds for
  $\vSs_{\alpha}$. Since $T\not\in\Cc_\alpha$, the set of all rational
  numbers $\alpha<w$ with $\vSs\cap\bt_\alpha\neq\emptyset$ is not
  bounded by a smaller number than $w$; this follows from
  Lemma~\ref{lem:TinC} and since $T$ is determined by $\vSs$. Thus
  there is a sequence of rational
  numbers $$\alpha_1<\alpha_2<\alpha_3<\dots <w$$ with
  $\lim_{i\to\infty}\alpha_i=w$ and
  \begin{equation}
    \label{eq:non-empty}
    \vSs\cap\bt_{\alpha_i}\neq\emptyset.
  \end{equation}
  By Lemma~\ref{lem:TinC} there is a tilting object $T_i$ with
  $\rperpe{{T_i}}=\rperpe{{\vSs_{\alpha_i}}}$ and
  $T_i\in\Cc_{\alpha_i}$ and with $t_{\alpha_i}(T_i)\neq 0$.  Now, by
  Lemma~\ref{lem:torsion-implies-slope} the tilting object $T_i$ has
  slope $\alpha_i$. Then we get
  $\Gen(T_i)\subseteq\Gen(\bL_{\alpha_i})=\Bb_{\alpha_i}$ (the largest
  tilting class of slope $\alpha_i$). Since
  $\vSs_{\alpha_i}\subseteq\vSs$, we
  get $$\Bb_{\alpha_i}\supseteq\rperpe{{\vSs_{\alpha_i}}}\supseteq\rperpe{\vSs}\ni
  T$$ for all $i$, and thus $T\in\bigcap_{i\geq
    1}\Bb_{\alpha_i}=\Bb_w$.
\end{proof}
\begin{theorem}\label{thm:tubular-full-classification}
  Let $\XX$ be a noncommutative curve of genus zero of tubular type.
  \begin{enumerate}
  \item The sheaves $\bL_w$ with $w\in\widehat{\RR}$ are, up to
    equivalence, the unique torsionfree large tilting
    sheaves (in the sense of
  Definition~\ref{def:torsionfree-tilting}).\smallskip
\item The equivalence classes of large non-torsionfree tilting sheaves
  are in bijective correspondence with triples $(\alpha,B,V)$, where
  $\alpha\in\widehat{\QQ}$, $V\subseteq\XX_\alpha$ and
  $B\in\add\bt_\alpha$ is a branch sheaf, and
  $(B,V)\neq (0,\emptyset)$.
  \end{enumerate}
\end{theorem}
\begin{proof}
  (1) Let $T$ be a torsionfree tilting sheaf of slope $w$. Then
  $\rperpe{T}\subseteq\Bb_w=\rperpe{{\bL_w}}$. Hence we have
  $\lperpe{(\rperpe{{\bL_w}})}\cap\Hh=\add(\bp_w)
  =\lperpe{(\rperpe{T})}\cap\Hh$; the last equality follows, since $T$
  generates every sheaf of finite length. Now
  $\rperpe{T}=\rperpe{{\bL_w}}$ follows from
  Proposition~\ref{prop:tilting-sheaves-finite-type-tub}.

  (2) Every large non-torsionfree tilting sheaf $T$ has a slope
  $\alpha\in\widehat{\QQ}$. By Lemma~\ref{lem:tilting-corr-tubular},
  $T$ is a large tilting sheaf in $\QHh\spitz{\alpha}$, having a
  non-zero torsion part $t_{\alpha}(T)$. We now apply
  Theorems~\ref{thm:classif-tubular-torsion}
  and~\ref{thm:full-classif-domestic} to the category
  $\QHh\spitz{\alpha}$. The non-torsionfree tilting sheaves of slope
  $\alpha$ are given by
  \begin{itemize}
  \item $T_{(B,V)}$ with $\emptyset\neq V\subseteq\XX$ (here
    $t_{\alpha}(T)$ is non-coherent);
  \item $\bL'\oplus B$, with $0\neq B\in\add\bt_{\alpha}$ a branch sheaf
    and $\bL'\in\rperp{B}=\Qcoh\XX'_{\alpha}$ the Lukas tilting sheaf
    over the domestic curve $\XX'_{\alpha}$ (here $t_{\alpha}(T)=B$ is
    coherent).
  \end{itemize}
This finishes the proof.
\end{proof}
We say that a resolving class $\vSs\subseteq\Hh$ \emph{has slope} $w$
if $\bp_w\subseteq\vSs$ and $\vSs$ does not contain any indecomposable
of slope $\beta>w$.
\begin{corollary}\label{restub}
  For a tubular curve $\XX$, the complete list of the resolving
  classes $\vSs$ in $\Hh=\coh\XX$ having a slope is given by
  \begin{itemize}
  \item $\add\bp_w$ with $w\in\widehat{\RR}$; and \smallskip
  \item
    $\add\,(\,\bp_{\alpha}\cup\tau^- (B^{>})\cup\bigcup_{x\in
      V}\{\tau^j S_x[n]\mid j\in \Rr_x,\,n\in\NN\}\,)$
    with $\alpha\in\widehat{\QQ}$, $V\subseteq\XX_{\alpha}$,
    $B\in\add\bt_\alpha$ a branch sheaf, and $(B,V)\neq
    (0,\emptyset)$. 
  \end{itemize}
\end{corollary}
\begin{proof}
  By Theorem~\ref{thm:tubular-full-classification}, the list contains
  precisely the resolving classes corresponding to the large tilting
  sheaves under the bijection of
  Theorem~\ref{thm:class-correspondence}, and they all have a
  slope. Conversely, let $\vSs$ be resolving having a slope $w$ and
  $T$ a tilting sheaf such that $\rperpe{T}=\rperpe{\vSs}$. If $w$ is
  irrational, then $\vSs=\add\bp_w$. If, on the other hand,
  $w=\alpha\in\widehat{\QQ}$, then
  $\Add(T)\cap\Hh=\vSs\cap \rperpe{\vSs}\subseteq\add
  (\bp_{\alpha}\cup
  \bt_{\alpha})\cap\rperpe{{\bp_{\alpha}}}\subseteq\add\bt_\alpha$,
  that is, all coherent summands of $T$ belong to the same tubular
  family, and therefore $T$ cannot be coherent.
\end{proof}
\begin{corollary}[Property~(TS3)]\label{TS3tub}
  Let $T_{\can}$ be the canonical tilting bundle. Let $T\in\QHh$ be a
  large tilting sheaf. Then for any homogeneous point $x_0$ and
  $n\gg 0$ there is a short exact sequence
  $$0\ra T_{\can}(-nx_0)\ra T_0\ra T_1\ra 0$$ with
  $\add(T_0\oplus T_1)=\add(T)$.
\end{corollary}
\begin{proof}
  If $T$ has slope $w$, choose $n\gg 0$ such that all indecomposable
  summands of $T_{\can}(-nx_0)$ have slope smaller than $w$.
\end{proof}

\section{The elliptic case}\label{sec:elliptic}   
The tubular case, where all indecomposable coherent sheaves lie in tubes,
is very similar (but weighted) to Atiyah's classification of
indecomposable vector bundles over elliptic
curves~\cite{atiyah:1957}. There are even more affinities between
elliptic and tubular curves (we refer
to~\cite{chen:chen:zhenqiang:2014}). After the treatment of the
tubular case, it is thus  natural to investigate  tilting sheaves
over elliptic curves. Since these curves are non-weighted, that
is, do not have exceptional tubes, there is \emph{no} coherent tilting
sheaf. But there are  large tilting sheaves.\medskip

Let $(\Hh,L)$ be a non-weighted, noncommutative regular projective
curve over the field $k$ with structure sheaf $L$. Assume that $\Hh$ (or $\XX$) is 
elliptic, that is,  its genus  $g(\XX)=1$, or equivalently,
$\chi'(\XX)=0$.  Examples are the
``classical'' (commutative) elliptic curves over an algebraically
closed field, and real elliptic
curves like the Klein bottle and the M\"obius band
(\cite{kussin:2014}). 

With the degree
$\deg(F)=\frac{1}{\kappa\eps}\LF{L}{F}$ for $F\in\Hh$, we have a
Riemann-Roch formula~\cite{kussin:2014}, and for the slope $\mu$ we obtain:
if $X,\,Y\in\Hh$ are indecomposable, then
\begin{equation}
  \label{eq:RR-Homs}
   \mu(X)<\mu(Y)\quad\Rightarrow\quad\Hom(X,Y)\neq 0.
\end{equation}
We recall the following result  from
Theorem~\ref{thm:stability}, similar to Atiyah's
classification~\cite{atiyah:1957}. 
\begin{proposition}[\cite{kussin:2014}]\label{prop:elliptic}
  Let $\Hh=\coh\XX$ be a noncommutative elliptic curve. Then the
  following holds:
  \begin{enumerate}
  \item[(1)] Every indecomposable object $E$ in $\Hh$ is semistable and
    satisfies $\Ext^1(E,E)\neq 0$, and $\tau E\simeq E$.\smallskip
  \item[(2)] For every $\alpha\in\widehat{\QQ}$ the subcategory of
    semistable objects of slope $\alpha$ is non-trivial and forms a
    family of homogeneous tubes, again parametrized by a
    noncommutative elliptic curve $\XX_{\alpha}$, which is
    derived-equivalent to $\XX$. \qed
  \end{enumerate}
\end{proposition}
 Like in the tubular case, we consider the
interval categories $\QHh\spitz{\alpha}$ for every rational $\alpha$.
For each $\QHh\spitz{\alpha}$,
Corollary~\ref{cor:general-homogeneous-torsion} yields tilting sheaves
$T_{\alpha,V}$ of slope $\alpha$ with non-zero torsion part supported
in $\emptyset\neq V\subseteq\XX_{\alpha}$.
\begin{theorem}\label{thm:elliptic-tilting}
  Let $\QHh=\Qcoh\XX$ be the category of quasicoherent sheaves over a
  noncommutative elliptic curve.
  \begin{enumerate}
  \item[(1)] Every tilting sheaf in $\QHh$ has a slope
  $w\in\widehat{\RR}$.\smallskip
  \item[(2)] For every $w\in\widehat{\RR}$ there is  a tilting
  sheaf $\bL_w$ with $\rperpe{{\bL_w}}=\Bb_w$ which is torsionfree (in the sense of
  Definition~\ref{def:torsionfree-tilting}).\smallskip
  \item[(3)] For every $\alpha\in\widehat{\QQ}$ and every non-empty
    $V\subseteq\XX_{\alpha}$ there is, up to equivalence, precisely
    one tilting sheaf $T$ of slope $\alpha$ with $t_{\alpha}(T)$
    supported in $V$, namely $T=T_{\alpha,V}$.\smallskip
  \item[(4)] Every tilting sheaf of finite type is  equivalent
    to one listed in (2) or (3).\smallskip
  \item[(5)] The  resolving subclasses of $\Hh$ are given precisely by 
      $\add\bp_w$ with $w\in\widehat{\RR}$, and
    $\add\bigl(\,\bp_{\alpha}\cup\bigcup_{x\in V}\bt_{\alpha,x}\,\bigr)$ with
    $\alpha\in\widehat{\QQ}$ and
    $\emptyset\neq V\subseteq\XX_{\alpha}$.
  \end{enumerate}
\end{theorem}
\begin{proof}
  (1), (2), (3) We show that every tilting sheaf $T\in\QHh$ has a
  slope $w\in\widehat{\RR}$. To this end, let
  $w=\inf\{r\in\widehat{\RR}\mid T\in\Cc_r\}\in\widehat{\RR}$. We
  assume first that $w$ is rational; then without loss of generality
  $w=\infty$. If $tT\neq 0$, then $T$ is by
  Corollary~\ref{cor:general-homogeneous-torsion} of the form $T_V$
  with $\emptyset\neq V\subseteq\XX$ (in particular, we also have
  uniqueness in this case). Let now $T$ be torsionfree. Then
  $\vect\XX\subseteq\lperpe{T}$. Indeed, otherwise one finds a line
  bundle $L'$, say of slope $\alpha<\infty$, such that $T$ maps onto
  $L'$. By~\eqref{eq:RR-Homs}, $L'$ maps non-trivially to \emph{each}
  vector bundle of slope $>\alpha$. Since, by torsionfreeness, all
  simple sheaves lie in $\Gen(T)$, it follows  that all line
  bundles of slope $>\alpha$ lie in $\Gen(T)$. Let $E$ be an
  indecomposable vector bundle of slope $>\alpha$. Then $L'$ is a
  subsheaf of $E$, and we find a line bundle $L''$ with
  $L'\subseteq L''\subseteq E$ such that $E'=E/L''$ is torsionfree,
  thus a line bundle. Since $\rk(E')=\rk(E)-1$ and
  $\mu(E')\geq\mu(E)>\alpha$ we see by induction that every
  indecomposable vector bundle of slope $>\alpha$ lies in
  $\Gen(T)=\rperpe{T}$. By Serre duality $\Hom(\bq_{\beta},T)=0$ for
  all rational $\beta$ with $\alpha<\beta<\infty$. But then
  $T\in\Cc_{\alpha}$, which gives a contradiction to the choice of $w$
  ($=\infty$). It follows that $T$ has slope $\infty$, moreover
  $\vSs:=\lperpe{(\rperpe{T})}\cap\Hh=\vect\XX$.

  Let now $w$ be irrational. We have to show
  $T\in\rperpe{{\bp_w}}$. Otherwise, there is a rational $\alpha<w$
  such that $\Hom(T,\bt_{\alpha})\neq 0$. Considering images, we can
  assume with loss of generality that there is an epimorphism in this
  set. Then it is easy to see that there is $x\in\XX_{\alpha}$ such
  that $T$ generates a tube $\bt_{\alpha,x}$. Then it follows like
  in~\cite[Rem.~13.3]{reiten:ringel:2006}, that $T$ generates all
  coherent objects $E$ of all rational slopes $\beta$ with
  $\alpha<\beta\leq\infty$. But this means, by Serre duality, that for
  all those $E$ we have $\Hom(E,T)=0$, and thus
  $T\in\Cc_{\alpha}$. This is a contradiction to the choice of $w$. We
  conclude $T\in\rperpe{{\bp_w}}=\Bb_w$, and $T$ has slope $w$. (We
  remark that this argument for irrational $w$ also applies to the
  torsionfree case when $w$ is rational.)

  Finally, for every $w\in\widehat{\RR}$ there is a torsionfree
  tilting sheaf $\bL_w$. Indeed, $\vSs=\add\bp_w$ generates $\QHh$ and
  is thus resolving. The claim now follows from
  Theorem~\ref{thm:tilting-from-resolving}.\medskip

  (4) Let $T$ be tilting of finite type, $\rperpe{T}=\rperpe{\vSs}$
  for some $\vSs\subseteq\Hh$ which we choose  as
  $\vSs=\lperpe{(\rperpe{T})}\cap\Hh$. By~(1), $T$ has a slope $w$. If
  $T$ has a non-trivial torsion part, then $T$ is equivalent to a tilting sheaf in~(3) by
  Corollary~\ref{cor:general-homogeneous-torsion}. So we assume
  that $T$ is torsionfree. Since a coherent object $X$ is in $\vSs$ if
  and only if $\Ext^1(X,T)=0$, we have $\bp_w\subseteq\vSs$: indeed,
    $\Ext^1(\bt_{\alpha},T)=\D\Hom(T,\bt_{\alpha})=0$ for all rational $\alpha<w$ 
by slope
  reasons. 
  Furthermore, if $X\in\bq_w$, then $\Ext^1(T,\tau X)=0$
  as $T\in\Cc_w=\lperpe{{\bq_w}}$,
  so $\Ext^1(X,T)\simeq\D\Hom(T,\tau X)\neq 0$, and $X\not\in\vSs$.
 Finally, in case $w\in\widehat{\QQ}$, it follows as in Lemma~\ref{lem:TinC} that
   $\vSs\subseteq\Cc_w\cap\Hh$ satisfies $\vSs\cap\bt_w=\emptyset$.
   We thus conclude
  $\vSs=\add\bp_w$, and  $T$ is equivalent to the tilting sheaf $\bL_w$ from~(2).\medskip

  (5) Using~(4), the claim follows from
  Theorem~\ref{thm:class-correspondence} and
  Lemma~\ref{lemma:undercut}.
\end{proof}

\section{Combinatorial descriptions and an example}\label{sec:TS3}
\emph{Let $\XX$ be a noncommutative curve of genus zero, of arbitrary
  weight type.} In this section we further investigate the large
tilting sheaves $T_{(B,V)}$ with $V\neq\emptyset$. We already know
that they are of finite type and satisfy condition (TS3). We give an
explicit construction for the sequence in (TS3), and we verify the
stronger property~(TS3+).

We denote by $\Lambda$ a canonical tilting bundle $T_{\can}$, as in
Remark~\ref{rem:canonical-subconfiguration}. By copresenting each
indecomposable summand of $T_{\can}$ by summands of $T_{(B,V)}$ we
will prove the following.
\begin{theorem}\label{thm:axiom-TS3-for-T_VB}
  Let $\XX$ be of genus zero and $T=T_{(B,V)}$ as
  in~\eqref{eq:def_full-T_V,B}. The canonical configuration
  $T_{\can}=\Lambda$ has an $\add(T)$-copresentation as follows:
  \begin{equation}
    \label{eq:TS3-copresentation}
    0\ra T_{\can}\ra T'_0 \oplus B_0\ra T'_1\oplus B_1\ra 0
  \end{equation}
  with $T'_0\in\add(\Lambda'_{V})$ torsionfree,
  $T'_1\in\add(\bigoplus_{x\in V}\bigoplus_{j\in \Rr_x}\tau^j
  S_x[\infty])$ and $B_0,\,B_1\in\add(B)$ such that $\Hom(B_1,B_0)=0$;
  moreover, in $T'_1$ all Pr\"ufer summands $\tau^j S_x[\infty]$ of $T$
  occur and $\add(B_0\oplus B_1)=\add(B)$.
\end{theorem}
As a first preparation we have the following simple fact.
\begin{lemma}\label{lem:subbranch}
  Let $B$ be a connected branch and $B'$ a proper subbranch of $B$,
  rooted in $Z\in B$. Then one of the following two cases holds.
  \begin{enumerate}
  \item[(1)] There is an epimorphism $X\ra Z$ with
    $X\in B\setminus B'$, and then there is no non-zero morphism from
    $B'$ to $B\setminus B'$.
  \item[(2)] There is a monomorphism $Z\ra Y$ with
    $Y\in B\setminus B'$, and then there is no non-zero morphism from
    $B\setminus B'$ to $B'$.
  \end{enumerate}
\end{lemma}
\begin{proof}
  Since $B'$ is proper, it is clear that there is either an
  epimorphism $X\ra Z$ or a monomorphism $Z\ra Y$ with $X$ or $Y$ in
  $B\setminus B'$, respectively. Let $\Ww'$ be the wing rooted in
  $Z$. Since $B'$ forms a tilting object in $\Ww'$, it is clear, that
  $\Ww'$ is disjoint with $B\setminus B'$. Let $U\in B'$ and
  $V\in B\setminus B'$. Assume the first case, and $\Hom(U,V)\neq 0$.
  Then $V$ lies on a ray starting in the basis of $\Ww'$, but not in
  $\Ww'$. We then get $\Hom(X,\tau V)\neq 0$. By Serre duality we get
  $\Ext^1(V,X)\neq 0$, which gives a contradiction because of
  $\Ext^1(B,B)=0$. The second case follows similarly.
\end{proof}
\begin{numb}
  Let $T=T_{(B,V)}$ be a given large tilting sheaf with
  $V\neq\emptyset$. For the moment we assume, for notational
  simplicity, that $B$ is an \emph{inner} branch sheaf. Let us explain
  the strategy we are going to pursue for the proof of the
  theorem.\medskip

\noindent\underline{Step~1}: \emph{Initial step.} We start with the canonical
configuration $\Lambda=T_{\can}$ in $\QHh=\Qcoh\XX$, which consists of
arms between $L$ and $\overline{L}$, compare~\eqref{eqn:can-conf}. By
applying suitable tubular shifts to $\Lambda$, we can assume without
loss of generality that $\Hom(L,B)=0=\Hom(\overline{L},B)$. We then
form $\QHh'=\Qcoh\XX'=\rperp{(\tau^-B)}$. Then the subconfiguration
$\Lambda'$ of indecomposable summands of $\Lambda$ lying in
$\rperp{(\tau^- B)}$ forms a canonical configuration
$\Lambda'=T'_{\can}$ in $\QHh'$, containing $L$ and
$\overline{L}$, compare Remark~\ref{rem:canonical-subconfiguration}. 
For $\Lambda'$ we have the copresentation
\begin{equation}
  \label{eq:copres-Lambda-prime}
  0\ra\Lambda'\ra
  \Lambda'_{V}\ra\bigoplus_{x\in V}\bigoplus_{j=0}^{p'\!(x)-1}(\tau^j
  S_x[\infty])^{e(j,x)}\ra 0.
\end{equation}
from~\eqref{eq:def-Lambda_V}, which is already of the desired form
with respect to the theorem we want to prove; by construction, it
gives an $\add(T)$-copresentation of each indecomposable summand of
$\Lambda'$. It remains to compute suitable copresentations for each
indecomposable summand of $\Lambda$ not in $\Lambda'$, and then to
take the direct sum of all of these sequences
with~\eqref{eq:copres-Lambda-prime}. This will be done inductively
working in each connected branch component, starting with the root of
that component. Let us consider one such component lying in a wing
$\Ww$ rooted in, say, $S[r-1]$ with $2\leq r\leq p$, concentrated in a
point $x\in V$.
We will call $S[\infty]$ the Pr\"ufer sheaf \emph{above} $\Ww$. \medskip

\noindent\underline{Step~2}: \emph{Induction start with root.}  Note that
$S[r]\in\rperp{(\tau^- B)}\simeq\Qcoh\XX'$ becomes simple. The basis
of $\Ww$ is given by the simple sheaves
$S,\tau^- S,\dots,\tau^{-(r-2)}S$. This segment of simples corresponds
to a segment of direct summands of $\Lambda'$ lying in the inner of
one arm. We denote this segment by $L(1),\dots,L(r-1)$, so that there
are epimorphisms
\begin{equation}
  \label{eq:L(i)-onto}
  L(i)\twoheadrightarrow\tau^{-i+1}S\quad\quad i=1,\dots,r-1.  
\end{equation}
(We will do this for every branch component, and then we will need, of
course, a shift of indices. In the notation of~\eqref{eqn:can-conf}
the segment $L(1),\dots,L(r-1)$ is $L_i(j),\dots,L_i(j+r-2)$ for some
arm-index $i$ and some $j$.) With this ``calibration'' the
sequence~\eqref{eq:sesA0} becomes
\begin{equation}
  \label{eq:w-root-sequence}
  0\ra L(0)\ra
  L(r-1)\ra S[r-1]\ra 0 
\end{equation}
where $L(0)$ is a predecessor of $L(1)$, either still in the inner of
the same arm, or $L(0)=L^{\varepsilon f(x)}$; in any case
$L(0)\in\add(\Lambda')$. This means that for $L(0)$ we already have a
copresentation. Using Lemma~\ref{lem:branch-type-2} below, we will get a
copresentation for $L(r-1)$, which will be compatible with 
the statement of our theorem. 

We will then proceed in a similar way with the other
members of the connected branch $B$, going down the branch
inductively, as described in the next step.\medskip

\noindent\underline{Step~3}: \emph{Induction step.} We introduce
further notation. We define $$W_{ij}=S[i]/S[i-j]\in\Ww$$ for
$i=1,\dots,r-1;\,j=1,\dots,i$, where $S[0]=0$. We call $W_{ij}$
\emph{wing objects}, and the pair of indices $(i,j)$ \emph{wing
  pairs}. The length of $W_{ij}$ is $j$; we say that $i$ is the
\emph{level} and $i-j$ the \emph{colevel} of $W_{ij}$. So $W_{ij}$ is
uniquely determined by its level and colevel, which fix the ray and
coray $W_{ij}$ belongs to. Applying the construction of an
$\add(L)$-couniversal extension to the short exact sequences
$0\ra W_{jj}\ra W_{ii}\ra W_{i,i-j}\ra 0$, and recalling that we have
$\Hom(L,\Ww)=0$, we deduce
from~\cite[Prop.~5.1]{lenzing:delapena:1999} that there are short
exact sequences
\begin{equation}
  \label{eq:Wij-sequences}
  0\ra L(j)\ra
  L(i)\ra W_{i,i-j}\ra 0 
\end{equation}
for $1\leq j<i$. We assume now that $W_{i,i-j}$ be part of $B$. The
(direct) neighbours of smaller length in the same component of the
branch might be
$$
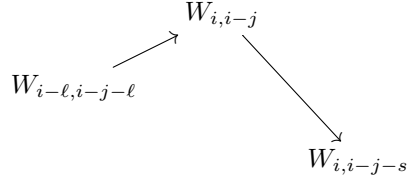
\xy\xymatrixcolsep{1pc}\xymatrixrowsep{1pc}\xymatrix{ & W_{i,i-j}
  \ar
  @{->}[rdd] & \\
  W_{i-\ell,i-j-\ell}\ar @{->}[ur] & & \\
  & & W_{i,i-j-s}}\endxy$$
where $W_{i-\ell,i-j-\ell}\ra W_{i,i-j}$ denotes a composition of
$\ell$ irreducible monomorphisms and $W_{i,i-j}\ra W_{i,i-j-s}$ a
composition of $s$ irreducible epimorphisms. In this situation we
compute an $\add(T)$-copresentation of $L(i-\ell)$ and $L(j+s)$,
respectively, if $\add(T)$-copresentations of $L(j)$ or $L(i)$,
respectively, are already known. In other words: having already
exploited $W_{i,i-j}$ for computing a suitable copresentation of an
indecomposable summand of $\Lambda$, we will then use its lower
neighbours for computing copresentations for further summands. The two
different kinds of neighbours are reflected by the following two
lemmas. Roughly speaking, the first lemma (treating the epimorphism
case) adds the branch summand $W_{i,i-j-s}$ to the end term, the
second (treating the monomorphism case) the branch summand
$W_{i-\ell,i-j-\ell}$ to the middle term in the copresentation of
$\Lambda$.
\end{numb}
\begin{lemma}\label{lem:branch-type-1}
  Let $(i,i-j)$ and $(i,i-j-s)$ be wing pairs and assume that
  $W_{i,i-j}$ and $W_{i,i-j-s}$ are summands of $B$. Assume there is an exact
  sequence $$0\ra L(i)\ra\Gg\oplus B_0\ra
  P\oplus B_1\ra 0$$ such that
  \begin{enumerate}
  \item[(i)] $B_0,\,B_1\in\add (B)$ are disjoint with the subbranch rooted in
    $W_{i,i-j-s}$;
      \item[(ii)] $\Hom(B_1,B_0)=0$;
  \item[(iii)] $\Gg$ is torsionfree and $x$-divisible;
  \item[(iv)] $P$ is a direct sum of copies of the
    Pr\"ufer sheaf $S[\infty]$ above the wing $\Ww$.
  \end{enumerate}
  Then there is an exact sequence $$0\ra L(j+s)\ra \Gg\oplus B_0\ra
  P\oplus W_{i,i-j-s}\oplus B_1\ra 0$$ with $\Hom(W_{i,i-j-s},B_0)=0$.
\end{lemma}
\begin{proof}
  The sequence $$0\ra L(j+s)\ra L(i)\ra W_{i,i-j-s}\ra 0$$ together with
  the given sequence yields the exact commutative diagram
  $$\xy\xymatrixcolsep{1pc}\xymatrixrowsep{1pc}\xymatrix{ & & 0\ar
    @{->}[d] & 0\ar @{->}[d] & \\
    0\ar @{->}[r] & L(j+s)\ar @{=}[d] \ar @{->}[r] & L(i)\ar @{->}[d]
    \ar
    @{->}[r] & W_{i,i-j-s} \ar @{->}[d] \ar @{->}[r] & 0\\
    0\ar @{->}[r] & L(j+s)\ar @{->}[r] & \Gg\oplus B_0\ar @{->}[d] \ar
    @{->}[r] & C \ar @{->}[d] \ar @{->}[r] & 0\\
    & & P\oplus B_1\ar @{=}[r]\ar @{->}[d] & P\oplus B_1\ar @{->}[d]
    & \\
    & & 0 & 0 & }\endxy$$
  The right column splits, since $W_{i,i-j-s}$ and $B_1$ as summands
  of the branch $B$ are Ext-orthogonal, and since
  $\Ext^1(P,W_{i,i-j-s})=\D\Hom(\tau^- W_{i,i-j-s},P)=0$. Because
  of~(i) we get $\Hom(W_{i,i-j-s},B_0)=0$ from
  Lemma~\ref{lem:subbranch}.
\end{proof}
\begin{lemma}\label{lem:branch-type-2}
  Let $(i,i-j)$ and
  $(i-\ell,i-j-\ell)$ be wing pairs such that $W_{i,i-j}$ and
  $W_{i-\ell,i-j-\ell}$ are summands of $B$ (the case $\ell=0$ is
  permitted). Assume there is an exact sequence
  $$0\ra L(j) \ra \Gg\oplus B_0\ra P\oplus B_1 \ra 0$$ such that
  $B_0,\,B_1\in\add (B)$ are disjoint from the subbranch rooted in
  $W_{i-\ell,i-j-\ell}$, $\Hom(B_1,B_0)=0$,   $\Gg$ is
  torsionfree and $x$-divisible, and $P$ is a direct sum of copies  of the Pr\"ufer
  sheaf above the wing $\Ww$. Then there is an exact sequence
  $$0\ra L(i-\ell)\ra\Gg\oplus B_0\oplus W_{i-\ell,i-j-\ell}\ra
  P\oplus B_1\ra 0,$$ and $\Hom(B_1,W_{i-\ell,i-j-\ell})=0$.
\end{lemma}
\begin{proof}
  There is the push-out diagram
  $$\xy\xymatrixcolsep{1pc}\xymatrixrowsep{1pc}\xymatrix{ & 0\ar
    @{->}[d] & 0\ar @{->}[d] & &\\
    0\ar @{->}[r] & L(j) \ar @{->}[r] \ar @{->}[d] & \Gg\oplus B_0\ar
    @{->}[r] \ar @{->}[d] & P\oplus B_1 \ar
    @{->}[r] \ar @{=}[d] & 0\\
    0\ar @{->}[r] & L(i-\ell)\ar @{->}[r] \ar @{->}[d] & E \ar
    @{->}[r] \ar @{->}[d] & P\oplus B_1\ar
    @{->}[r] & 0\\
    & W_{i-\ell,i-j-\ell}\ar @{=}[r] \ar @{->}[d] &
    W_{i-\ell,i-j-\ell}\ar @{->}[d] & & \\
    & 0 & 0 & & }\endxy$$
  Now, since $\Gg$ is $x$-divisible and $W_{i-\ell,i-j-\ell}$ and
  $B_{0}$ as summands of the branch $B$ are Ext-orthogonal, the middle
  column splits. Moreover, $\Hom(B_1,W_{i-\ell,i-j-\ell})=0$ follows
  again from Lemma~\ref{lem:subbranch}.
\end{proof}
\begin{numb}\label{nr:exterior-treatment}
  Let now $B$ be an \emph{exterior} branch part, the inner branch
  parts already treated. We proceed similarly to the inner case. We
  briefly explain the differences. By applying suitable tubular shifts
  to $\Lambda$, we can assume without loss of generality that
  $\Hom(L,\tau B)=0=\Hom(\overline{L},\tau B)$. We then form
  $\QHh'=\Qcoh\XX'=\rperp{B}$. Then the subconfiguration $\Lambda'$ of
  indecomposable summands of $\Lambda$ lying in $\rperp{B}$ forms a
  canonical configuration $\Lambda'=T'_{\can}$ in $\QHh'$, containing
  $L$ and $\overline{L}$. Note that
  $\tau S[r]\in\rperp{B}\simeq\Qcoh\XX'$ becomes simple. The basis of
  a wing $\Ww$ corresponding to a connected component of $B$ is given
  by the simple sheaves (concentrated in $x$)
  $S,\tau^- S,\dots,\tau^{-(r-2)}S$. This segment of simples
  corresponds to a segment of direct summands of $\Lambda'$ lying in
  the inner of one arm. We denote this segment by $L(1),\dots,L(r-1)$,
  so that there are epimorphisms
\begin{equation}
  \label{eq:L(i)-onto-2}
  L(i)\twoheadrightarrow\tau^{-i+2}S\quad\quad i=1,\dots,r-1.  
\end{equation}
This yields a short exact sequence
\begin{equation}
  \label{eq:w-root-sequence-2}
  0\ra L(1)\ra
  L(r)\ra S[r-1]\ra 0 
\end{equation}
where $L(r)$ is either in the inner of the same arm, or
$L(r)=\overline{L}^{f(x)}$. (We refer to the diagram
in~\cite[p.~536]{lenzing:delapena:1999}.) Thus the desired
copresentation of $L(r)$ is already given. Then, for $L(1)$ and for
the induction step we have modified versions of
Lemma~\ref{lem:branch-type-1} and~\ref{lem:branch-type-2}, just taking
into account the different notation~\eqref{eq:L(i)-onto-2}.
\end{numb}
\begin{proof}[Proof of Theorem~\ref{thm:axiom-TS3-for-T_VB}]
  Let $\Lambda$ be a given canonical configuration, considered as full
  subcategory of $\Hh$. As usual we write
  $B=B_{\mathfrak{i}}\oplus B_{\mathfrak{e}}$ with respect to $V$. We
  can assume that the canonical configuration $\Lambda'$ in
  $\rperp{(\tau^- B_{\mathfrak{i}}\oplus
    B_{\mathfrak{e}})}\simeq\Qcoh\XX'$
  is a subconfiguration of $\Lambda$, containing $L$ and
  $\overline{L}$.  Recall that $B$ decomposes into
  $B=\bigoplus_{i=1}^t B_{x_i}$ over the exceptional points
  $x_1,\ldots,x_t$, and each $B_{x_i}$ (in case it is nonzero) is a
  direct sum of finitely many connected branches in non-adjacent
  wings.  Then
  $$\Lambda=\Lambda'\oplus\bigoplus_{i=1}^t \bigoplus_{\ell} L_i(\ell)$$
  for suitable $\ell$, forming finitely many non-adjacent segments in
  $\{1,\dots,p_i-1\}$, corresponding to the connected branches as
  described above.
  
  Step~1  yields the $\add(T)$-copresentation
\begin{equation}
  \label{eq:initial-copresentation}
  0\ra\Lambda'\ra T'_0\ra T'_1\ra 0
\end{equation}
of $\Lambda'$, given by~\eqref{eq:copres-Lambda-prime}. We then have
to compute suitable copresentations for the $L_i(\ell)$. By
forming the direct sum we will get the desired copresentation for
$\Lambda$. This can be done separately by performing Step~2 and Step~3
for every connected branch (using Lemma~\ref{lem:branch-type-1}
and~\ref{lem:branch-type-2} and keeping in mind the modifications
in~\ref{nr:exterior-treatment}). \medskip

We still have to show that in this way we obtain
$\add(T)$-copresentations of \emph{all} indecomposable summands of
$\Lambda$. It is enough to do this for every single wing $\Ww$
involved, say $\Ww$ is rooted in $S[r-1]$, and the corresponding
summands of $\Lambda$ are given by $L(1),\dots,L(r-1)$. (So this
notation applies to the inner case, the exterior is treated
similarly.) The kernel of the epimorphism
$L(r-1)\ra S[r-1]=W_{r-1,r-1}$ is (a power of) an indecomposable
summand of $\Lambda'$, and from Lemma~\ref{lem:branch-type-2} (case
$\ell=0$) we get an $\add(T)$-copresentation of $L(r-1)$.  Let
$W_{ij}$ be a summand of $B$, different from the root $S[r-1]$. Then
$W_{ij}$ has a unique upper neighbour $Z$ in $B$. There are two cases:
\begin{enumerate}
\item[(a)] There is an epimorphism $Z\ra W_{ij}$. Then
  Lemma~\ref{lem:branch-type-1} gives a copresentation of $L(i-j)$
  where $i-j$ is the colevel of $W_{ij}$.
\item[(b)] There is a monomorphism $W_{ij}\ra Z$. Then
Lemma~\ref{lem:branch-type-2} gives a copresentation of $L(i)$ where
$i$ is the level of $W_{ij}$. 
\end{enumerate}
So either the level or the colevel determines the index of the summand
of $\Lambda$ we can treat with the help of $W_{ij}$. In both cases the
obtained index lies between $1$ and $r-2$. Assume now that there are
two different summands $W_{ij}$ and $W_{k\ell}$ of $B$, which are also
different from the root of $\Ww$, and which yield the \emph{same}
index under the procedure above. We consider the upper neighbours of
$U$ of $W_{ij}$ and $V$ of $W_{k\ell}$. If there are epimorphisms
$U\ra W_{ij}$ and $V\ra W_{k\ell}$, then we conclude that the colevels
of $W_{ij}$ and $W_{k\ell}$ coincide; similarly if there are
monomorphisms $W_{ij}\ra U$ and $W_{k\ell}\ra V$, then the levels of
both coincide. In the mixed case, when there is a monomorphism
$W_{ij}\ra U$ and an epimorphism $V\ra W_{k\ell}$, then the level of
$W_{ij}$ is the colevel of $W_{k\ell}$. In all these cases it is easy
to see that there are non-zero extensions between one of these objects
and the other or the upper neighbour of the other, which gives a
contradiction. Indeed, if $W_{ij}$ and $W_{k\ell}$ have the same
colevel, they belong to the same ray and $i\neq k$, say $i<k$. Then
$\Ext^1(W_{kl},U)=\D\Hom(U,\tau W_{kl})\neq 0$.  The level case is
similar.  In the mixed case, let $c$ be the level of $W_{ij}=W_{cj}$
and the colevel of $W_{k\ell}=W_{k,k-c}$. Then $W_{ij}$ lies on the
coray ending in $W_{c,1}$ and $\tau W_{k\ell}=W_{k-1,\ell}$ lies on
the ray starting in $W_{c,1}$, so
$\Ext^1(W_{kl},W_{ij})=\D\Hom(W_{ij},\tau W_{k\ell})\neq 0$.  \medskip

It follows that the $r-1$ summands of the branch $B$ yield
copresentations for $r-1$ distinct indecomposable summands of
$\Lambda$, which are then necessarily given by $L(1),\dots,L(r-1)$.
\end{proof}
We now illustrate the procedure, which can be done for each involved
exceptional tube separately. In the following example we have two
wings in the same tube to consider. (Note that compared with
Lemmas~\ref{lem:branch-type-1} and~\ref{lem:branch-type-2} by a matter
of notation there are unavoidable shifts of indices.)
\begin{example}\label{examp:tube-11}
  In the following we will use the numerical invariants
  from~\ref{nr:numerical-invariants} and the short exact sequences
  from~\ref{nr:canonical-sequences}, which are the building blocks of
  the canonical configuration~\eqref{eqn:can-conf}. Let $\Lambda$ be a
  canonical algebra of weight type given by the sequence $(11)$, the
  only exceptional point given by $x$, let $V=\{x\}$ and $e=e(x)$,
  $f=f(x)$, $d=ef$ and $\varepsilon\in\{1,2\}$ be the numerical type
  of $\XX$.  Then $\Lambda$ is realized as canonical configuration
  $$L\ra L(1)\ra L(2)\ra L(3)\ra L(4)\ra\dots\ra L(9)\ra
  L(10)\ra\overline{L}$$
  in $\Hh$.  Let
  $$B=S[4]\oplus\tau^{-2}S[2]\oplus\tau^{-2}S\oplus S\oplus S'[3]
  \oplus S'[2]\oplus\tau^{-}S'$$
  be a branch, where we assume that $S$ is simple with
  $\Hom(L,\tau^2 S)\neq 0$ and $S'=\tau^{-6}S$. Then
  $\Hom(L(i+2),\tau^{-i}S)\neq 0$ for $i=-1,\,0,\dots,8$. There are
  two connected components of $B$, lying in the wings rooted in $S[4]$
  and $S'[3]$, respectively. The situation is illustrated in
  Figure~\ref{figure:tube-11}, where the indecomposable summands of
  the branch $B$ are denoted by $\bullet$, the roots of the two wings
  by $\widehat{\bullet}$. The two vertical lines indicate the
  identification by the $\tau$-period. We also exhibit the undercuts
  by $\bar{\circ}$, and the four Pr\"ufer sheaves belonging to
  $T_{(B,V)}$ by the symbol $\infty$ over the corresponding ray. We
  have
  $$\Lambda'=L\oplus L(1)\oplus L(6)\oplus
  L(7)\oplus\overline{L}\in\rperp{(\tau^-B)}.$$
  There are the universal exact sequences in
  $\rperp{(\tau^-B)}=\Qcoh\XX'$ (where the only weight of $\XX'$ is
  given by $p'=5$)
  \begin{gather}
    \label{eq:ses-1-5}
    0\ra L\ra\Gg\ra\tau S[\infty]^e\ra 0\\
    0\ra\overline{L}\ra\Gg^{\varepsilon}\ra\tau S[\infty]^{\varepsilon
      e}\ra 0\\ 0\ra L(i+2)\ra\Gg_i^{\varepsilon
      f}\ra\tau^{-(i+1)}S[\infty]^{\varepsilon d}\ra 0\quad\text{for}\
    i=-1,\,4,\,5.
  \end{gather}
  with torsionfree, indecomposable $\Gg$, $\Gg_i$; note that
  $\Gg,\,\Gg_i \in\rperp{(\tau^-B)}$, and thus these objects are
  $x$-divisible. Their direct sum gives the short exact sequence
  $$0\ra\Lambda'\ra\Lambda'_{V}\ra\tau
  S[\infty]^{(1+\varepsilon)e}\oplus S[\infty]^{\varepsilon
    d}\oplus\tau^{-5}S[\infty]^{\varepsilon d}
  \oplus S'[\infty]^{\varepsilon d}\ra 0$$ where
  $\Lambda'_{V}=\Gg^{1+\varepsilon}\oplus\Gg_{-1}^{\varepsilon
    f}\oplus\Gg_4^{\varepsilon f}\oplus\Gg_5^{\varepsilon f}$. This
  was Step~1.\medskip

  We now treat the first branch. This corresponds to the segment
  $L(2)$, $L(3)$, $L(4)$, $L(5)$ of $\Lambda$.  Step~2: Applying
  Lemma~\ref{lem:branch-type-2} (to the sequence
  $0\ra L(1)\ra L(5)\ra S[4]\ra 0$) gives the exact sequence
  \begin{equation}
    \label{eq:ses-6}
    0\ra L(5)\ra\Gg_{-1}^{\varepsilon f}\oplus S[4]\ra S[\infty]^{\eps d}\ra 0.
  \end{equation}
  Step~3: Applying Lemma~\ref{lem:branch-type-1} again yields
  \begin{equation}
    \label{eq:ses-7}
    0\ra L(3)\ra\Gg_{-1}^{\eps f}\oplus
  S[4]\ra\tau^{-2}S[2]\oplus S[\infty]^{\eps d}\ra 0.
  \end{equation}
  Now applying Lemma~\ref{lem:branch-type-2} two times yields
  \begin{equation}
    \label{eq:ses-8}
    0\ra
  L(4)\ra\Gg_{-1}^{\eps f}\oplus
  S[4]\oplus\tau^{-2}S\ra\tau^{-2}S[2]\oplus S[\infty]^{\eps d}\ra
  0
  \end{equation}
  and
  \begin{equation}
    \label{eq:ses-9}
    0\ra L(2)\ra\Gg_{-1}^{\eps f}\oplus
  S\ra S[\infty]^{\eps d}\ra 0.
  \end{equation}
  The second branch corresponds to the segment $L(8)$, $L(9)$, $L(10)$
  of $\Lambda$. Step~2, and then Step~3, which is applying
  Lemma~\ref{lem:branch-type-2} two times and then
  Lemma~\ref{lem:branch-type-1}, yields the exact sequences
  \begin{equation}
    \label{eq:ses-10}
    0\ra
  L(10)\ra\Gg_{5}^{\varepsilon f}\oplus S'[3]\ra S'[\infty]^{\eps
    d}\ra 0,
  \end{equation}
 then
 \begin{equation}
   \label{eq:ses-11}
   0\ra L(9)\ra\Gg_{5}^{\varepsilon f}\oplus
  S'[3]\oplus S'[2]\ra S'[\infty]^{\eps d}\ra 0
 \end{equation}
 and
  finally
  \begin{equation}
    \label{eq:ses-12}
    0\ra L(8)\ra\Gg_{5}^{\varepsilon f}\oplus S'[3]\oplus
  S'[2]\ra\tau^- S'\oplus S'[\infty]^{\eps d}\ra 0
  \end{equation}
  Forming the direct sum of all 12 short exact
  sequences~\eqref{eq:ses-1-5}--\eqref{eq:ses-12} we get the
  $\add(T)$-copresentation of $\Lambda$ as in
  Theorem~\ref{thm:axiom-TS3-for-T_VB}.
\end{example}

\begin{figure}[h]
$$
\def\c{\circ} \def\b{\bullet} \xymatrix@-1.16pc@!R=6pt@!C=6pt{ & & & &
  \infty & & \infty & & & & & & & & & & \infty & & \infty & &
  &  & \\
  \vdots &\c\ar[rd]& &\c\ar[rd]& &\c\ar[rd]& &\c\ar[rd]& &\c\ar[rd]&
  &\c\ar[rd]& &\c\ar[rd]& &\c\ar[rd]& &\c\ar[rd]&
  &\c\ar[rd]& &\c\ar[rd]& \vdots \\
  \c\ar[ru]\ar[rd] & &\c\ar[rd] & &\c\ar[rd] & &\c\ar[ru]\ar[rd] &
  &\c\ar[ru]\ar[rd]& &\c\ar[ru]\ar[rd] & &\c\ar[ru]\ar[rd]&
  &\c\ar[rd]& &\c\ar[rd]&
  &\c\ar[ru]\ar[rd]& &\c\ar[ru]\ar[rd]& &\c\\
  & \c\ar[rd]& &\widehat{\b}\ar[rrruuu]\ar[rd]& &\c\ar[ru]\ar[rd]&
  &\c\ar[ru]\ar[rd]& & \c\ar[ru]\ar[rd]& &\c\ar[ru]\ar[rd]&
  &\c\ar[rd]& &\c\ar[rd]& &\c\ar[ru]\ar[rd]&
  &\c\ar[ru]\ar[rd]& &\c\ar[ru]\ar[rd]& \\
  \c\ar[rrrruuuu]\ar[rd]& &\c\ar[ru]\ar[rd]& &\c\ar[ru]\ar[rd]&
  &\c\ar[ru]\ar[rd]& &\c\ar[ru]\ar[rd]& &\c\ar[ru]\ar[rd] &
  &\c\ar[rd]& &\widehat{\b}\ar[rrrruuuu]\ar[rd]& &\c\ar[ru]\ar[rd]&
  &\c\ar[ru]\ar[rd]& &\c\ar[ru]\ar[rd]& & \c\\
  &\c\ar[ru]\ar[rd]& &\bar{\c}\ar[ru]\ar[rd]& &\b\ar[ru]\ar[rd]&
  &\c\ar[ru]\ar[rd]& &\c\ar[ru]\ar[rd]& &\c\ar[rd]& &\b\ar[ru]\ar[rd]&
  &\c\ar[ru]\ar[rd]& &\c\ar[ru]\ar[rd]&
  &\c\ar[ru]\ar[rd]& &\c\ar[rd]& \\
  \b\ar[ru] \ar @{-}[uuuuu]& &\bar{\c}\ar[ru] & &\b\ar[ru] &
  &\c\ar[ru] & &\c\ar[ru] & &\c\ar[rrrrrruuuuuu] & & \bar{\c}\ar[ru]
  & & \b\ar[ru] & & \c\ar[ru] & & \c\ar[ru] & & \c\ar @{-}[rruu] & &
  \b \ar @{-}[uuuuu]
}
$$
\caption{The branches of Example~\ref{examp:tube-11}}\label{figure:tube-11}
\end{figure}
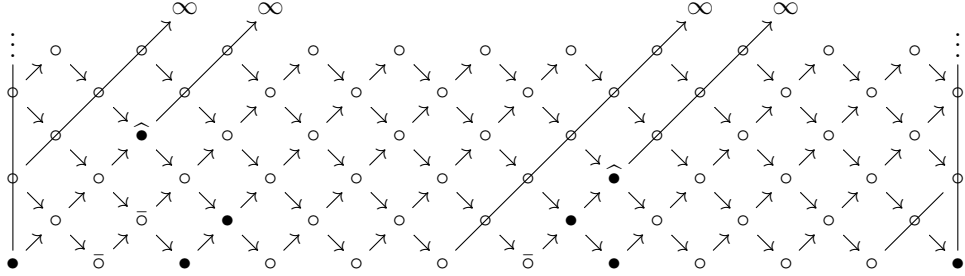

\subsection*{Acknowledgements}
This research started while the second named author was visiting the
University of Verona with a research grant of the Department of
Computer Science. The first named author is partially supported by
Fondazione Cariparo, Progetto di Eccellenza ASATA.

\bibliographystyle{elsarticle-harv}
\def\cprime{$'$}

\end{document}